\documentclass[9pt]{IEEEtran}
\usepackage{yfonts}

\addtocontents{toc}{\protect\setcounter{tocdepth}{2}}

\usepackage[caption = false,font=footnotesize]{subfig}

\usepackage[svgnames]{xcolor}

\usepackage{graphicx}

  \usepackage{rotating}
  \usepackage{floatpag}
  \rotfloatpagestyle{empty}

  \usepackage{amsmath,amssymb}  
 \usepackage{amsthm}

\usepackage{mathrsfs,multicol}
\usepackage[mathscr]{euscript}

\usepackage{booktabs}

\usepackage{tikz}
\usepackage{pgfplots}
\usepackage{tikz-3dplot}

\usepackage{tkz-euclide}

\usepgfplotslibrary{patchplots}
\pgfplotsset{compat=newest}
\usetikzlibrary{quotes,angles}
\usetikzlibrary{decorations.pathmorphing} 
\usetikzlibrary{decorations.text} 
\usetikzlibrary{matrix} 
\usetikzlibrary{arrows.meta} 
\usetikzlibrary{shapes} 
\usetikzlibrary{calc} 
\usetikzlibrary{positioning} 
\usetikzlibrary{patterns}
\usetikzlibrary{intersections}
\usepgfplotslibrary{fillbetween}
\usepgfplotslibrary{polar}
\usepgfplotslibrary{colorbrewer}
\usepgfplotslibrary{groupplots}
\usetikzlibrary{decorations.markings}

\usepgfplotslibrary{fillbetween}

\usetikzlibrary{hobby}

\usepgfplotslibrary{groupplots}

\pgfdeclarelayer{foreground}
\pgfsetlayers{main,foreground}

\tikzstyle{block} = [draw,rectangle,thick,minimum height=2em,minimum width=2em]
\tikzstyle{sum} = [draw,circle,inner sep=0mm,minimum size=4mm]
\tikzstyle{connector} = [->,thick]
\tikzstyle{line} = [thick]
\tikzstyle{branch} = [circle,inner sep=0pt,minimum size=1mm,fill=black,draw=black]
\tikzstyle{guide} = []
\tikzstyle{legendBlock} = [rectangle,minimum height=2em,minimum width=2em]

\renewcommand{\vec}[1]{\ensuremath{\boldsymbol{#1}}} 

\makeatletter    
\pgfplotsset{
    compat=1.14,
    /tikz/max node/.style={
        anchor=south,
    },
    /tikz/min node/.style={
        anchor=south,
        name=minimum
    },
    mark min/.style={
        point meta rel=per plot,
        scatter/@pre marker code/.code={%
            \ifx\pgfplotspointmeta\pgfplots@metamin
                \def\markopts{}%
                \coordinate (minimum);
                \node [min node] {
                };
            \else
                \def\markopts{mark=none}
            \fi
            \expandafter\scope\expandafter[\markopts,every node near
coord/.style=green]
        },%
        scatter/@post marker code/.code={%
            \endscope
        },
        scatter,
    },
    mark max/.style={
        point meta rel=per plot,
        visualization depends on={x \as \xvalue},
        scatter/@pre marker code/.code={%
        \ifx\pgfplotspointmeta\pgfplots@metamax
            \def\markopts{}%
            \coordinate (maximum);
            \node [max node] {
                \pgfmathprintnumber[fixed]{\xvalue},%
                \pgfmathprintnumber[fixed]{\pgfplotspointmeta}
            };
        \else
            \def\markopts{mark=none}
        \fi
            \expandafter\scope\expandafter[\markopts]
        },%
        scatter/@post marker code/.code={%
            \endscope
        },
        scatter
    }
}
\makeatother    

\definecolor{yellow1}{RGB}{255,255,204}
\definecolor{blue2}{RGB}{161,218,180}
\definecolor{blue3}{RGB}{65,182,196}
\definecolor{blue4}{RGB}{44,127,184}
\definecolor{blue5}{RGB}{37,52,148}
\definecolor{blueForRed1}{RGB}{5,113,176}
\definecolor{brown1}{RGB}{166,97,26}
\definecolor{brown2}{RGB}{223,194,125}

\definecolor{red1}{RGB}{215,25,28}
\definecolor{red2}{RGB}{253,174,97}

\definecolor{grayMax}{cmyk}{0,0,0,85}
\definecolor{grayOne}{cmyk}{0,0,0,1}
\definecolor{grayTwo}{cmyk}{0,0,0,20}
\definecolor{grayTwoAndHalf}{cmyk}{0,0,0,28}
\definecolor{grayThree}{cmyk}{0,0,0,41}
\definecolor{grayFour}{cmyk}{0,0,0,61}
\definecolor{magentaForBlack}{RGB}{202,0,32}
\definecolor{peachForBlack}{RGB}{244,165,130}

\definecolor{darkViolet}{cmyk}{65,70,0,0}
\definecolor{darkSea}{cmyk}{85,30,0,0}
\definecolor{paleSea}{cmyk}{33,3,0,0}
\definecolor{paleViolet}{cmyk}{65,70,0,0}
\definecolor{paleGreen}{cmyk}{24,0,39,0}
\definecolor{paleOrange}{cmyk}{5,35,70,0}

\definecolor{c1}{RGB}{179,88,6}
\definecolor{c2}{RGB}{241,163,64}
\definecolor{c3}{RGB}{216,218,235}
\definecolor{c4}{RGB}{153,142,195}
\definecolor{c5}{RGB}{84,39,136}

\definecolor{tacGreen1}{cmyk}{50,0,17,0}
\definecolor{tacGreen2}{cmyk}{100,10,55,0}
\definecolor{tacBrown1}{cmyk}{12,20,45,0}
\definecolor{tacBrown2}{cmyk}{35,55,90,0}

\definecolor{darkTeal}{cmyk}{100,10,55,0}
\definecolor{lightTeal}{cmyk}{50,0,17,0}
\definecolor{toffee}{cmyk}{12,20,45,0}
\definecolor{darkToffee}{RGB}{166,97,26}

\definecolor{purDark}{RGB}{123,50,148}
\definecolor{purLight}{RGB}{194,165,207}
\definecolor{paleGreen}{RGB}{166,219,160}
\definecolor{darkGreen}{RGB}{0,136,55}

\definecolor{myCrimson}{RGB}{202,0,32}

   \newtheorem{theorem}{Theorem}[section]
    
    \newtheorem{lemma}[theorem]{Lemma}
     
    \newtheorem{definition}[theorem]{Definition}

 \pgfplotsset{compat=newest}

\usepackage{mathtools}
\usepackage{amsfonts}
\usepackage{ucs}
\usepackage[utf8x]{inputenc}
\usepackage[T1]{fontenc}
\usepackage{ae,aecompl}

\usepackage{dsfont}

\begin{document}
\title{Relay self-oscillations for second order, stable, nonminimum phase plants}
\author{Maben~Rabi
\thanks{The author is with Halmstad University, Sweden.
}%
}

\maketitle

\begin{abstract}%
    We study a relay feedback system (RFS) having an ideal relay element and a linear, time-invariant, second order plant.
We model the relay element using an ideal on-off switch. And we model the second order plant with
a transfer function that: (i)~is Hurwitz stable,
(ii)~is proper, (iii)~has a positive real zero, and (iv)~has a positive DC gain. 

We analyze this RFS using a state space description, with closed form expressions for the
state trajectory from one switching time to the next. We prove that the state transformation
from one switching time to the next: (a)~has a Schur stable linearization, (b)~is a contraction
mapping, and (c)~maps points of large magnitudes to points with lesser magnitudes. Then using
the Banach contraction mapping theorem, we prove that every trajectory of this RFS
converges asymptotically to an unique limit cycle. This limit cycle is symmetric,
and is unimodal as it has exactly two relay switches per period. This result helps
understand the behaviour of the relay autotuning method, when applied to second order
plants with no time delay.

We also treat cases where the plant either has no finite zero, or has exactly one zero and that is negative.
\end{abstract}

\section{Introduction}
\newcommand{\relay}{
    \begin{tikzpicture}[yscale=0.25,xscale=0.34]
        \draw[ultra thick] (-2,-1)--(0,-1)--(0,1)--(2,1);
        \draw[thin, <->] (-2,0)--(2,0);
        \draw[thin, <->] (0,-2.5)--(0,2.5);
        \draw[thin, dashed] (1,1)--(1,0);
    \end{tikzpicture}
       }
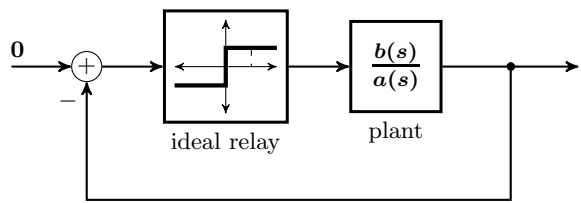
\begin{figure}[h]
\begin{center}
\begin{tikzpicture}[scale=1, auto, >=stealth']
 \matrix[ampersand replacement=\&, row sep=0.4cm, column sep=0.8cm]
{
    \node(startOne) {};    \&
      \node[sum] (summer) {$\mathbf{+}$} ;  \&
    \node[block,minimum size=1.1cm,line width=0.4mm,label=below:{ideal relay}] (relay)  {\relay }; \&
    \node[block,minimum size=1.2cm,line width=0.4mm,label=below:plant] (plant)  {  {\mbox{\Large$\vec{{\frac{b(s)}{a(s)}}}$}}  }; \&
      \node[branch] (bendOne) {} ; \& 
      \node (output) {};  
      \\
      \&
      \node (bendThree) {} ; \&
      \& 
     \&
      \node (bendTwo) {} ;  \&
  \\
    };
    \draw [connector] (startOne) -- node[pos=0.1] {$\vec{0}$} (summer);
    \draw [connector] (summer) -- (relay);
    \draw [connector] (relay)--(plant);     \draw [connector] (plant)--(output);
    \draw [connector] (bendOne)  -| (bendTwo.center) -- (bendThree) -|  node[pos=0.92]
  {$\mathbf{-}$} (summer);
\end{tikzpicture}
\end{center}
\caption{Relay feedback system}
\label{fig:blockDiagram}
       \end{figure}
\IEEEPARstart{F}{or} what plant transfer functions does the relay feedback system~(RFS) of Figure~\ref{fig:blockDiagram} 
converge to a globally asymptotically stable limit cycle~?
In this paper we give a sufficient condition, applicable to second order plants.

A {\textit{symmetric limit cycle}} is one whose locus in state space is symmetric w.r.t. the origin. 
If a limit cycle of the RFS is such that the relay element's output changes sign exactly two times per period, then we call it an {\textit{unimodal limit cycle.}}
We say the RFS  {\textit{self-oscillates}} if it has a limit cycle to which all trajectories converge asymptotically.

Suppose that the plant Figure~\ref{fig:blockDiagram} is unknown, and an experiment shows that the loop goes into a symmetrical, 
limit cycle, then the parameters of this limit cycle can be used to estimate essential details of the plant's Nyquist locus. Using this, the {\textit{Relay autotuning method}}~\cite{astromHagglund1984relayAutotuning,astromHagglund2011relayAutotuning} tunes a PID controller for the plant. 
This motivates the study of relay oscillations.

\subsection{Previous results}
The classic study of Andronov, Vitt and Khaikin~\cite{andronovVittKhaikin1966theoryOfOscillations,minorsky1962nonlinearOscillations} contains an approach for analyzing individual second order systems, based on properties of ordinary differential equations on the plane. 
For this approach, the authors give the name {\textit{the method of point transformations,}} because the method focuses on the transformations of the system state from one switching instant to the next. Using this method they proved the stability of some second order valve oscillators, which were significant at that time for Radio engineering.

   Complex periodic behaviours can arise in RFS. Di~Bernardo~et.al~\cite{marioDiBernardoKarlJohansson2001ijbc} report a third order RFS where as the plant's DC gain is varied, a symmetric unimodal limit cycle bifurcates into two unsymmetrical unimodal limit cycles. Other third order RFS can possess limit cycles with four switches per period~\cite{varigondaGeorgiou2001relay}.  And yet other third order RFS can possess limit cycles that have infinitely many switches, or possess limit cycles whose trajectory segments are sliding motions~\cite{johanssonRantzerAstrom1999fastSwitchesAutomatica,johanssonBarabanovAstrom2002slidingTAC}.%

\subsubsection{Oscillation conditions for plants of arbitrary order}
Frequency domain tools such as the Hamel locus~\cite{gillePelegrinDecaulne1959feedbackControlSystems} and the Tsypkin locus~\cite{tsypkin1984relayControlSystems} give necessary conditions for limit cycles. But these tools give no information about stability. 

A state-space necessary condition for a symmetrical, periodic orbit is given in {\AA}str{\"o}m~\cite{astrom1995oscillationsRelay}. This paper also gives a linearized analysis for local stability in the neighbourhood of a limit cyle.
Goncalves~et~al.~\cite{goncalvesMegretskiDahleh2001relay} show how a Lyapunov analysis for global stability can be carried out.

Bliman and Krasnosel'skii~\cite{blimanKrasnoselskii1997periodicSolutions}, as well as Varigonda and Georgiou~\cite{varigondaGeorgiou2001relay} mention the possibility of using fixed point theorems to prove convergence
 to a limit cycle. Johansson and Rantzer~\cite{johanssonRantzer1996globalAnalysisOfThirdOrderRelay} treat the case of a third order plant with no zeros, and with three distinct, real, stable poles. They show that trajectories obey an `area contraction' property.

Megretski~\cite{megretski1996globalStabilityOfRelayOscillations} gives a graphical template for the step response of a plant, for it to self-oscillate under relay feedback. The template closely resembles that of a stable, nonminimum phase, second order plant. The author proves that if a plant transfer function of arbitrary order has a step response that resembles this template, then self-oscillations are guaranteed if a discrete time iteration of inter-switching times has an unique stationary solution. And such a unique stationary solution exists if a related discrete time linearized iteration satisfies an~$l^1$-norm condition.

\subsubsection{Oscillation conditions for second order RFS}
The doctoral thesis of Holmberg~\cite{holmberg1991thesisRelay} contains detailed studies of relay oscillations in first and second order RFS. 
One significant result is its Theorem~5.1 for  RFS with stable, second order plants. This theorem states that if the RFS possesses a unique limit cycle, then all trajectories converge globally asymptotically to it.

    Given a specific plant, we can apply~{\AA}str{\"o}m's necessary condition~\cite{astrom1995oscillationsRelay} to count the number of limit cycles possessed by the RFS. If we can verify that only one limit exists, then by Holmberg's theorem we can be sure of global asymptotic convergence, as is illustrated by Example~2 in~\cite{johanssonRantzerAstrom1999fastSwitchesAutomatica}.%

    Holmberg's thesis also gives some classes of second order plants and RFS  with guaranteed unique limit cycles. But the plants in these classes have no finite zeroes, and the relay element is required to have a non-zero hysteresis width. If we take a stable plant from Holmberg's classes, and let the hysteresis width be zero, then the oscillations vanish.%

In specific, if the second order plant (a)~is stable, (b)~nominimum phase, and  (c)~has positive DC gain, then the RFS  has (i)~bounded trajectories,
(ii)~no equilibrium point, and (iii)~no chattering point. Because of these constraints and because trajectories lie on the plane and cannot cross each other, any trajectory faces three mutually exclusive choices: (i)~be a limit cycle, (ii) curve inwards, or (iii)~curve outwards. Then is it not inevitable that this RFS must have an unique limit cycle, so that Holmberg's theorem kicks in~?

 Not as such. In principle, it is possible that the RFS has a nested sequence of limit cycles, with attracting limit cycles alternating with repelling limit cycles. One has to explicitly rule out such possibilities.

\subsection{Outline of the rest of the paper}
We work with a state space realization of the RFS. One can express the trajectory between switching points as an explicit expression. Because the plant is second order, this can be done `by hand.' But we apply a short-cut, made possible because the literature gives an explicit similarity transformation for putting a companion matrix into its Jordan canonical form.

With the closed form trajectory, and basic Calculus we study the function that maps the state at one switching instant to the state at the next switching instant.
We prove that this switching point transformation function has a stable linearization.
    By the way, this map is frequently called the~{\textit{Poincar{\'{e}} map~\cite{holmberg1991thesisRelay,astrom1995oscillationsRelay,johanssonRantzerAstrom1999fastSwitchesAutomatica,varigondaGeorgiou2001relay,marioDiBernardoKarlJohansson2001ijbc},}} and in some cases called the {\textit{Surface impact map}} from the switching surface to itself~\cite{goncalvesMegretskiDahleh2003relayStability}.

We prove that self-oscillations are guaranteed for every RFS with a second order plant that is stable, nonminimum phase, and has a positive DC gain. We show that the convergent limit cycle is symmetric and unimodal.
 
We prove that if the second order plant is stable, has positive DC gain, and has no finite zero (numerator of transfer function is a constant), then the RFS converges asymptotically to the origin.

We prove that if the second order plant is stable, has positive DC gain, and has exactly one zero which is negative, then the RFS converges asymptotically to the chattering set.

When compared to~\cite{holmberg1991thesisRelay,megretski1996globalStabilityOfRelayOscillations}, we focus attention of a relatively narrow class of relay feedback systems - we do not allow time delays or hysteresis.

 \section{Preliminary definitions\label{section:notationAndDefinitions}}


Let 
 $(A,B,C)$ be a minimal realization of the plant. Then the relay feedback system evolves as per:
\begin{align}
    {\frac{d}{dt}} x & =
    A x -B  \,{\text{sign}} \left(Cx\right),  x\in {\mathbb{R}}^2,
    \label{eqn:RFSdynamics}
\end{align}
where ${\text{sign}} (\cdot )$ is the ideal signum function (on-off switch).
\begin{definition}
    The switching plane is the hyperplane 
\begin{align*}
{} {\mathscr{S}} & \triangleq \left\{ \xi \in  {\mathbb{R}}^2 : C \xi = 0 \right\} ,
\end{align*}
and it is precisely that set where the relay switches sign.
\end{definition}
\begin{figure}
\begin{center}
\begin{tikzpicture}[
    declare function={tauone(\z)= ln(2*\z+5) - ln(3) ;},
    declare function={xnewone(\z)= -11/2 + 6*(\z+4)/(2*\z+5) - 9/2/(2*\z+5) ;},
    declare function={tautwo(\z)= ln( 2*\z + 1 ) - ln(3) ;},
    declare function={xnewtwo(\z)= -7/2 + 6*(\z+2)/(2*\z+1) - 9/2/(2*\z+1) ;}
    ]
     \begin{groupplot}[
            group style={
            group size=1 by 3,
            vertical sep=20pt,
            group name=G},
            width=9.5cm,height=4.75cm,
            axis lines=middle,
            clip=false,
            no markers,
            domain=-5:5,
            grid = both,
            samples=100]
                                                            
      \nextgroupplot[
            xlabel={$\xi$},
            ylabel={${{\tau_+} \left( \xi , 0  \right)}$},
           y label style={right,font=\normalsize,at={(rel axis cs:0.5,0.99)}},
            yticklabels={,,},
            xticklabels={,,},
            legend pos=north west,
        ]
            \addplot[domain=-5:1,darkToffee,line width = 2,forget plot] {0};
            \addplot[domain=1:5,darkToffee,line width = 2,] {tautwo(x)};
            \addplot[domain=-5:-1,dashed,lightTeal,line width = 2,forget plot] {0};
            \addplot[domain=-1:5,dashed,lightTeal,line width = 2] {tauone(x)};
              \legend{${\frac{s+3}{s^2+3s+2}}$,${\frac{-s+3}{s^2+3s+2}}$};

      \nextgroupplot[
        xlabel={$\xi$},
        ylabel={${\begin{pmatrix} 1 & 0 \end{pmatrix} \times {\psi_+} \left(  \xi , 0  \right)}$},
           y label style={right,font=\normalsize,at={(rel axis cs:0.10,0.98)}},
            x label style={right,font=\large,at={(rel axis cs:0.94,0.92)}},
            yticklabels={,,},
            xticklabels={,,},
        ]
            \addplot[domain=-5:1,darkToffee,line width = 2] {x};
            \addplot[domain=1:5,darkToffee,line width = 2] {xnewtwo(x)};
            \addplot[domain=-5:-1,dashed,lightTeal,line width = 2] {x};
            \addplot [domain=-1:5,dashed,lightTeal,line width = 2]{xnewone(x)};
 
     \nextgroupplot[
        xlabel={$\xi$},
        ylabel={${\begin{pmatrix} 1 & 0 \end{pmatrix} \times {\psi_-} \bigl( {\psi_+} \left(  \xi , 0  \right) \bigr) }$},
           y label style={right,font=\normalsize,at={(rel axis cs:0.27,1.1)}},
            yticklabels={,,},
            xticklabels={,,},
        ]
        \addplot[domain=-4:4,dotted,black!80,thin] {x};
        \addplot[domain=-5:-1,darkToffee,line width = 2] {-1*xnewtwo(-1*x)};
        \addplot [domain=1:5,darkToffee,line width = 2]{xnewtwo(x)};
        \addplot [domain=-1:1,darkToffee,line width = 2]{x};
        \addplot[domain=-5:-1,dashed,lightTeal,line width = 2] {-1*xnewone(-1*x)};
        \addplot [domain=-1:5,dashed,lightTeal,line width = 2]{-1*xnewone(-1*xnewone(x))};
        \node at (2,2) {\large\textcolor{red}{$\vec{\bullet}$}};
        \node at (2,2) {\huge\textcolor{red}{$\vec{\circlearrowleft}$}};

\end{groupplot}
\end{tikzpicture}
\end{center}
\caption{\label{fig:plotsForExampleOne}
    The first chart shows the first exit time~$\tau_+\left(\cdot\right)$ of points on the switching plane, which have the form:~$\left( \xi, 0 \right)$.%
\newline
    The second chart shows the $x_1$-coordinate of the images of points on the switching plane, under the first exit map~$\psi_+\left(\cdot\right).$%
\newline
    The third chart shows the $x_1$-coordinate of the images of points on the switching plane, under the first return map~$\psi_-\left( \psi_+\left( \cdot\right) \right).$%
\newline
When the plant transfer function is~${\left({-s+3}\right)/\left({s^2+3s+2}\right)}$, the first return map has an unique fixed point, shown by:~{\Large\textcolor{red}{$\vec{\circlearrowleft}$}}.%
}
\end{figure}
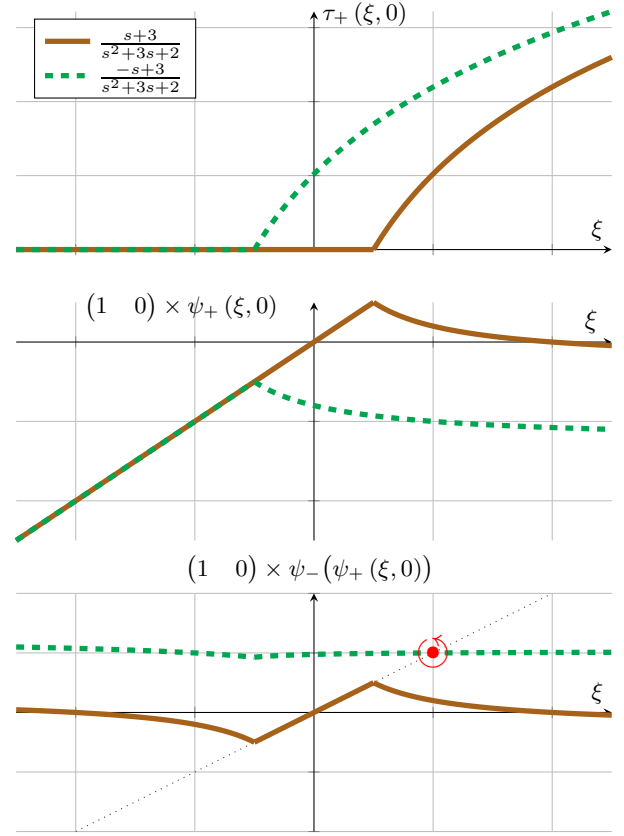
\begin{definition}
    The first exit time from positive sign~(FET+ for short) is the non-negative function:
\begin{gather*}
    {\tau_{+}} \ : \ {\mathbb{R}}^2   \to
    {\mathbb{R}} \ \text{such that} \ \
{\tau_{+}} \left( \xi \right) \ = \  
    \inf \left\{ t > 0: C x\left(t\right) < 0 \right\} , \\
    \text{where} \
    {\dot{x}} \ = \ A x - B, \ \text{and} \ x\left( 0 \right) = \xi.
\end{gather*}
\end{definition}
\begin{definition}
    The first exit time from negative sign~(FET- for short) is the non-negative function:
\begin{gather*}
    {\tau_{-}} \ : \   {\mathbb{R}}^2  \to
    {\mathbb{R}} \ \text{such that} \ \
{\tau_{-}} \left( \xi \right) \ = \  
    \inf \left\{ t > 0: C x\left(t\right) > 0 \right\} , \\ \text{where} \
    {\dot{x}} \ = \ A x + B, \ \text{and} \ x\left( 0 \right) = \xi.
\end{gather*}
\end{definition}
\begin{definition}
    The first exit map from positive sign~(FEM+ for short) is defined at every point for which, the trajectory starting there  crosses the switching plane in finite time:
\begin{gather*}
    {\psi_{+}} \ : \ {\mathbb{R}}^2   \to
    {\mathbb{R}}^2 \ \text{such that} \ \
    {\psi_{+}} \left( \xi \right) \ = \  x\left(\tau_{+}\right), \ \text{where} \\ 
    \text{where} \
    {\dot{x}} \ = \ A x - B, \ \text{and} \ x\left( 0 \right) = \xi.
\end{gather*}
\end{definition}
\begin{definition}
    The first exit map from negative sign~(FEM- for short) is defined at every point for which, the trajectory starting there  crosses the switching plane in finite time:
\begin{gather*}
    {\psi_{-}} \ : \ {\mathbb{R}}^2   \to
    {\mathbb{R}}^2 \ \text{such that} \ \
    {\psi_{-}} \left( \xi \right) \ = \  x\left(\tau_{-}\right), \ \text{where} \\ 
    \text{where} \
    {\dot{x}} \ = \ A x + B, \ \text{and} \ x\left( 0 \right) = \xi.
\end{gather*}
\end{definition}
Figure~\ref{fig:plotsForExampleOne} illustrates the above definitions by plotting the first exit time etc. for two examples of second order plants.
\subsubsection*{Symmetry w.r.t. the origin of the state space}
\noindent
The following identities hold: for every~$x \in {\mathbb{R}}^2,$
\begin{gather*}
{\tau_{-}} \left( x \right) \; = \; {\tau_{+}} \left(  - x \right), \quad \text{and,} \quad
{\psi_{-}} \left( x \right) \; = \;
 - {\psi_{+}} \left( - x \right).
\end{gather*}
\section{Theorems guaranteeing self-oscillations\label{section:secondOrder}}
We consider plant transfer functions of the form
\begin{gather}
    {\frac{ - \kappa s + \gamma}{\left(s - p_1 \right) \left( s - p_2 \right)}},
		\label{eqn:secondOrderTransferFunction}
\end{gather}
where the real parameters~$\kappa, \gamma$ are both positive, and the poles~$p_1, p_2$ are both stable. For this transfer function, the observer realization takes the form:
\begin{gather}
{\dot{x}} \ = \ A x + Bu, \quad
y \ = \ Cx, \ \text{where,} \\
    A = \begin{bmatrix}
        0 & - p_1 p_2 \\
        1 & \left( p_1 + p_2 \right) 
        \end{bmatrix},
    \quad
    B = \begin{bmatrix}
        \gamma \\ -\kappa 
        \end{bmatrix}, \quad
    C = \begin{bmatrix}
        \ 0 &  1 \
        \end{bmatrix}. \label{eqn:observerRealization}
\end{gather}
Because the poles are all nonzero, the matrix~$A$ is invertible, and so we can rewrite the evolution equation~\eqref{eqn:RFSdynamics} as:
\begin{align}
{\frac{d}{dt}}\left( x - A^{-1}B  \right) & =
    A \left( x - A^{-1}B   \right)  , \ x_2 \ge 0,
    \label{eqn:secondOrderRplus} \\
{\frac{d}{dt}}\left( x + A^{-1}B \right) & =
    A \left( x + A^{-1}B  \right)  ,  \ x_2 \le 0.
    \label{eqn:secondOrderRminus}		
\end{align}
In the appendix~\ref{sec:fillipovAppendix}, we shall pay closer attention to how the vector field could be defined at the switching plane, and what notions of solution to apply for this ODE.
\subsection{The switching point transformation function}
In Section~\ref{section:notationAndDefinitions} we defined the first exit map from positive sign. Here we shall define a function with a similar meaning, but maps real numbers to real numbers.
Notice the following facts for the second order realization~\eqref{eqn:observerRealization}:
\begin{itemize}
\item{the switching `plane' is in fact a line, and}
\item{this line has the easy description: $\left\{ ( \xi , 0 ) : \xi \in {\mathbb{R}} \right\}$}.
\end{itemize}
\begin{definition}
The switching point transformation function is:
\begin{align*}
   f_{+}{\left( \xi \right)} & \triangleq 
	   \begin{pmatrix}  1 &  0 \end{pmatrix}  
  \psi_+  \left( \xi , 0 \right) , \
	\forall \xi \in  {\mathbb{R}},
\end{align*}
where~$ \psi_+  \left( \cdot \right) $ is the first exit map under plus sign for the second order RFS~\eqref{eqn:observerRealization},~ \eqref{eqn:secondOrderRplus},~\eqref{eqn:secondOrderRminus}.
\end{definition}
It follows that the function $ f_{+}{\left( \cdot \right)}$ is defined at the real number~$\xi$ if the first exit time~$\tau_+\left( \xi , 0 \right)$ is finite, and it also follows that $ f_{+}{\left( \cdot \right)}$ is undefined otherwise.
\begin{lemma}\label{lemma:tauAndFpusDefinedEverywhereOnSwitchingLine}
Consider the RFS with the second order plant transfer function~$ {{\left( - \kappa s + \gamma \right)}/{\left(s - p_1 \right) \left( s - p_2 \right)}}$. If this transfer function: (i)~is Hurwitz stable, and (ii)~has a positive DC gain, then
the first exit time and the switching point transformation function are defined everywhere on the switching line.
\end{lemma}
\begin{proof}
We prove by showing that trajectories cannot forever postpone crossing the switching line, which in this case is the $x_1$-axis.
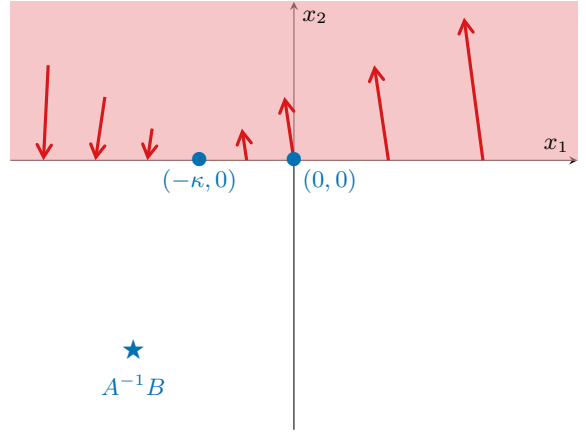
\begin{figure}
\begin{center}
\begin{tikzpicture} 
\begin{axis}[axis lines = middle, ticks = none, xmin = -3, xmax = 3, ymin = -1.7, ymax = 1, xlabel=$x_1$,  ylabel=$x_2$,
            x post scale = 1.1, 
						]
\fill [red1!40, fill opacity=0.6] (-3,0) rectangle (3,1.7);
\draw[color=red1,->,>=angle 60, very thick]   (axis cs:-0.5, 0) to (axis cs: -0.55,0.2);
\draw[color=red1,->,>=angle 60, very thick]   (axis cs:0, 0) to (axis cs: -0.1,0.4);
\draw[color=red1,->,>=angle 60, very thick]   (axis cs:1, 0) to (axis cs: 0.85,0.6);
\draw[color=red1,->,>=angle 60, very thick]   (axis cs:2, 0) to (axis cs: 1.8,0.9);
\draw[color=red1,->,>=angle 60, very thick]   (axis cs:-1.5, 0.2) to (axis cs: -1.55,0);
\draw[color=red1,->,>=angle 60, very thick]   (axis cs:-2.0, 0.4) to (axis cs: -2.1,0);
\draw[color=red1,->,>=angle 60, very thick]   (axis cs:-2.6, 0.6) to (axis cs: -2.65,0);
\node[color=blueForRed1,scale=1.5]  at (axis cs:-1, 0) {$\bullet$};
\node[color=blueForRed1,scale=1,below]  at (axis cs:-1, 0) {$(-\kappa,0)$};
\node[color=blueForRed1,scale=1.5]  at (axis cs:0, 0) {$\bullet$};
\node[color=blueForRed1,scale=1,below right]  at (axis cs:0, 0) {$(0,0)$};
\node[color=blueForRed1,scale=1]  at (axis cs:-1.7, -1.2) {$\bigstar$};
\node[color=blueForRed1,scale=1,below]  at (axis cs:-1.7, -1.3) {$A^{-1}B$};
\end{axis}
\end{tikzpicture}
\end{center}
\caption{\label{fig:switchingLineUnderPositiveRelayOutput}Vector field at the switching line under~${\dot{x}}  =  A x - B$.}
\end{figure}
Consider the point~$\left( x_1 , x_2 \right)$ where~$x_2 \ge 0.$ Under the flow of the ODE~\eqref{eqn:secondOrderRplus}, the trajectories shall asymptotically converge to the sink~$A^{-1}B.$ Since
\begin{align}
A^{-1}B & =
  \begin{pmatrix}
	   {\frac{ p_1 + p_2  }{p_1 p_2}}  \gamma - \kappa \\
		-  {\frac{1}{p_1 p_2}} \gamma
	\end{pmatrix} ,
        \label{eqn:sinkCoordinates}
\end{align}
and since $\gamma/p_1 p_2 > 0 ,$ it follows that the sink~$A^{-1}B$ lies below the $x_1$-axis, as shown in 
Figure~\ref{fig:switchingLineUnderPositiveRelayOutput}. Hence it is inevitable that in some finite time, the
trajectory that started at the point~$\left( x_1 , x_2 \right)$ shall cross the $x_1$-axis as it moves towards the sink~$A^{-1}B .$
Thus the first exit time~$\tau_+\left( \cdot \right)$ is defined everywhere on the switching line. Since we have:
\begin{align*}
  f_+\left( \xi \right)  & = 
	 \begin{pmatrix}  1 &  0 \end{pmatrix}  
  \psi_+  \left( \xi , 0 \right) , \\
	& =  \begin{pmatrix}  1 &  0 \end{pmatrix} 
	    \left(
			   A^{-1}B + e^{A\tau_+\left( \xi , 0 \right)} \left( \begin{pmatrix}  \xi \\  0 \end{pmatrix}   -  A^{-1} B \right)
			\right) ,
\end{align*}
it follows that~$ f_+\left( \cdot \right)$ is defined for all real values. 
\end{proof}
\subsection{Banach's fixed point theorem\label{section:sub:contractionMapping}}
\noindent
The following two items are tailored to functions on~${\mathbb{R}}.$
\begin{definition}
    Let $\Omega \subseteq {\mathbb{R}}.$ Then a function~$f : \Omega \to \Omega$ is a contraction mapping if there is a~$\rho$ such that $ 0 \le \rho < 1 , $ and
\begin{align*}
    \left\lvert f (  \xi )  - f (  {\xi}^{\prime} )\right\rvert   & \le  \rho \,  {\left| \xi - {\xi}^{\prime} \right|}  , \ \text{for any two} \ \xi , \, {\xi}^{\prime} \in \Omega .
\end{align*}
\end{definition}
\vskip 5pt
\begin{theorem}[Banach's fixed point theorem~\cite{palais2007banachFixedPointTheorem}]
\label{theorem:banachFixedPoint}
Let the set~$\Omega \subset {\mathbb{R}},$ and let this set be non-empty, closed and bounded. If the function~$f$ is a contraction mapping on~$\Omega,$ 
    then it follows that: (i)~the function~$f$ has a unique fixed point in~$\Omega$ and, 
    (ii)~for every~$\xi \in \Omega , $ the infinite sequence of iterates:
\begin{gather*}
  \left\{ \: f\left( \xi \right) , \quad  f\left( f\left( \xi \right) \right) ,  \quad   f\left( f\left( f\left( \xi \right) \right) \right) , \quad  \ldots \; \right\}
\end{gather*}
converges to the said fixed point.
\end{theorem}
\subsection{Properties of the switching point transformation function that guarantee self-oscillations\label{section:sub:conditionsOnLevelCrossingLocationFunction}}
    First we show that~$f_+$ is monotonic over the interval~$ \left( - \kappa , + \infty \right).$ 
    Consider two real numbers~$\xi , {\xi}^{\prime} \in \left( - \kappa , + \infty \right)$ such that~$ \xi $ is less than $ {\xi}^{\prime} . $ Under the flow of the ODE~${\dot{x}} = Ax - B , $ the two trajectories starting from the points:~$ {\left(  {\xi} \; 0  \right)}^T , \;  {\left(  {\xi}^{\prime} \; 0  \right)}^T  $ must satisfy the following:
\begin{enumerate}
\item{the trajectories do not cross each other,}
\item{until their next switching times, the two trajectories lie in the half space~$\bigl\{ {(x_1\;x_2 )}^T : x_2 \ge 0 \bigr\}$.}
\item{these two trajectories must start at points on the switching line that are to the right of the point~${\left( - \kappa \; 0\right)}^T$, and must end at points on the switching plane that are to the left of the point~${\left( - \kappa \; 0\right)}^T$.}
\end{enumerate}
Therefore we can conclude that
\begin{align}
f_+\left(  \xi \right) & \; > \; f_+\left(  {\xi}^{\prime}  \right) , \ \text{if} \
  \xi , {\xi}^{\prime} \in \left( - \kappa , + \infty \right) \ \text{and} \ \xi < {\xi}^{\prime} .
\label{eqn:fPlusIsMonotonicallyDecreasing}
\end{align}
In other words, the function~$f_+\left(  \cdot \right) $ is monotonically decreasing on the interval~$  \left( - \kappa , + \infty \right).$
\begin{theorem}
    Consider the RFS with the second order transfer function~$ {{\left( - \kappa s + \gamma \right)}/{\bigl( \left(s - p_1 \right) \left( s - p_2 \right) \bigr)}}$. Suppose that the following four assumptions hold:
\begin{description}
\item[A~1]{the parameters~$\kappa, \gamma$ are positive,}
\item[A~2]{the poles~$p_1 , p_2$ are stable,} 
\end{description}
and the ODE:~${\dot{x}} = Ax - B$ derived from the transfer function via~\eqref{eqn:observerRealization} is such that the resulting switching point transformation function~$f_+\left( \cdot \right)$ obeys the next two conditions:
\begin{description}
\item[A~3]{there exists a positive real number~$\eta$ such that
    the function~$ {\mathbf{-}} f_+\left( \cdot \right)$ is a contraction mapping on the interval:~$\left[ 0 , \theta \right] , $ whenever~$   \theta \ge \eta . $}
\end{description}
Then all trajectories of the RFS converge globally asymptotically to a symmetric, unimodal limit cycle.
\label{theorem:selfOscillationsForSecondOrderPlants}
\end{theorem}
\begin{proof}
    We investigate a generic trajectory of the RFS by investigating the sequence of its switching points. We shall show that 
    this sequence is infinite, and converging to a discrete time periodic orbit. We then show that this corresponds to the RFS converging to a limit cycle.

Denote the first switching point by~$ {\left( \xi_0  \; 0 \right)}^T .$
    The sequence of switching points can then be described as:
\begin{align*}
 \Upsilon \left( \xi_0 \right) & = 
    \left\{  \begin{pmatrix} \xi_0 \\ 0 \end{pmatrix}  ,  \; \begin{pmatrix}  f_+\left( \xi_0 \right) \\ 0 \end{pmatrix}  ,  \; \begin{pmatrix}  - f_+\left( - f_+\left( \xi_0 \right) \right) \\ 0 \end{pmatrix}  ,  \right. \\
	& \quad \quad \quad \quad \quad  \left.
        \begin{pmatrix}   f_+\left( - f_+\left( - f_+\left(
           \xi_0  
        \right) \right) \right) \\ 0 \end{pmatrix}  , 
\ldots \right\}
\end{align*}
By Lemma~\ref{lemma:tauAndFpusDefinedEverywhereOnSwitchingLine} every element of this sequence is well defined. The following sequence
of real numbers:
\begin{align*}
 \Xi \left( \xi_0 \right) & = 
  \bigl\{   \xi_0  , \ - f_+\left(  \xi_0 \right) , \   -  f_+\left( - f_+\left( \xi_0 \right) \right) , \;  \bigr. \\
		& \quad \quad \quad \qquad \qquad \bigl. -f_+ \left( - f_+\left( - f_+\left( \xi_0 \right) \right) \right) , \ldots \bigr\} ,
\end{align*}%
can be constructed from the sequence~$\Upsilon\left( \xi_0 \right)$, which in turn can be exactly reconstructed from~$\Xi \left( \xi_0 \right) .$ Since the function~$ - f_+ \left(  \cdot \right)$ is defined everywhere on~${\mathbb{R}}$, it follows that the sequence~$\Xi \left( \xi_0 \right) $ is infinite.

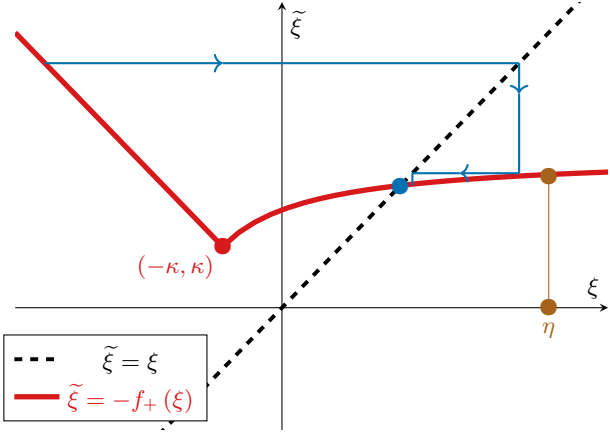
\begin{figure}
\begin{center}
\begin{tikzpicture}[declare function={phiPlusTwo(\z)= + 11/2 - 6*(\z+4)/(2*\z+5) + 9/2/(2*\z+5) ;}]
      \begin{axis}[ axis lines = middle, ticks = none, xmin = -4.5, xmax = 5.5, ymin = -2, ymax = 5,
        xlabel={$\xi$}, ylabel={${\widetilde{\xi}}$},
				legend style ={at={(-0.02,0.02)},anchor = south west},  x post scale = 1.15,  ]
            \addplot[domain=-4.5:5.5, dashed, line width = 1.5] {x};
						\addlegendentry{${\widetilde{\xi}} = \xi$};
            \addplot[domain=-4.5:-1,red1, line width = 2] {-1*x};
						\addlegendentry[red1]{${\widetilde{\xi}} =  - f_{+}\left(\xi\right)$};
            \addplot[domain=-1:5.5,red1, line width = 2]{phiPlusTwo(x)};
						\addplot[->,blueForRed1, thick]coordinates{ (-4,4) (-1,4) };
								\addplot[blueForRed1, thick]coordinates{ (-1,4) (4,4) };
						\addplot[->,blueForRed1, thick]coordinates{ (4,4) (4,3.5) };
						    \addplot[blueForRed1, thick]coordinates{ (4,3.5) (4,2.2) };
						\addplot[->,blueForRed1, thick]coordinates{ (4,2.2) (3,2.2) };
						    \addplot[blueForRed1, thick]coordinates{ (3,2.2) (2.2,2.2) };
						\addplot[blueForRed1, thick]coordinates{ (2.2,2.2) (2.2,2.0) };
						\addplot[darkToffee]coordinates{ (4.5,0) (4.5,2.15) };
\node[color=red1,scale=1.5]  at (axis cs:-1, 1) {$\vec{\bullet}$};
\node[color=red1,scale=1,below left]  at (axis cs:-1, 1) {$\left( - \kappa ,  \kappa \right)$};
\node[color=darkToffee,scale=1.5]  at (axis cs:4.5, 2.15) {$\vec{\bullet}$};
\node[color=darkToffee,scale=1.5]  at (axis cs:4.5, 0) {$\vec{\bullet}$};
\node[color=darkToffee,scale=1,below]  at (axis cs:4.5, -0.1) {$\eta$};
\node[color=blueForRed1,scale=1.5]  at (axis cs:1.99, 1.99) {$\vec{\bullet}$};
     \end{axis}
\end{tikzpicture}
\end{center}
\label{fig:convergenceForSecondOrderPlants}
\caption{Convergence of the iterates of $ - f_{+} \left( \cdot \right)$}
\end{figure}

Consider the flow of the ODE~${\dot{x}} = Ax - B $ at the switching line. 
If $\xi < - \kappa ,$ then the trajectory starting at the point~${\left( \xi \; 0 \right)}^T$ must instantaneously
cross into the open half space~$\bigl\{ {( x_1 \; x_2 )}^T: x_2 < 0 \bigr\} . $ Hence if $ \xi < - \kappa , $
then the switching time $\tau_+ \bigl( {\left( \xi \; 0 \right)}^T \bigr)   $ is zero, and so $ f_{+}\left( \xi \right) = \xi .$

Hence the second and later elements of the sequence~$\Xi \left( \xi_0 \right) $ are strictly positive. 
    For any positive~$\xi$, the following set-inclusion relations hold, due to the monotonicity~\eqref{eqn:fPlusIsMonotonicallyDecreasing} and because~$-f_+$ maps the interval~$ 
    \left[ 0, \, \max{\left\{ \xi, \eta \right\}}  \right]  $ to within itself:
\begin{align*}
- f_+ \bigl( \, \left[ 0, \, \xi  \right] \, \bigr) 
         & \subseteq 
    -               f_+ \bigl( \, \left[ 0, \, \max{\left\{ \xi, \eta \right\}}  \right] \, \bigr) \\
         & \subseteq 
         \left[ 
                               -f_+ \left(  0 \right) ,  \, 
			\max{ \left\{ -  f_+ \left( \xi \right) , -  f_+ \left( \eta \right) \right\} } 
                        \right] , \\  
         & \subseteq \left[ 
                               -f_+ \left(  0 \right) ,  \, 
			\max{ \left\{ \xi , -  f_+ \left( \eta \right) \right\} } 
                        \right] , \\  
& \subset \left[ 
                              0 , \,
															\max{ \left\{ \xi , \eta \right\} } \right] . 
\end{align*}
Thus 
        the function~$ - f_+ $ is a contraction mapping on the interval~$
  \bigl[  0 , \,	\max{ \left\{ -f_+\left(\xi_0\right) , \eta \right\} } \bigr] .
$
Then it follows from Banach's Theorem~\ref{theorem:banachFixedPoint} that the function~$- f_+ \left( \cdot \right) $ has a unique fixed point in
the interval~$ 
  \bigl[  0 , \,	\max{ \left\{ -f_+\left(\xi_0\right) , \eta \right\} } \bigr]
, $
and that the sequence~$\Xi\left( \xi_0 \right)$ converges to this fixed point. Denote this fixed point by~$ {\xi}^{\circlearrowleft}$ (pronounced as xi cycle). Then we get:
\begin{gather*}
\psi_+ \begin{pmatrix} {\xi}^{\circlearrowleft} \\ 0 \end{pmatrix}   =    \begin{pmatrix}  - {\xi}^{\circlearrowleft} \\ 0 \end{pmatrix}  , 
 \psi_- \begin{pmatrix}  - {\xi}^{\circlearrowleft} \\ 0 \end{pmatrix} =   -  \psi_+ \begin{pmatrix}   {\xi}^{\circlearrowleft} \\ 0 \end{pmatrix}  
= \begin{pmatrix} {\xi}^{\circlearrowleft} \\ 0 \end{pmatrix}  .
\end{gather*}
In other words, the symmetric points~$ {\left( {\xi}^{\circlearrowleft} \; 0 \right)}^T , {\left( - {\xi}^{\circlearrowleft} \; 0 \right)}^T $ are the two switching points of a symmetric unimodal limit cycle.

Because solutions of any linear ODE depend continuously on the initial condition and time, we can conclude that every trajectory approaches this limit cycle, as the trajectory's switching points approach those of the limit cycle. 
\end{proof}
Next we analyze three cases of stable, nonminimum phase, second order transfer functions: (a)~having distinct real poles, (b)~having a repeated real pole, and, (c)~having a pair of complex conjugate poles. In each case we show that the above theorem holds.

\subsection{Main result}
\begin{theorem}
Suppose the relay feedback system has a second order linear time invariant plant whose transfer function: (a)~is proper, (b)~is stable, (c)~is nonminimum phase, and (d)~has a positive DC gain.
Then there is a symmetric, unimodal limit cycle to which every trajectory converges asymptotically. 
\end{theorem}
\begin{proof}
It is enough to show that Theorem~\ref{theorem:selfOscillationsForSecondOrderPlants} holds for the above listed three cases of stable, nonminimum phase, second order systems. And to show that this theorem holds, it is enough to show that the switching point transformation function satisfies both the contraction mapping property, and assumption~A~4 in the statement of Theorem~\ref{theorem:selfOscillationsForSecondOrderPlants}. We show precisely this, in the coming three sections where we deal individually with each of the above listed cases.
\end{proof}
\section{Case of two distinct, stable, real poles\label{section:secondOrder:twoRealPoles}}
Consider the case where the transfer function takes the form:
\begin{gather*}
    {\frac{ - \kappa s + \gamma}{\left(s + \alpha \right) \left( s + \beta \right)}},
\end{gather*}
where the real parameters~$\kappa, \gamma, \alpha, \beta$ are all positive, and~$\alpha > \beta.$ Then the observer realization takes the form: 
\begin{gather}
{\dot{x}} \ = \ A x + B u, \quad
y \ = \ Cx, \ \text{where,}
        \label{eqn:companionDistinctReal}
    \\
    A = \begin{bmatrix}
        0 & - \alpha\beta \\
        1 & - \left( \alpha + \beta \right) 
        \end{bmatrix},
    \quad
    B = \begin{bmatrix}
        \gamma \\ -\kappa 
        \end{bmatrix}, \quad
    C = \begin{bmatrix}
        \ 0 & 1 \
        \end{bmatrix}.
        \nonumber
\end{gather}
\subsection{What the geometry of the phase portrait reveals}
\begin{figure}[t]
\begin{center}
\begin{tikzpicture}
\begin{axis}[
        width = 1.0\linewidth, 
        axis lines=middle,
        xmin = -2, xmax = 10,
        ymin = -2, ymax = 8,
        ticks = none,
        colormap/viridis,
        view = {0}{90},
   ]
   \addplot3[ samples = 8,
point meta = {sqrt(x^2+y^2)},
quiver = {
u = {-0.75*y},
v = {x-2*y},
scale arrows = 0.10,
},
quiver/colored = {mapped color},
-stealth,
domain = -2:10,
domain y = -2:8,
] {0};
    %
    \draw[very thick,draw=green!50!blue,,
        decoration={text along path,
            text={invariant line},raise=1ex,
           text align={align=center}
                   },
        postaction={decorate},
         ] (-8/2,-16/3/2) -- (2.5*8,2.5*16/3);
    \draw[very thick,draw=green!50!blue,
        decoration={text along path,
            text={invariant line},raise=1ex,
           text align={align=center}
                   },
        postaction={decorate},
         ] (-8/2,-16/2) -- (10,20);
\draw[ultra thick,draw=green!30!black,] (-5,3) -- (10,3);
    \node[above] at (-0.7,3) [] {\large\textcolor{green!30!black}{Switching}};
    \node[below] at (-1.3,3) [] {\large\textcolor{green!30!black}{line}};
    \node[color=blueForRed1,scale=1.5]  at (6,3) {$\bullet$};
    \node[color=blueForRed1,scale=1,below right]  at (5.8,3) {$\left( - \kappa , 0 \right)$};
    \node[color=blueForRed1,scale=1.5]  at (8,3) {$\bullet$};
    \node[color=blueForRed1,scale=1,below right]  at (7.8,3) {$\left( 0 , 0 \right)$};
    \node[color=blueForRed1,scale=1.5]  at (0,0) {$\bullet$};
    \node[color=blueForRed1,scale=1,below right]  at (0,0) {$A^{-1} b$};
      \draw [draw=none,pattern = north west lines , pattern color = red]
     (3*3/2, 3-0.3)
   --(3*3/2, 3+0.3)
   --(6, 3+0.3)
   --(6, 3-0.3)
   --(3*3/2, 3-0.3);
    \node[color=black,scale=1.2,above right]  at (1.6,-0.1) {Image of interval $\left( - \kappa , \infty \right)$};
    \node[color=black,scale=1.2,below right]  at (1.5,0.1) {on $x_1$-axis, under the map $f_+$};
      \draw [very thick, black, ->]
     (5.6,1.0) -- (5.1,2.65) ;
   %
\end{axis}
\end{tikzpicture}
\end{center}
\caption{\label{fig:distinctRealPoles}Phase portrait under the flow: ${\dot{x}} = A x - b ,$ when $A$ has distinct negative real poles. The flow takes  the semi-infinite interval~$ \left( - \kappa , \infty \right)$ on the $x_1$-axis to the bounded interval~$ \left( - \kappa - \gamma / \alpha , -\kappa \right)  $ on the same axis, where~$\alpha$ is the largest of magnitudes of the poles. }
\end{figure}
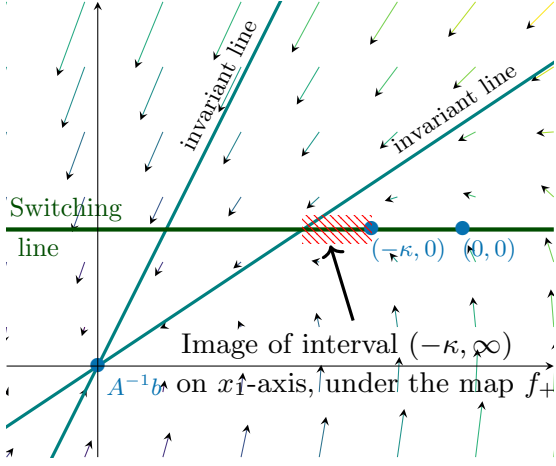
The phase portrait of the affine system in companion form~\eqref{eqn:companionDistinctReal} is shown in Figure~\ref{fig:distinctRealPoles}. This picture visually tells us the fact that the image of the semi-infinite interval~$ \left( - \kappa  , + \infty\right) $ under the function~$ f_+$ is the bounded interval~$  \left( - \kappa - \gamma / \beta , - \kappa \right) .$
So straight-away we get the useful set-inclusion relation:
\begin{align}
    {\vec{-}}  f_+ \bigl(  \left( - \kappa  , + \infty\right) \bigr) 
    &  \subseteq \left( \kappa , \kappa + \gamma / \beta  \right) .
    \label{eqn:setInclusionFromPhasePortrait}
\end{align}
Now we have: (i)~this boundedness of~$ - f_+,$ and (ii)~its monotone increasing property over the interval~$   \left( - \kappa  , + \infty\right)  .$  Are these two properties alone sufficient to guarantee that a unique fixed point exists on this interval~? Not quite.

An example of a function that possesses both of these properties is shown in Figure~\ref{fig:boundednessNotEnough}. Clearly we also need a bound on the magnitude of the derivative of the function~$f_+ .$
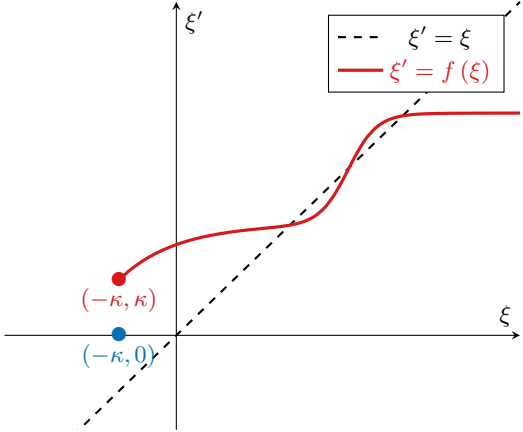
\begin{figure}[t]
 \begin{center}
 \begin{tikzpicture} 
 \begin{axis}[
                    axis lines = middle, 
                    ticks = none,
                    xmin = -3, xmax = 6,
                    ymin = -1.7, ymax = 6,
                    xlabel=$\xi$,  ylabel=$\xi^{\prime}$,
                    legend pos = {north east},
         ]
                \addplot[color=black, dashed,thick, domain=-3:6]   {x};
                \addlegendentry{$ {\xi}^{\prime} =  \xi$}
                \addplot[color=red1, very thick, smooth, samples = 40, domain=-1:6]   {3 + tanh(2*x-6) - exp(-1*x-1)};
                \addlegendentry[red1]{$ {\xi}^{\prime} = f\left( \xi \right)$}
    \node[color=blueForRed1,scale=1.5]  at (axis cs:-1, 0) {$\bullet$};
    \node[color=blueForRed1,scale=1,below]  at (axis cs:-1, 0) {$(-\kappa,0)$};
    \node[color=red1,scale=1.5]  at (axis cs:-1, 1) {$\bullet$};
    \node[color=red1,scale=1,below]  at (axis cs:-1, 1) {$(-\kappa,\kappa)$};
\end{axis}
\end{tikzpicture} 
\end{center}
    \caption{\label{fig:boundednessNotEnough}Example of a monotonic function~$f , $  bounded on the interval~$\left(  - \kappa , \infty \right) . $ Unless suitable further properties are known about~$f$, we cannot rule out the existence of more than one fixed point for~$ f. $}
\end{figure}
\subsection{The trajectory over the duration of a first exit time}
Consider the point~$ \left( \xi_0 , \; 0 \right) $ where $\xi > -\kappa . $ Let~$ \left( p , q \right) $ denote the state of the RFS after~$t$~seconds from starting at~$ \left( \xi_0 , \; 0 \right) . $ Then for time~$t$ such that $0 \le t \le \tau_+\left( \xi , 0 \right)  $:
\begin{align}
\begin{pmatrix}  p \\ q \end{pmatrix}
& =
			   A^{-1}B + e^{A t } \left( \begin{pmatrix}  \xi \\  0 \end{pmatrix}   -  A^{-1} B \right) .
\label{eqn:evolutionLawForDistinctPoles}
\end{align}
The RHS of the above equation can be put in closed form if the matrix exponential in it can be put in closed form.
\subsubsection{Matrix exponential via the Jordan diagonal form}
Since~$A$ is a companion matrix with distinct eigenvalues, we can get its Jordan diagonal form by applying a special similarity transformation involving a Vandermonde matrix and its inverse~\cite{csaki1974conversion}:
\begin{gather*}
 V A V^{-1} \ = \ \text{diag} { \left\{ -\alpha , -\beta \right\} } , \ \  \text{where}  \\
V \ =\ \begin{bmatrix}
          1 & - \alpha \\
					1 & - \beta
       \end{bmatrix}  ,
\quad
V^{-1} \ = \  {\frac{1}{\alpha - \beta}} 
      \begin{bmatrix}
          - \beta &  \alpha \\
					- 1     &  1
       \end{bmatrix} .
\end{gather*}
Then Equation~\eqref{eqn:evolutionLawForDistinctPoles} reduces to
\begin{align}
\begin{pmatrix}  p \\ q \end{pmatrix}
& =
			   A^{-1}B +  V^{-1} V  e^{A t }  V^{-1} V \left( \begin{pmatrix}  \xi \\  0 \end{pmatrix}   -  A^{-1} B \right)
			, \nonumber \\
& =
			   A^{-1}B +  V^{-1}  e^{ V  A  V^{-1} \, t }   V \left( \begin{pmatrix}  \xi \\  0 \end{pmatrix}   -  A^{-1} B \right)
		 , \nonumber \\			
& =
			   A^{-1}B +  V^{-1}  \text{diag} {\left\{ e^{-\alpha t} , e^{-\beta t} \right\}}  V \left( \begin{pmatrix}  \xi \\  0 \end{pmatrix}   -  A^{-1} B \right)
		, \nonumber \\
& =
			     \begin{pmatrix}
	               - {\frac{    \alpha + \beta }{  \alpha \beta }} \gamma - \kappa \\
		-   {\frac{  1 }{  \alpha \beta }} \gamma 
	\end{pmatrix} \nonumber \\
& \quad \quad	+     {\frac{  1 }{  \alpha  - \beta }}  
	                   \begin{pmatrix}  
									      \alpha \nu_{\beta}  \left(  \xi \right) e^{-\beta t}  - 
												  \beta \mu_{\alpha}  \left(  \xi \right) e^{-\alpha t} \\
												\nu_{\beta}  \left(  \xi \right) e^{-\beta t}  - 
													\mu_{\alpha}  \left(  \xi \right) e^{-\alpha t}
										\end{pmatrix} 		,				
\label{eqn:explicitRHSforDistinctPoles}
\end{align}
where~$ \mu_{\alpha} \left(  \xi \right) \triangleq  \xi + \kappa + \gamma / \alpha , $
 and~$ \nu_{\beta} \left(  \xi \right) \triangleq  \xi + \kappa + \gamma / \beta . $

\subsection{Expression for the derivative~${f_+}^{\prime}\left( \xi \right)$}
We can set~$t = \tau_+\left( \xi , 0 \right)$ to get $q = 0, $ and $p = 
\begin{pmatrix}  1 & 0 \end{pmatrix} \psi_+\left( \xi , 0 \right) .$ 
Thus the function~$ {f_+} $ is defined implicitly by Equation~\eqref{eqn:explicitRHSforDistinctPoles}, and in general we cannot get a closed form expression for it.

But we do get a closed form expression for the derivative of this function.
Indeed, by differentiating the component scalar equations in the vector Equation~\eqref{eqn:explicitRHSforDistinctPoles} we get:
\begin{align}
     {\frac{\partial{\tau_+( \xi , 0 ) }}{\partial\xi}} 
& =   
     {\frac
		       {  e^{{\vec{-}}\alpha \tau_+\left( \xi , 0 \right)}  -  e^{{\vec{-}}\beta \tau_+\left( \xi , 0 \right) } 
					 }
					 { \alpha \mu_{\alpha} \left(  \xi \right) e^{{\vec{-}}\alpha \tau_+\left( \xi , 0 \right)  } 
					     - \beta \nu_{\beta}  \left(  \xi \right)  e^{{\vec{-}}\beta \tau_+\left( \xi , 0 \right) } 
					 }
	   }  ,
\label{eqn:tauPrimeForDistinctPoles} \\
     {f_+}^{\prime}\left( \xi \right) 
& =  
     {\frac
		       { \alpha \mu_{\alpha} \left(  \xi \right)  - \beta \nu_{\beta}  \left(  \xi \right)  
					 }
					 { \alpha \mu_{\alpha} \left(  \xi \right) e^{{\vec{+}}\beta \tau_+\left( \xi , 0 \right)  } 
					     - \beta \nu_{\beta}  \left(  \xi \right)  e^{{\vec{+}}\alpha \tau_+\left( \xi , 0 \right) } 
					 }
	   }  .
\label{eqn:fPrimeForDistinctPoles}
\end{align}
\subsection{The derivative~${f_+}^{\prime}\left( \xi \right)$ has magnitude less than one}
We shall bound the magnitude of~${f_+}^{\prime}\left( \xi \right)$ by studying the properties of the functional form of~${f_+}^{\prime}\left( \cdot \right) .$ So let
\begin{align*}
   \phi \left( t \right) & \triangleq 
	 {\frac
		       { \alpha \mu_{\alpha}   - \beta \nu_{\beta}  
					 }
					 { \alpha \mu_{\alpha} e^{{\vec{+}}\beta t  } 
					     - \beta \nu_{\beta}   e^{{\vec{+}}\alpha t } 
					 }
	   } , \ \text{for} \ t \in \left[  0 , \tau_+  \right]
\end{align*}
where, for convenience we have abbreviated the symbols~$  \tau_+\left( \xi , 0 \right)  ,  \mu_{\alpha} \left( \xi , 0 \right)  ,  \nu_{\beta}  \left( \xi , 0 \right) $ as~$  \tau_+  ,  \mu_{\alpha}  ,  \nu_{\beta} $ respectively.  
\subsubsection{Relationship between the functions~$ q ( \cdot ) ,  \phi \left( \cdot \right) $}
Recall that~$q(t)$ is the $x_2$-coordinate as a function of time, as given by Equation~\eqref{eqn:explicitRHSforDistinctPoles}. 
We shall study the behaviour of the function~$q ( \cdot ) $  on~$ \left[ 0 , \tau_+ \right] $   together with the behaviour of the function~$\phi \left( \cdot \right)$ on the same interval. 
If we describe by~$\chi_1 , \chi_2 , \chi_3$ the terms that are independent of~$t$ in the equations for~$q(t)  , \phi (t) ,$ then:
\begin{align*}
q \left( t \right) & = \chi_1  + \chi_2  
                      \left(  
												  \nu_{\beta}   e^{{\vec{-}}\beta t } -  \mu_{\alpha} e^{{\vec{-}}\alpha t  }  
							       \right) \\										
1 / \phi \left( t \right) & = \left( 1 / \chi_3 \right)    
                      \left(  
												 \alpha \mu_{\alpha} e^{{\vec{+}}\beta t  } - \beta \nu_{\beta}   e^{{\vec{+}}\alpha t }  
							       \right) \\
							 & = \left( 1 / \chi_3 \right)  \times e^{\left( \alpha + \beta \right) t} \times {\frac{d}{dt} } 
                      \left(  
												 \nu_{\beta}   e^{{\vec{-}}\beta t } -  \mu_{\alpha} e^{{\vec{-}}\alpha t }  
							       \right) .
\end{align*}
This implies that~$ q^{\prime} \left( t \right) = 0 $ if and only if $ 1 / \phi \left( t \right) = 0  .$ The equation~$ q^{\prime} \left( t \right) = 0 $ has exactly one root because:
\begin{align*}
    q^{\prime} \left( t \right) / \chi_2 & = 
           \alpha \mu_{\alpha} e^{{\vec{-}}\alpha t }  - \beta \nu_{\beta}  e^{{\vec{-}}\beta t },  
\end{align*}
and
(a)~the function~$e^{-\alpha t}$ decays
faster than the function~$e^{-\beta t},$ and
(b)~the coeffecient~$\alpha \mu_{\alpha} $ is bigger than the
coeffecient~$\beta \nu_{\beta}.$  
Denote by~$\tau^*$ the common root of the equations:~$q^{\prime} \left( t \right) = 0  , \; 1 / \phi \left( t \right) = 0 .$
\subsubsection{The first exit time~$\tau_+$ is greater than~$\tau^*$, and $ \phi ( \cdot ) $ grows in magnitude on~$ \left[  \tau^* , \tau_+ \right]$}
We shall now study the rise and fall of~$q (t)$ over~$ \left[ 0 , \tau_+ \right] .$ Because of these four facts:
(i)~{$q \left( \cdot \right)$ is continuously differentiable on~$ \left[ 0 , + \infty \right] ,$ }
(ii)~{$q \left( 0 \right) = 0 ,$ and as~$t \to +\infty , q \left( t \right) \to {-\gamma}/{\alpha\beta} $,}
(iii)~{$q^{\prime} \left( 0 \right) = \alpha \mu_{\alpha}  - \beta \nu_{\beta} = \left( \alpha - \beta \right)  \left( \xi + \kappa \right) > 0 $, and,}
(iv)~{the only critical point of~$q \left( \cdot \right)$ is at the time~$\tau^*,$}
we can make the two inferences:
(a)~$ \left[ 0 , \tau^* \right] $ is an interval of ascent where the function~$q(\cdot )$ rises from~$q(0) = 0$ to its peak~$  q\left( \tau^* \right) ,$
(b)~$ \left[ \tau^* , +\infty \right] $ is an interval of descent where the function~$q(\cdot )$ falls from its peak~$  q\left( \tau^* \right) $ to~$ -\gamma/{\alpha\beta}$. 

Therefore there is exactly one positive time instant~(namely~$\tau_+$) when $q(\cdot )$ equals its initial value~$q(0) = 0.$ And the time instants~$ \tau^* , \tau_+$ must satisfy the order: 
\begin{align*}
\tau^* & < \tau_+ , \ \text{if} \ \xi > - \kappa .
\end{align*}
Note that the denominator of~$\phi\left( \cdot \right)$ has a magnitude that is an increasing function of~$t$ for~$t > \tau^* .$ Since the numerator of~$\phi\left( \cdot \right)$ is independent of~$t,$ it follows that~$\phi\left( \cdot \right)$ has a magnitude that is a decreasing function of~$t$ for~$t > \tau^* .$
\subsubsection{Growth of~$ q ( \cdot ) $ on~$ \left[ 0 , \tau^* \right] $ and on~$ \left[ \tau^* , \tau_+ \right] $}
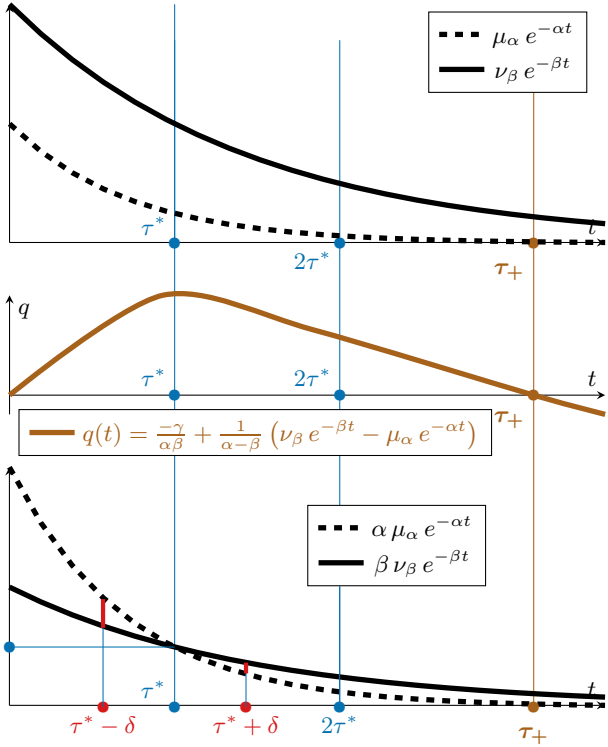
\begin{figure}
\begin{center}
\begin{tikzpicture}[scale=1]
     \begin{groupplot}[
            group style={  group size=1 by 3, vertical sep=20pt,
                           group name=H },
            xtick=\empty,ytick=\empty,
            width=9.5cm,height=4.75cm,
            axis lines=middle,
            clip=false,
            no markers,
            domain=-0.5:2.5,
            samples=20]
                                                            
      \nextgroupplot[  
            xlabel={$t$},
            legend pos=north east,]
				 \node[color=blueForRed1,] at (1.386,0) {$\vec{\bullet}$};
				\node[color=blueForRed1,,below left] at (1.386,0) {$2 {\tau}^*$};
						\draw[blueForRed1] (axis cs:0.693, 1.7) to (axis cs:0.693, -3.92); 
						\draw[blueForRed1] (axis cs:1.386, 1.7) to (axis cs:1.386, -3.92); 
						\draw[darkToffee] (axis cs:2.2, 1.7) to (axis cs:2.2, -3.92); 
          \node[color=blueForRed1,] at (0.693,0) {$\vec{\bullet}$};
				  \node[color=blueForRed1,,above left] at (0.693,0) {${\tau}^*$};
					\draw[blueForRed1] (axis cs:0.693, 2) to (axis cs:0.693, 0);
			  \node[color=darkToffee,] at (2.2,0) {$\vec{\bullet}$};
				\node[color=darkToffee,,below left] at (2.2,-0.1) {${\vec{{\tau}_+}}$};
            \addplot[domain=0:2.5,dashed,line width = 2,] {exp(-2*x)};
            \addplot[domain=0:2.5,line width = 2] {2*exp(-1*x)};
            \legend{%
						        {$ 
										\mu_{\alpha} \, e^{-\alpha t}$},
										{$ 
										\nu_{\beta} \, e^{-\beta t}$}
										};

      \nextgroupplot[  
        xlabel={$t$}, ylabel={$q$}, 
				y post scale = 0.5,
				legend style={at={(0.015,-0.35)},anchor=south west},
        ]
            \node[color=blueForRed1,] at (0.693,0) {$\vec{\bullet}$};
				    \node[color=blueForRed1,,above left] at (0.693,0) {${\tau}^*$};
				\node[color=blueForRed1,] at (1.386,0) {$\vec{\bullet}$};
				\node[color=blueForRed1,,above left] at (1.386,0) {$2 {\tau}^*$};
				    \addplot[ smooth, line width = 2, darkToffee,] coordinates {  (0,0) (0.643,1.3) (1.275,0.85) (2.2,0) (2.5,-0.25) };
 			    	  \node[color=darkToffee,] at (2.2,0) {$\vec{\bullet}$};
				      \node[color=darkToffee,,below left] at (2.2,-0.1) {${\vec{{\tau}_+}}$};
            \legend{\textcolor{darkToffee}{$ q(t) = {\frac{-\gamma}{\alpha\beta}} + {\frac{1}{\alpha - \beta}} \left(  \nu_{\beta} \, e^{-\beta t} -  \mu_{\alpha} \, e^{-\alpha t} \right)$}};

     \nextgroupplot[  
        xlabel={$t$},
        yticklabels={,,}, xticklabels={,,},
				legend style={at={(0.5,0.50)},anchor=south west},
        ]
				\addplot[domain=0:2.5,dashed,line width = 2,] {2*exp(-2*x)};
        \addplot[domain=0:2.5,line width = 2] {exp(-1*x)};
        \node[color=blueForRed1,] at (0.693,0) {$\vec{\bullet}$};
				\node[color=blueForRed1,,above left] at (0.693,0) {${\tau}^*$};
				    \draw[blueForRed1] (axis cs:0, 0.5) to (axis cs:0.693, 0.5);
			  \node[color=blueForRed1,] at (0,0.5) {$\vec{\bullet}$};
        %
				\node[color=red1,] at (0.393,0) {$\vec{\bullet}$};
				\node[color=red1,,below] at (0.393,0) {${\tau}^* - \delta$};
        \node[color=red1,] at (0.993,0) {$\vec{\bullet}$};
				\node[color=red1,,below] at (0.993,0) {${\tau}^* + \delta$};
						\draw[blueForRed1] (axis cs:0.393, 0) to (axis cs:0.393, 0.66);
					      \draw[red1, ultra thick] (axis cs:0.393, 0.66) to (axis cs:0.393, 0.9);
				    \draw[blueForRed1] (axis cs:0.993, 0) to (axis cs:0.993, 0.28);
								\draw[red1, ultra thick] (axis cs:0.993, 0.28) to (axis cs:0.993, 0.37);
        \node[color=blueForRed1,] at (1.386,0) {$\vec{\bullet}$};
				\node[color=blueForRed1,,below] at (1.386,0) {$2 {\tau}^*$};
				\node[color=darkToffee,] at (2.2,0) {$\vec{\bullet}$};
				\node[color=darkToffee,,below] at (2.2,-0.1) {${\vec{{\tau}_+}}$};
        \draw[blueForRed1] (axis cs:1.386, 1) to (axis cs:1.386, 0);
        \legend{$ 
				         \alpha\, \mu_{\alpha} \, e^{-\alpha t}$,
								$\beta\, \nu_{\beta} \, e^{-\beta t}$
								};
\end{groupplot}
\end{tikzpicture}
\end{center}
\label{fig:tauPlusBiggerThanTwiceTauStar}
\caption{Illustration of:~$2 \tau^* < \tau_+$ if $-\kappa < \xi.$}
\end{figure}
 We shall now compare the growth of~$q ( \cdot )$ on the two finite intervals:~$ 
 \left[ 0 , \tau^* \right] ,  \left[ \tau^* , \tau_+ \right] .$ 
Let the time~$\delta$ be chosen such that~$0 \le \delta \le \tau^* .$ Since
$ \alpha \mu_{\alpha} e^{{\vec{-}}\alpha \tau^*} = \beta \nu_{\beta} e^{{\vec{-}}\beta \tau^*} $ we get:
\begin{align*}
q^{\prime} \left( \tau^* - \delta \right) & = \alpha \mu_{\alpha} e^{{\vec{-}}\alpha \tau^*} \times \left( \gamma/\left(\alpha - \beta\right)\right) \times
               \left(  e^{{\vec{+}}\alpha \delta} - e^{{\vec{+}}\beta \delta}  \right) , \\
q^{\prime} \left( \tau^* + \delta \right) & = \alpha \mu_{\alpha} e^{{\vec{-}}\alpha \tau^*} \times \left( \gamma/\left(\alpha - \beta\right)\right) \times
               \left(  e^{{\vec{-}}\alpha \delta} - e^{{\vec{-}}\beta \delta}  \right) ,
\end{align*}
\begin{multline*}
q^{\prime} \left( \tau^* - \delta \right) +
   q^{\prime} \left( \tau^* + \delta \right)
   \\  = {\frac{2  \alpha\mu_{\alpha} \gamma e^{{\vec{-}}\alpha \tau^*} }{\alpha - \beta}} \left( \cosh{\alpha\delta} -  \cosh{\beta\delta} \right) , 
	\end{multline*}
\begin{multline*}
q^{\prime} \left( \tau^* - \delta \right) -
   q^{\prime} \left( \tau^* + \delta \right)
    \\  = {\frac{2  \alpha \mu_{\alpha} \gamma e^{{\vec{-}}\alpha \tau^*} }{\alpha - \beta}}  \left( \sinh{\alpha\delta} -  \sinh{\beta\delta} \right)  .
\end{multline*}
The right hand sides of both of the above equations are positive because~$ \alpha > \beta > 0 .$ 
And this implies that, for $0 \le \delta \le \tau^* ,$
\begin{align*}
    \left\lvert q^{\prime} \left( \tau^* - \delta \right) \right\rvert 
    & >
  \left\lvert q^{\prime} \left( \tau^* + \delta \right) \right\rvert .
\end{align*}
And this inequality in turn implies that the  rise in magnitude of~$q(\cdot )$ over the interval~$\left[ 0 , \tau^* \right] $ is greater than the fall in magnitude of~$q(\cdot )$ over the interval~$\left[ \tau^* , 2\tau^* \right] .$ Indeed,
\begin{align*}
q \left( \tau^* \right) & =  q \left( \tau^* \right) - 0 \; = \; q \left( \tau^* \right) - q(0) , \\
      & = \int_0^{\tau^*}{  {q}^{\prime} \left( \tau^* -\delta \right)  d{\delta}} 
      \; = \; \int_0^{\tau^*}{  \left\lvert{q}^{\prime} \left( \tau^* - \delta \right) \right\rvert d{\delta}} , \\
      & >  \int_0^{\tau^*}{  \left\lvert{q}^{\prime} \left( \tau^* + \delta \right) \right\rvert d{\delta}} 
      \; = \; \int_0^{\tau^*}{   - {q}^{\prime} \left( \tau^* +  \delta \right)  d{\delta}} , \\
      & = q \left( \tau^* \right) - q \left( 2 \tau^* \right) .
\end{align*}
Hence~$ q \left( 2 \tau^* \right) > 0.$ And this has the important consequence:
\begin{align}
  \tau_+ & > 2 \tau^* , \ \text{for} \ \xi > - \kappa . \label{eqn:tauPlusBiggerThanTwiceTauStar} 
\end{align}
\subsubsection{Growth of~$ \phi \left( \cdot \right) $ on~$ \left[ 0 , \tau^* \right] $ and on~$ \left[ \tau^* , \tau_+ \right] $
\label{section:paragraph:growthOfPhiDistinctPoles}}
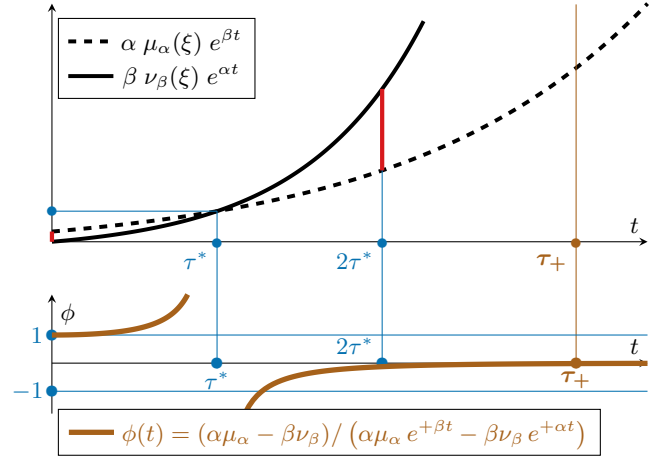
\begin{figure}
\begin{center}
\begin{tikzpicture}[scale = 1.0]
 \begin{groupplot}[
            group style={  group size=1 by 2, vertical sep=20pt,
                           group name=I },
            xtick=\empty,ytick=\empty,
            width=9.5cm,height=4.75cm,
            axis lines=middle,
            clip=false,
            no markers,
            domain=-0.5:2.5,  
            samples=100]
						
		 \nextgroupplot[  
        xlabel={$t$}, 
				legend style={at={(0.01,0.94)},anchor=north west},
        ]      
            \addplot[domain=0:2.5, dashed, line width = 1.5] {0.7*exp(x)};
						\addlegendentry{$ 
						                 \alpha \; \mu_{\alpha}{\left( \xi \right)} \; e^{\beta t} $};
            \addplot[domain=0:1.56, line width = 1.5] {0.35*exp(2*x)};						
						\addlegendentry{$ 
						                 \beta \; \nu_{\beta}{\left( \xi\right)} \; e^{\alpha t} $};
\node[color=blueForRed1,scale=1]  at (axis cs:0.693, 0.3) {$\bullet$};
\node[color=blueForRed1,scale=1,below left]  at (axis cs:0.693, 0.3) {${\tau}^*$};
\node[color=blueForRed1,scale=1]  at (axis cs:0, 1.4) {$\bullet$};
%
\node[color=blueForRed1,scale=1]  at (axis cs:1.386, 0.3) {$\bullet$};
\node[color=blueForRed1,scale=1,below left]  at (axis cs:1.386, 0.3) {$2 {\tau}^*$};
\node[color=darkToffee,scale=1]  at (axis cs:2.2, 0.3) {$\bullet$};
\node[color=darkToffee,scale=1,below left]  at (axis cs:2.2, 0.2) {${\vec{{\tau}_+}}$};
					  	\draw[blueForRed1] (axis cs:0.693, 1.4) to (axis cs: 0.693,-4.0);
						  \draw[blueForRed1] (axis cs:0, 1.4) to (axis cs: 0.693,1.4);
							\draw[blueForRed1] (axis cs:1.386, 3.4) to (axis cs: 1.386,-4.0);
							\draw[darkToffee] (axis cs:2.2, 8.5) to (axis cs: 2.2,-4.0);
             	\draw[red1, ultra thick] (axis cs:0, 0.35) to (axis cs: 0,0.7);
             	\draw[red1, ultra thick] (axis cs:1.386, 2.8) to (axis cs: 1.386,5.6);
%

		 \nextgroupplot[  
        xlabel={$t$}, ylabel={$\phi$}, 
				y post scale = 0.5,
				legend style={at={(0.01,0.04)},anchor=north west},
        ]
				 \node[color=blueForRed1,] at (0,0.1) {$\vec{\bullet}$};
				    \node[color=blueForRed1,,left] at (0,0.1) {$1$};
						 \node[color=blueForRed1,] at (0,-0.1) {$\vec{\bullet}$};
				    \node[color=blueForRed1,,left] at (0,-0.1) {$-1$};
            \node[color=blueForRed1,] at (0.693,0) {$\vec{\bullet}$};
				    \node[color=blueForRed1,,below] at (0.693,0) {${\tau}^*$};
				\node[color=blueForRed1,] at (1.386,0) {$\vec{\bullet}$};
				\node[color=blueForRed1,,above left] at (1.386,0) {$2 {\tau}^*$};
												\addplot[blueForRed1,domain=0:2.5, forget plot]{0.1};
											  \addplot[blueForRed1,domain=0:2.5, forget plot]{-0.1};
						 \addplot[darkToffee, domain=0:0.57,line width = 2, forget plot] {0.1/(2*exp(x) - exp(2*x))};
						\addplot[darkToffee, domain=0.81:2.5,line width = 2,] {0.1/(2*exp(x) - exp(2*x))};
 			    	  \node[color=darkToffee,] at (2.2,0) {$\vec{\bullet}$};
				      \node[color=darkToffee,,below] at (2.2,0) {${\vec{{\tau}_+}}$};
            \legend{\textcolor{darkToffee}{$ \phi (t) = {\left( \alpha\mu_{\alpha} -  \beta\nu_{\beta} \right)} / \left(  \alpha \mu_{\alpha} \, e^{+\beta t} - \beta\nu_{\beta}  \, e^{+\alpha t} \right)$}};
						
			\end{groupplot}
\end{tikzpicture}
\end{center}
\label{fig:caseOneDenominatorOfLevelCrossing}
\caption{Illustration of: $ \left\lvert \phi \left( 2\tau^* \right) \right\rvert < \left\lvert \phi \left( \tau^* \right) \right\rvert $ if $- \kappa < \xi$}
\end{figure}
 We shall now do a similar comparison of the growth of~$\phi ( \cdot )$ on the two finite intervals:~$ 
 \left[ 0 , \tau^* \right] ,  \left[ \tau^* , \tau_+ \right] .$  
Let the time~$\delta$ be chosen such that~$0 \le \delta \le \tau^* .$ Since
$ \alpha \mu_{\alpha} e^{{\vec{+}}\beta \tau^*} = \beta \nu_{\beta} e^{{\vec{+}}\alpha \tau^*} $
\begin{align*}
\chi_3 / \phi \left( \tau^* - \delta \right) & = \alpha \mu_{\alpha} e^{{\vec{+}}\beta \tau^*} 
               \left(  e^{{\vec{-}}\alpha \delta} - e^{{\vec{-}}\beta \delta}  \right) , \\
\chi_3 / \phi \left( \tau^* + \delta \right) & = \alpha \mu_{\alpha} e^{{\vec{+}}\beta \tau^*} 
               \left(  e^{{\vec{+}}\alpha \delta} - e^{{\vec{+}}\beta \delta}  \right) ,
\end{align*}
\begin{multline*}
\chi_3 / \phi \left( \tau^* - \delta \right) +
   \chi_3 / \phi \left( \tau^* + \delta \right)
   \\  = {2  \alpha\mu_{\alpha} e^{{\vec{+}}\beta \tau^*} } \left( \cosh{\alpha\delta} -  \cosh{\beta\delta} \right) ,
\end{multline*}
\begin{multline*}
\chi_3 / \phi \left( \tau^* - \delta \right) -
   \chi_3 / \phi \left( \tau^* + \delta \right)
   \\  = { {\mathbf{-}} 2  \alpha\mu_{\alpha} e^{{\vec{+}}\beta \tau^*} } \left( \sinh{\alpha\delta} -  \sinh{\beta\delta} \right)  .
\end{multline*}
Because~$ \alpha > \beta > 0 ,$ the right hand side of the first equation is positive, but that of the second equation is negative. 
And this implies that~$ \chi_3 /  \left\lvert \phi \left( \tau^* - \delta \right) \right\rvert <
 \chi_3 /  \left\lvert \phi \left( \tau^* + \delta \right) \right\rvert $ for $0 \le \delta \le \tau^* .$ Setting~$\delta = \tau^*$ we get:
\begin{align*}
  \left\lvert \phi \left(  2\tau^*  \right) \right\rvert & < \left\lvert \phi \left(  0 \right) \right\rvert  \; = \; 1.
\end{align*}
And since the magnitude of~$\phi\left( \cdot\right)$ decreases monotonically on the interval~$ \left[ \tau^* , \tau_+ \right] ,$ we can now say that if~$\xi > -\kappa,$ then
\begin{align}
  \left\lvert {f_+}^{\prime} \left(  \xi  \right) \right\rvert & =  \left\lvert \phi \left(  \tau_+  \right) \right\rvert \; < \;
	\left\lvert \phi \left(  2\tau^*  \right) \right\rvert \; < \; 1 .
	\label{eqn:fPlusPrimeLessThanOne}
\end{align}
\subsubsection{On two different routes to prove contractiveness}
We have established that on the interval~$\left( - \kappa , + \infty \right)$ the derivative~$ {f_+}^{\prime}\left( \xi \right)  $ is negative, and that its magnitude is less than one. These properties by themselves are insufficient to prove that~$- f_+$ is a contraction mapping.   

For example the function~$  \kappa + \kappa e^{- \left( \xi -1 \right) / \kappa}  + \xi $ for positive values of~$\xi,$ is positive and motonically increasing. Its derivative has a magnitude less than one. Yet as seen in Figure~\ref{fig:notEnough} the graph of this function always lies above that of the function~$\xi,$ and never intersects the latter.
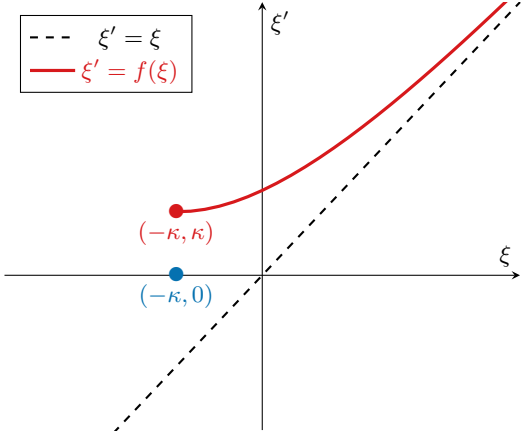
\begin{figure}
 \begin{center}
 \begin{tikzpicture} 
 \begin{axis}[
                    axis lines = middle, 
                    ticks = none,
                    xmin = -3, xmax = 3,
                    ymin = -1.7, ymax = 3,
                    xlabel=$\xi$,  ylabel=$\xi^{\prime}$,
                    legend pos = {north west},
         ]
                \addplot[color=black, dashed,thick, domain=-3:3]   {x};
                \addlegendentry{$ {\xi}^{\prime} =  \xi$}
                \addplot[color=red1, very thick, domain=-1:3]   {x + 2.8*exp(-1.1-0.6*x)};
                \addlegendentry[red1]{$ {\xi}^{\prime} = f ( \xi ) 
                $}
    \node[color=blueForRed1,scale=1.5]  at (axis cs:-1, 0) {$\bullet$};
    \node[color=blueForRed1,scale=1,below]  at (axis cs:-1, 0) {$(-\kappa,0)$};
    \node[color=red1,scale=1.5]  at (axis cs:-1, 0.69) {$\bullet$};
    \node[color=red1,scale=1,below]  at (axis cs:-1, 0.69) {$(-\kappa,\kappa)$};
\end{axis}
\end{tikzpicture} 
\end{center}
\caption{\label{fig:notEnough}Example of a monotonic function~$f$ that has a derivative whose magnitude is less than one, on the interval~$\left(  - \kappa , \infty \right) . $ Unless suitable further properties are known about~$f$, we cannot guarantee that a  fixed point exists for~$ f . $}
\end{figure}

To be a contraction mapping, a function needs two things: (i)~an interval that is mapped to within itself,  and (ii)~a Lipshitz constant that is less than one, at least on this interval. 
Equation~\eqref{eqn:fPlusPrimeLessThanOne} certifies that the Lipshitz constant of~$ - f_+$ is less than one over any subset of the positive real axis.

    The geometry of the phase flow had revealed a boundedness~\eqref{eqn:setInclusionFromPhasePortrait}. This helps to make~$ - f_+$ a contraction mapping. Given any~$\xi_0 \ge 0,$ we get the following set-inclusion relation:
\begin{align*}
    - {f_+} \bigl( \left[ 0, \, \max{\left\{ \xi_0 , \kappa + \gamma / \beta \right\}} \right] 
          \bigr)
    & \subseteq \left[ 0, \, \max{\left\{ \xi, \kappa + \gamma / \beta \right\} } \right] .
\end{align*}%
Thus given any non-negative initial value~$\xi_0 , $ we can pick an interval over which $ - f_+$~is a contraction mapping. And this is sufficient to apply Theorem~\ref{theorem:selfOscillationsForSecondOrderPlants}.
Another geometric proof is given in Appendix~\ref{sec:geometricProof}. 

But we shall also give a different way using analytic properties of~$f_+$ to prove that~$- f_+$ is a contraction map. We do this for two reasons: (i)~along the way, we prove that~$f_+$ is strictly convex at any nonnegative value of its argument, which is interesting to know, and, (ii)~in the case where the plant has complex conjugate poles, the phase portrait does not reveal any set-inclusion properties of~$f_+ . $ 
\subsection{The function~${f_+}\left( \xi \right)$ is strictly convex}
To show that~$ {f_+} \left(  \cdot  \right)$ is strictly convex on the interval~$ \left( -\kappa + \gamma/\beta , + \infty \right)$  it is enough to show that~${f_+}^{\prime \prime} \left(  \cdot  \right) > 0$ on this interval. 
Using Equations~\eqref{eqn:tauPrimeForDistinctPoles},~\eqref{eqn:fPrimeForDistinctPoles} we get:
\begin{align}
     {f_+}^{\prime \prime}\left( \xi \right) 
& =  
     {\frac
		       { G_0   - G_{\text{plus}}   e^{{\vec{+}}\left( \alpha - \beta \right) \tau_+ }
					        -  G_{\text{minus}}  e^{{\vec{-}}\left( \alpha - \beta \right) \tau_+ }
					 }
					 {   
					      e^{ {\vec{-}}\left( \alpha + \beta \right) \tau_+ } \;
							 {  \bigl( \alpha \mu_{\alpha} \ e^{{\vec{+}}\beta \tau_+  } 
					        - \beta \nu_{\beta}   e^{{\vec{+}}\alpha \tau_+ } 
							   \bigr) }^3
					 }
	   }  ,
\label{eqn:fDoublePrimeForDistinctPoles}
\end{align}
where the coefficients of the numerator are given by:
\begin{gather*} 
G_0 \, = \, 2 \left( \alpha - \beta \right) \left( \alpha \beta  {\left( \xi + \kappa \right)}^2 - \gamma^2 \right) , \\
G_{\text{plus}} \, =  \, \left( \alpha - \beta \right) \left( \alpha \beta  {\left( \xi + \kappa \right)}^2 +  \gamma \left( \alpha - \beta \right) {\left( \xi + \kappa \right)} - \gamma^2  \right) , \\
G_{\text{minus}} \, =  \,  \left( \alpha - \beta \right) \left(\alpha \beta  {\left( \xi + \kappa \right)}^2 - \gamma  \left( \alpha - \beta \right) {\left( \xi + \kappa \right)} - \gamma^2 \right)
.
\end{gather*}
The denominator in the RHS of Equation~\eqref{eqn:fDoublePrimeForDistinctPoles} is negative, because~$ \tau^* < \tau_+ .$ We shall show that the numerator too is negative, by studying its functional form. So let
\begin{align*}
 g\left(X\right)
&  \triangleq 
{ G_0  \, -  \, G_{\text{plus}}   X
					       \, -  \, G_{\text{minus}} / X ,
					 }  \quad \text{for} \ X \in \left( 0 ,  + \infty \right) .
\end{align*}
The function~$g\left(\cdot\right)$ is continuously differentiable on~$\left( 0 , + \infty \right) .$ On this interval it has exactly one critical point given by:
\begin{align*}
- G_{\text{plus}}  + G_{\text{minus}} / {\left({X^*}\right)}^2  = 0 \quad \implies \quad {X^*} = \sqrt{ G_{\text{minus}} /  G_{\text{plus}}   } .
\end{align*}
and~$g\left( \cdot \right)$ is monotonically decreasing on~$\left( X^* , + \infty \right) ,$ because~$g^{\prime}{\left( X \right)} < 0 $ 
for~$ X > X^*.$

Note that (a)~$ G_0 , G_{\text{plus}} , G_{\text{minus}} > 0 $ if~$\xi > -\kappa + \gamma/\beta ,$ and
(b)~$ G_{\text{plus}} - G_{\text{minus}} =
2 \gamma {\left( \xi + \kappa \right)} {\left( \alpha - \beta \right)}^2   
 > 0 .$ Hence the critical point~$  \sqrt{ G_{\text{minus}} /  G_{\text{plus}} }  < 1. $ Hence~$  g\left(X\right) < g\left( 1 \right) , $ if~$  X > 1.$
But~$  g\left( 1 \right) = 0. $ Hence~$  
g\left(X\right)  < 0 , \ \text{if} \  X > 1 , $ and thus for every~$ \xi > -\kappa +\gamma/\beta  $ we have:
\begin{align*}
  {f_+}^{\prime \prime}\left( \xi \right)  & =  
{\frac {  g\left( e^{ + \left( \alpha - \beta \right) \tau_+ } \right) } 
       {  e^{ {\vec{-}}\left( \alpha + \beta \right) \tau_+ } \;
							 {  \bigl( \alpha \mu_{\alpha} \ e^{{\vec{+}}\beta \tau_+  } 
					        - \beta \nu_{\beta}   e^{{\vec{+}}\alpha \tau_+ } 
							   \bigr) }^3
                              }
												}
											\; > \; 0 .
\end{align*}

\subsection{An interval mapped to within itself by~${\mathbf{-}}{f_+}\left( \cdot \right)$%
\label{section:subsub:assumptionA4distinctRealPoles}}
Since~$ f_+\left( \cdot \right) $ is strictly convex on~$\left( -\kappa + \gamma/\beta , +\infty\right)  , $
its graph on this interval lies above or on the tangent drawn at any point on this interval.
Fix~${\xi}^{\text{cvx}} \in \left( -\kappa + \gamma/\beta , +\infty\right)  . $
Then for every~$ \xi \in \left( -\kappa + \gamma/\beta , +\infty\right)$:
\begin{align*}
     {f_+} \left( \xi \right) 
& \ge  {f_+}^{\prime}\left( {\xi}^{\text{cvx}} \right) \bigl( \xi - {\xi}^{\text{cvx}} \bigr) + {f_+} \left( {\xi}^{\text{cvx}}\right) .
\end{align*}
The two straight lines in $ ( \xi , {\widetilde{\xi}} ) $-space given by:
\begin{align*}
     {\widetilde{\xi}} 
& =  {f_+}^{\prime}\left( {\xi}^{\text{cvx}} \right) \bigl( \xi - {\xi}^{\text{cvx}} \bigr) 
+  {f_+}\left( {\xi}^{\text{cvx}} \right) ,
 \\ 
 {\widetilde{\xi}}  & = -\xi ,
\end{align*}
have exactly one intersection point because they have different slopes. Let this intersection happen at~$ \xi = {\xi}^{\text{ixn}}  .$
Then at any~%
$ \xi \ge \max{\left\{ {\xi}^{\text{cvx}} ,  {\xi}^{\text{ixn}}  \right\}}$
we can use convexity to get:
\begin{align*}
     {f_+} \left( \xi \right) 
& \ge  {f_+}^{\prime}\left( {\xi}^{\text{cvx}} \right) \bigl( \xi - {\xi}^{\text{cvx}} \bigr) + {f_+} \left( {\xi}^{\text{cvx}}\right) , \\
& =  {f_+}^{\prime}\left( {\xi}^{\text{cvx}} \right) \bigl( \xi - {\xi}^{\text{cvx}} \bigr) - {\xi}^{\text{ixn}} +
   {f_+}^{\prime}\left( {\xi}^{\text{cvx}} \right) \bigl( {\xi}^{\text{cvx}} - {\xi}^{\text{ixn}} \bigr)
 , \\
& = {f_+}^{\prime}\left( {\xi}^{\text{cvx}} \right) \bigl( \xi - {\xi}^{\text{ixn}} \bigr) - {\xi}^{\text{ixn}} 
 \\
& \ge - \bigl( \xi - {\xi}^{\text{ixn}} \bigr)  - {\xi}^{\text{ixn}}  
\\
& \ge - \xi .
\end{align*}
Thus we have the set inclusion relation:
\begin{align}
    - {f_+} \left( \left[ 0, \xi \right] 
          \right)
    & \subseteq \left[ 0, \xi \right] , \quad \text{if} \  \xi \ge 
 \eta \triangleq   \max{ \left\{ 
    {\xi}^{\text{cvx}} 
    , {\xi}^{\text{ixn}} \right\}} .
\label{eqn:assumptionA3forDistinctPoles}
\end{align}%
\subsection{\label{section:fPrimeIsContinuous}{The derivative~${f_+}^{\prime}\left( \cdot \right)$ is continuous on~$ \left[ 0 , + \infty \right)  $}}
The first exit time~$\tau_+\bigl( {\left(\xi \; 0\right)}^T  \bigr) $ is implicitly defined by Equation~\eqref{eqn:explicitRHSforDistinctPoles}. Both the LHS and the RHS of this equation are continuously differentiable in~$ \xi .$ The implicit function theorem states that the implicitly defined function inherits the differentiability properties of the defining functions. Thus if~$\tau_+\bigl( {\left( \xi ,  0 \right)}^T \bigr)$ is defined at some~$ \xi $, then there is a small neighbourhood including~$\xi$ where~$\tau_+\bigl( {\left(\xi  \; 0\right)}^T \bigr)$ is also continuously differentiable in~$ \xi .$

The derivative~${f_+}^{\prime}\left( \cdot \right)$ is a fraction whose numerator and denominator are continuously differentiable  functions of~$ \xi .$ And the derivative is defined everywhere on the interval~$\left[ 0 , + \infty \right) . $ Hence the derivative is continuous  on the interval~$\left[ 0 , + \infty \right) . $

\subsection{The function~${f_+}\left( \xi \right)$ is a contraction mapping
\label{section:subsub:contractionMappingDistinctRelalPoles}}

The derivative~$ f_+^{\prime}\left( \xi \right) $ has a magnitude that is a decreasing function on~$\left( -\kappa + \gamma/\beta , + \infty \right) $ because
\begin{align*}
{\frac{d}{d\xi}}{{\left( {f_+}^{\prime}\left( \xi \right)   \right)}^2} & =  2 {f_+}^{\prime}\left( \xi \right)
                                {f_+}^{\prime\prime}\left( \xi \right) , \ \text{and for} \ \xi > -\kappa + \gamma/\beta , \\
       & < 0 , \quad \text{as} \ f_+ \ \text{is strictly convex, and} \ {f_+}^{\prime} < 0 .  
\end{align*}
Thus~$ - {f_+}^{\prime}\left( -\kappa + \gamma/\beta \right)  $ is a tight upper bound for the magnitude of~$ {f_+}^{\prime}\left( \cdot \right)  $ on the
semi-infinite interval~$\left[ -\kappa + \gamma/\beta , + \infty \right) .$

We now find a tight upper bound on the finite interval~$\left[ 0 , -\kappa + \gamma/\beta \right] .$ The function~$ {f_+}^{\prime}\left( \xi \right) $ is continuous on this interval. By the Weierstrass theorem on extreme values of continuous functions, it follows that the magnitude of~$ {f_+}^{\prime}\left( \cdot \right)  $ attains its maximum on this closed and bounded interval. Denote this maximum value by~${\overline{\lambda}} . $  

Since~${\overline{\lambda}} \ge  - {f_+}^{\prime}\left( -\kappa + \gamma/\beta \right) , $ it follows that~${\overline{\lambda}}$  is also an upper bound for the magnitude of~$ {f_+}^{\prime}\left( \cdot \right)  $ on the
semi-infinite interval~$\left[  0 , + \infty \right] .$ Clearly
\begin{align*}
  {\overline{\lambda}} & < 1 , \ \text{because} \ \left\lvert {f_+}^{\prime}\left( \xi \right) \right\rvert < 1,
         \ \text{for every} \ \xi \in \left( -\kappa , + \infty \right) .
\end{align*}
And so we have the following useful inequality:
\begin{align}
 \left\lvert {f_+}^{\prime}\left( \xi \right) \right\rvert & \le {\overline{\lambda}} \; < \; 1,  \ \text{for every} \ \xi > 0 .
\label{eqn:lambdaForFprime}
\end{align}
Then for any two nonegative~$\xi ,{\xi}^{\prime} $ we get:
\begin{align}
   \left\lvert f_+\left( \xi \right) - f_+\left( {\xi}^{\prime}  \right) \right\rvert 
& = \left\lvert   \int_{\min{ \left\{ \xi ,{\xi}^{\prime} \right\}}}^{\max{ \left\{ \xi ,{\xi}^{\prime} \right\}}}{
f_+^{\prime}\left( {\breve{\xi}} \right) d{\breve{\xi}}} \right\rvert , \nonumber \\
& \le \int_{\min{ \left\{ \xi ,{\xi}^{\prime} \right\}}}^{\max{ \left\{ \xi ,{\xi}^{\prime} \right\}}}{
 \left\lvert   f_+^{\prime}\left( {\breve{\xi}} \right) \right\rvert d{\breve{\xi}}} \nonumber \\
& \le \int_{\min{ \left\{ \xi ,{\xi}^{\prime} \right\}}}^{\max{ \left\{ \xi ,{\xi}^{\prime} \right\}}}{
 {\overline{\lambda}} d{\breve{\xi}}} , \nonumber \\
& \le \ {\overline{\lambda}} \left\lvert \xi - {\xi}^{\prime} \right\rvert  .
\label{eqn:contractionMappingPropertyForDistinctPoles}
\end{align}
    Hence by~\eqref{eqn:assumptionA3forDistinctPoles}, \eqref{eqn:contractionMappingPropertyForDistinctPoles} the function
$f_+(\cdot )$  is a contraction mapping over any interval of the form~$\left[ 0 , \theta \right]$ with $ \theta \ge \eta .$%

\section{Case of a stable, repeated, real pole\label{section:secondOrder:RepeatedRealPole}}
We carry out an entirely similar set of calculations for the case where the transfer function takes the form:
\begin{gather*}
    {\frac{ - \kappa s + \gamma}{ {\left(s + \alpha \right) }^2}},
\end{gather*}
where the real parameters~$\kappa, \gamma, \alpha$ are all positive. Because the calculations are similar, the presentation below is somewhat abbreviated when compared to the case of distinct real poles.

The observer realization takes the form: 
\begin{gather*}
{\dot{x}} \ = \ A x + B u, \quad
y \ = \ Cx, \ \text{where,} \\
    A = \begin{bmatrix}
        0 & - \alpha^2 \\
        1 & - 2 \alpha  
        \end{bmatrix},
    \quad
    B = \begin{bmatrix}
        \gamma \\ -\kappa 
        \end{bmatrix}, \quad
    C = \begin{bmatrix}
        \ 0 & 1 \
        \end{bmatrix}.
\end{gather*}
\subsection{The trajectory over the duration of a first exit time}
Consider the starting point~$ \left( \xi_0 , \; 0 \right) $ where $\xi > -\kappa . $ Let~$ \left( p , q \right) $ denote the state of the RFS after a duration of~$t$~seconds. 
\subsubsection{Matrix exponential via the Jordan canonical form}
Since~$A$ is a companion matrix with repeated eigenvalues, we get its Jordan canonical form via a similarity transformation involving a confluent Vandermonde matrix and its inverse~\cite{csaki1974conversion}:
\begin{gather*}
 {\widetilde{V}} A {\widetilde{V}}^{-1} \ = \ \begin{bmatrix}
           - \alpha & 0\\
					 1  &  - \alpha
       \end{bmatrix}  , \ \  \text{where}  \\
{\widetilde{V}} \ =\ \begin{bmatrix}
          1 & - \alpha \\
					0 & 1
       \end{bmatrix}  ,
\quad
{\widetilde{V}}^{-1} \ = \ 
      \begin{bmatrix}
          1 &  \alpha \\
					0     &  1
       \end{bmatrix} .
\end{gather*}
The matrix exponential of~$ {\widetilde{V}} A {\widetilde{V}}^{-1}  \; t $ can be computed by expressing it as a sum of two commuting matrices, as below:
\begin{align*}
e^{   \begin{bmatrix}
           - \alpha & 0\\
					 1        &  - \alpha
       \end{bmatrix}        t
	} 
& =  
e^{  \left\{
  -\alpha t \begin{bmatrix}
           1  & 0 \\
					 0  &  1
       \end{bmatrix} 
+  t
    \begin{bmatrix}
           0  & 0  \\
					 1  & 0
       \end{bmatrix} 
			\right\}
	} ,  \\
	& =  
\begin{bmatrix}
           e^{  -\alpha t }  & 0 \\
					 0  &   e^{  -\alpha t }
       \end{bmatrix} 
\times 
 \begin{bmatrix}
           1  & 0  \\
					 t  & 1
       \end{bmatrix}  ,  \\
& =	\begin{bmatrix}
           e^{  -\alpha t }     &   0  \\
					 t  e^{  -\alpha t }  &   e^{  -\alpha t }
       \end{bmatrix} .
\end{align*}
Then for time~$t$ such that $0 \le t \le \tau_+\left( \xi , 0 \right)  $:
\begin{align}
\begin{pmatrix}  p \\ q \end{pmatrix}
& =
			   A^{-1}B +  {\widetilde{V}}^{-1}  e^{{\widetilde{V}}  A  {\widetilde{V}}^{-1} \, t }   {\widetilde{V}} \left( \begin{pmatrix}  \xi \\  0 \end{pmatrix}   -  A^{-1} B \right)
		 , \nonumber \\
& =
	- \begin{pmatrix}
	                {\frac{  2 }{  \alpha  }} \gamma + \kappa \\
		  {\frac{  1 }{  \alpha^2}} \gamma 
	\end{pmatrix} +
	                   \begin{pmatrix}  
									    \bigl[   
											      {\widetilde{\mu}}_{\alpha}   +  \gamma / \alpha  
												    + \alpha {\widetilde{\mu}}_{\alpha} \,  t \bigr]  e^{-\alpha t} \\
											\bigl[	{\widetilde{\mu}}_{\alpha} \, t 
													 +  
													     \gamma / \alpha^2  
															\bigr]   e^{-\alpha t}
										\end{pmatrix} 		,			
\label{eqn:explicitRHSforRepeatedPoles}
\end{align}
where~$ {\widetilde{\mu}}_{\alpha}\left(  \xi \right) \triangleq  \xi + \kappa + \gamma / \alpha . $

\subsection{Expression for the derivative~${f_+}^{\prime}\left( \xi \right)$}
Differentiating the component scalar equations in~\eqref{eqn:explicitRHSforRepeatedPoles} gives:
\begin{align}
     {\frac{\partial{\tau_+( \xi , 0 ) }}{\partial\xi}} 
& =  - 
     {\frac
		       {  \tau_+ }
					 { \xi + \kappa  
					     - \alpha \tau_+  \left(  \xi + \kappa + \gamma / \alpha  \right)  
					 }
	   }  ,
\label{eqn:tauPrimeForRepeatedPoles} \\
     {f_+}^{\prime}\left( \xi \right) 
& =  
     {\frac
		       {  e^{ - \alpha \tau_+ } \;  \left(  \xi + \kappa  \right)  
					 }
					 { \xi + \kappa   
					     - \alpha \tau_+  \left(  \xi + \kappa + \gamma / \alpha  \right)  
					 }
	   } ,
\label{eqn:fPrimeForRepeatedPoles}
\end{align}
where we have abbreviated the symbol~$  \tau_+\left( \xi , 0 \right) $ as~$  \tau_+ . $
\subsection{The derivative~${f_+}^{\prime}\left( \xi \right)$ has magnitude less than one}
We shall bound the magnitude of~${f_+}^{\prime}\left( \xi \right)$ by studying the properties of the functional form of~${f_+}^{\prime}\left( \cdot \right) .$ So let
\begin{align*}
   {\widetilde{\phi}} \left( t \right) & \triangleq 
     {\frac
		       {  e^{ - \alpha t } \;  \left(  \xi + \kappa  \right)  
					 }
					 { \xi + \kappa   
					     - \alpha t  \left(  \xi + \kappa + \gamma / \alpha  \right)  
					 }
	   }  , \ \text{for} \ t \in \left[  0 , \tau_+  \right] .
\end{align*}
\subsubsection{Relationship between the functions~$ q ( \cdot ) ,  {\widetilde{\phi}} \left( \cdot \right) $}
Consider the behaviour of the function~$q ( \cdot ) $  on~$ \left[ 0 , \tau_+ \right] $   together with the behaviour of the
function~${\widetilde{\phi}}  \left( \cdot \right)$ on the same interval. We have:
\begin{align*}
q \left( t \right) & = -{ \frac{\gamma}{{\alpha}^2} } + \bigl[ 
											      \gamma / \alpha^2  
												    +   \left( 
											      \xi + \kappa  + \gamma / \alpha  \right) \,  t \bigr]  e^{-\alpha t} , \\										
\left( \xi + \kappa \right)  / {\widetilde{\phi}}  \left( t \right) & =  
                      \bigl[  
												\xi + \kappa  - \alpha t  \left( \xi + \kappa + \gamma / \alpha \right)   
							       \bigr]  e^{{\vec{+}}\alpha t } ,  \\
							 & = 
							           {\frac{d}{dt}} 
											\left\{ \bigl[
											       \gamma / \alpha^2  
												    +  \left( 
											      \xi + \kappa  + \gamma / \alpha  \right) \,  t \bigr]  e^{-\alpha t}  \right\} .
\end{align*}
This implies that~$ q^{\prime} \left( t \right) = 0 $ if and only if $ 1 / {\widetilde{\phi}} \left( t \right) = 0  .$ The equation~$ q^{\prime} \left( t \right) = 0 $ has exactly one root because ~$ q^{\prime} \left( t \right) $ is a product of two factors, one exponential in~$t$ and another linear in~$t$ - the root comes from the linear factor.
Denote by~$\tau^*$ the common root of the equations:~$q^{\prime} \left( t \right) = 0  , \; 1 /{\widetilde{\phi}}\left( t \right) = 0 .$
\subsubsection{The first exit time~$\tau_+$ is greater than~$\tau^*$, and $ \phi ( \cdot ) $ grows in magnitude on~$ \left[  \tau^* , \tau_+ \right]$}
We shall now study the rise and fall of~$q (t)$ over~$ \left[ 0 , \tau_+ \right] .$ Because of these four facts:
(i)~{$q \left( \cdot \right)$ is continuously differentiable on~$ \left[ 0 , + \infty \right] ,$ }
(ii)~{$q \left( 0 \right) = 0 ,$ and as~$t \to +\infty , q \left( t \right) \to {-\gamma}/{\alpha\beta} $,}
(iii)~{$q^{\prime} \left( 0 \right) =  \xi + \kappa > 0 $, and,}
(iv)~{the only critical point of~$q \left( \cdot \right)$ is at the time~$\tau^*,$}
we can make the two inferences:
(a)~the interval~$ \left[ 0 , \tau^* \right] $ is an interval of ascent where the function~$q(\cdot )$ rises from~$q(0) = 0$ to its peak~$  q\left( \tau^* \right) ,$
(b)~the interval~$ \left[ \tau^* , +\infty \right] $ is an interval of descent where the function~$q(\cdot )$ falls from its peak~$  q\left( \tau^* \right) $ to~$ -\gamma/{\alpha\beta}$. 

Therefore there is exactly one positive time instant~(namely~$\tau_+$) when $q(\cdot )$ equals its initial value~$q(0) = 0.$ And the time instant~$\tau_+$ must satisfy: 
\begin{align*}
\tau^* & < \tau_+ , \ \text{if} \ \xi > - \kappa .
\end{align*}
Note that the denominator of~${\widetilde{\phi}}\left( \cdot \right)$ has a magnitude that is an increasing function of~$t$ for~$t > \tau^* .$ Since the numerator of~${\widetilde{\phi}}\left( \cdot \right)$ is positive, and decreasing in~$t,$ it follows that~$\phi\left( \cdot \right)$ has a magnitude that is a decreasing function of~$t$ for~$t > \tau^* .$
\subsubsection{Growth of~$ q ( \cdot ) $ on~$ \left[ 0 , \tau^* \right] $ and on~$ \left[ \tau^* , \tau_+ \right] $
\label{section:paragraph:growthOfQDistinctPoles}}
We shall now compare the growth of~$q ( \cdot )$ on the two finite intervals:~$ 
 \left[ 0 , \tau^* \right] ,  \left[ \tau^* , \tau_+ \right] .$ 
Let the time~$\delta$ be chosen such that~$0 \le \delta \le \tau^* .$ 
Since~$ \xi + \kappa  = \alpha  \left( \xi + \kappa + \gamma / \alpha \right) \tau^* , $ we get:
\begin{align*}
 q^{\prime} \left( \tau^* - \delta \right)
 & = + {\vec{e^{+\alpha \delta}}} \times \alpha \delta \left( \xi + \kappa + \gamma / \alpha \right)  e^{-\alpha \tau^*}  , \\
q^{\prime} \left( \tau^* + \delta \right) 
& =  - {\vec{e^{-\alpha \delta}}}  \times \alpha \delta \left( \xi + \kappa + \gamma / \alpha \right) e^{-\alpha \tau^*}  .
\end{align*}
Clearly, $ \left\lvert  q^{\prime} \left( \tau^* - \delta \right)   \right\rvert
         > \left\lvert  q^{\prime} \left( \tau^* + \delta \right)  \right\rvert $ 
if~$ 0 \le \delta \le \tau^* .$  The derivative~$ q^{\prime} \left( \cdot \right) $ is positive over the
 interval~$\left[ 0 , \tau^* \right) ,$ and is negative over the interval~$\left[ \tau^* , +\infty \right) . $
In other words, the derivative changes sign only at the point~$\tau^* . $ Therefore we can infer that the magnitude of rise in the value of the function over the interval~$\left[ 0 , \tau^* \right) $ is greater than the magnitude of fall over the interval~$\left[ \tau^* , 2\tau^*\right) . $ Indeed, just at we derived in the last part of Section~\ref{section:paragraph:growthOfQDistinctPoles},
\begin{align*}
q \left( \tau^* \right)
      & = \int_0^{\tau^*}{  {q}^{\prime} \left( \tau^* -\delta \right)  d{\delta}} 
      \; = \; \int_0^{\tau^*}{  \left\lvert{q}^{\prime} \left( \tau^* - \delta \right) \right\rvert d{\delta}} , \\
      & >  \int_0^{\tau^*}{  \left\lvert{q}^{\prime} \left( \tau^* + \delta \right) \right\rvert d{\delta}} 
      \; = \; \int_0^{\tau^*}{   - {q}^{\prime} \left( \tau^* +  \delta \right)  d{\delta}} , \\
      & = q \left( \tau^* \right) - q \left( 2 \tau^* \right) .
\end{align*}
Hence~$ q \left( 2 \tau^* \right) > 0.$ And this has the important consequence:
\begin{align}
  \tau_+ & > 2 \tau^* , \ \text{for} \ \xi > - \kappa . \label{eqn:tauPlusBiggerThanTwiceTauStarRepeatedRoot} 
\end{align}
\subsubsection{Growth of~$ {\widetilde{\phi}} \left( \cdot \right) $ on~$ \left[ 0 , \tau^* \right] $ and on~$ \left[ \tau^* , \tau_+ \right] $}
Just as we did in Section~\ref{section:paragraph:growthOfPhiDistinctPoles} we shall now 
compare the growth of~${\widetilde{\phi}} ( \cdot )$ on the two finite intervals:~$ 
 \left[ 0 , \tau^* \right] ,  \left[ \tau^* , \tau_+ \right] .$  
Let the time~$\delta$ be chosen such that~$0 \le \delta \le \tau^* .$
Since~$ \xi + \kappa  = \alpha  \left( \xi + \kappa + \gamma / \alpha \right) \tau^* , $ we get:
\begin{align*}
\left( \xi + \kappa \right)  / {\widetilde{\phi}}  \left(  \tau^* - \delta \right) & = 
 & = + e^{{\vec{-}}\alpha \delta} \times \alpha \delta \left( \xi + \kappa + \gamma / \alpha \right)  e^{\alpha \tau^*}  , \\
\left( \xi + \kappa \right)  / {\widetilde{\phi}}  \left(  \tau^* - \delta \right) & = 
 & = - e^{{\vec{+}}\alpha \delta} \times \alpha \delta \left( \xi + \kappa + \gamma / \alpha \right)  e^{\alpha \tau^*}  . \\
\end{align*}
Clearly, $ \left\lvert  q^{\prime} \left( \tau^* - \delta \right)   \right\rvert
         < \left\lvert  q^{\prime} \left( \tau^* + \delta \right)  \right\rvert $ 
if~$ 0 \le \delta \le \tau^* .$  And so~$    \lvert \left( \xi + \kappa \right)  / {\widetilde{\phi}}  \left(  2\tau^* \right)   \rvert >  \lvert \left( \xi + \kappa \right)  / {\widetilde{\phi}}  \left(  0 \right)  \rvert = \xi + \kappa .  $
Hence~$  {\widetilde{\phi}}  \left(  2\tau^* \right)  < 1. $ And since the magnitude
of~${\widetilde{\phi}}  \left(  \cdot \right)  $ is a decreasing function on~$ \left[  \tau^* , + \infty\right)  $ we get:
\begin{align}
   \left\lvert {f_+}^{\prime} \left(  \xi  \right) \right\rvert & =  
   \left\lvert	{\widetilde{\phi}}  \left(  \tau_+ \right) \right\rvert \; <  \;   \left\lvert {\widetilde{\phi}}  \left(  2\tau^* \right)  \right\rvert \; < \; 1 .
\label{eqn:fPlusPrimeLessThanOneRepeatedPoles}
\end{align}
\subsection{The function~${f_+}\left( \xi \right)$ is strictly convex on a semi-infinite interval}
To show that~$ {f_+} \left(  \cdot  \right)$ is strictly convex on the interval~$ \left( -\kappa + \gamma/\beta , + \infty \right)$  it is enough to show that~${f_+}^{\prime \prime} \left(  \cdot  \right) > 0$ on this interval. 
Using Equations~\eqref{eqn:tauPrimeForRepeatedPoles},~\eqref{eqn:fPrimeForRepeatedPoles} we get:
\begin{align}
     {f_+}^{\prime\prime}\left( \xi \right) 
& =  
     {\frac
		       {    
					              - \alpha^2  \left( {\left(  \xi + \kappa \right)}^2 {\vec{-}} \gamma^2 / \alpha^2  \right)  {\tau_+}^2
							       	-
					                2 \gamma   \left(  \xi + \kappa  \right) \tau_+
					 }
					 {  e^{ + \alpha \tau_+ } \; {\Bigl( \xi + \kappa   
					     - \alpha \tau_+  \left(  \xi + \kappa + \gamma / \alpha  \right) 
							\Bigr)}^3
					 }
	   } ,\label{eqn:fDoublePrimeForRepeatedPoles}
\end{align}
If~$ \xi + \kappa > \gamma / \alpha  , $ then~$   {f_+}^{\prime\prime}\left( \xi \right) > 0 .$
\subsection{The derivative~${f_+}^{\prime}\left( \cdot \right)$ is continuous on~$ \left[ 0 , + \infty \right)  $}
Exactly like in the case of distinct real poles~(Section~\ref{section:fPrimeIsContinuous}), in this case too is~${f_+}^{\prime}\left( \cdot \right)$ a continuous function on~$ \left[ 0 , + \infty \right) . $
\subsection{Contraction mapping property}

Sections~\ref{section:subsub:contractionMappingDistinctRelalPoles} and \ref{section:subsub:assumptionA4distinctRealPoles} give arguments showing that the contraction mapping property and assumption A~4 hold in the case where the plant has distinct real poles. Those arguments depend on the convexity of~$f_+\left( \cdot \right)$ on a semi-infinite interval that extends to plus infinity.

Those arguments also work in the case where the plant has a repeated real pole, because we have proved  the convexity of~$f_+\left( \cdot \right)$ on the semi-infinite interval~$\left( -\kappa + \gamma/\alpha , + \infty \right) . $  



\section{Case of a complex conjugate pair of poles\label{section:secondOrder:complexConjugatePoles}}
Consider the case where the transfer function takes the form:
\begin{gather*}
    {\frac{ - \kappa s + \gamma}{ {\left(s + \sigma \right)}^2 + \omega^2}},
\end{gather*}
where the real parameters~$\kappa, \gamma, \eta, \omega$ are all positive.

The observer realization takes the form: 
\begin{gather*}
{\dot{x}} \ = \ A x + B u, \quad
y \ = \ Cx, \ \text{where,} \\
    A = \begin{bmatrix}
        0 & - \sigma^2 - \omega^2 \\
        1 & - 2 \sigma  
        \end{bmatrix},
    \quad
    B = \begin{bmatrix}
        \gamma \\ -\kappa 
        \end{bmatrix}, \quad
    C = \begin{bmatrix}
        \ 0 & 1 \
        \end{bmatrix}.
\end{gather*}
\subsection{The trajectory over the duration of a first exit time}
Consider the starting point~$ \left( \xi_0 , \; 0 \right) $ where $\xi > -\kappa . $ Let~$ \left( p , q \right) $ denote the state of the RFS after a duration of~$t$~seconds. 
\subsubsection{Matrix exponential via the Jordan canonical form}
We have the following similarity transformation: 
\begin{gather*}
 {{T}} A {{T}}^{-1} \ =  \ \begin{bmatrix}
           - \sigma & -\omega\\
					 \omega  &  - \sigma
       \end{bmatrix}  , \ \  \text{where}  \\
{{T}} \ = \  {\frac{1}{\sqrt{\omega}}} 
       \begin{bmatrix}
          1 & - \sigma \\
					0 & \omega
       \end{bmatrix}  ,
\quad
{{T}}^{-1} \ = \  {\frac{1}{\sqrt{\omega}}} 
       \begin{bmatrix}
          \omega &  \sigma \\
					0 & 1
       \end{bmatrix}  .
\end{gather*}
The matrix exponential of~$ {{T}} A {{T}}^{-1}  \; t $ can be computed by expressing it as a sum of two commuting matrices, as below:
\begin{align*}
e^{   \begin{bmatrix}
           - \sigma & -\omega\\
					 \omega  &  - \sigma
       \end{bmatrix}         t
	} 
& =  
e^{  \left\{
   -\sigma t \begin{bmatrix}
           1  & 0 \\
					 0  &  1
       \end{bmatrix} 
+  t
    \begin{bmatrix}
           0  & - \omega  \\
					 \omega  & 0
       \end{bmatrix}
			\right\}
	} ,  \\
	& =  
\begin{bmatrix}
           e^{  -\sigma t }  & 0 \\
					 0  &   e^{  -\sigma t }
       \end{bmatrix} 
\times 
 \begin{bmatrix}
           \cos{\omega t}  &  - \sin{\omega t}  \\
					  \sin{\omega t}  &  \cos{\omega t}
       \end{bmatrix}  ,  \\
& =	  e^{  -\sigma t } \begin{bmatrix}
           \cos{\omega t}  &  - \sin{\omega t}  \\
					  \sin{\omega t}  &  \cos{\omega t}
       \end{bmatrix}  .		
\end{align*}
Then for time~$t$ such that $0 \le t \le \tau_+\left( \xi , 0 \right)  $:
\begin{align}
\begin{pmatrix}  p \\ q \end{pmatrix}
& = 
		   A^{-1}B +  T^{-1}  e^{ T  A  T^{-1} \, t }   T \left( \begin{pmatrix}  \xi \\  0 \end{pmatrix}   -  A^{-1} B \right)
		 , \nonumber \\
& =
	- \begin{pmatrix}
	                {\frac{  2 \sigma }{  \sigma^2 + \omega^2  }} \gamma + \kappa \\
		  {\frac{  1 }{ \sigma^2 + \omega^2 }} \gamma 
	\end{pmatrix}  +  \quad {\frac{e^{-\sigma t}}{\omega}} \times  \nonumber \\
& \ \
	                   \begin{pmatrix}  
									     { \left( 
											       \omega {\mu}_{1}   +  \sigma \nu_0  \right) } \cos{\omega t}
														+  { \left( 
											       \sigma {\mu}_{1}   -  \omega \nu_0  \right) } \sin{\omega t}
										    \\
											 \nu_0  \cos{\omega t} +  {\mu}_{1} \sin{\omega t} 
										\end{pmatrix} 		,
\label{eqn:explicitRHSforComplexConjugatePoles}
\end{align}
where~$ {\mu}_{1} \left(  \xi \right) \triangleq  \xi + \kappa + \gamma \sigma /{ \left( \sigma^2 + \omega^2 \right) } , $
 and~$ \nu_{0}  \triangleq   \gamma \omega /{ \left( \sigma^2 + \omega^2 \right) } . $ These~$ {\mu}_{1} , \nu_{0}  $ should  not
 be confused with the~$ \mu_{\alpha}, \nu_{\beta} $  from Section~\ref{section:secondOrder:twoRealPoles}.

\subsection{Expression for the derivative~${f_+}^{\prime}\left( \xi \right)$}
Differentiating the component scalar equations in~\eqref{eqn:explicitRHSforRepeatedPoles} gives:
\begin{align}
     {\frac{\partial{\tau_+
		}}{\partial\xi}} 
& =  
     {\frac
		       {  - \sin{\omega \tau_+} }
					 {  { \left( 
											       \omega {\mu}_{1}   -  \sigma \nu_0  \right) } \cos{\omega \tau_+}
														-  { \left( 
											       \sigma {\mu}_{1}   +  \omega \nu_0  \right) } \sin{\omega \tau_+}
					 }
	   }  ,
\label{eqn:tauPrimeForComplexConjugatePoles} \\
      {\frac{d{f_+
		}}{d{\xi}}}  
& =  
     {\frac
		       {   { \left( 
											       \omega {\mu}_{1}   -  \sigma \nu_0  \right) } e^{-\sigma \tau_+} }
					 {  { \left( 
											       \omega {\mu}_{1}   -  \sigma \nu_0  \right) } \cos{\omega \tau_+}
														-  { \left( 
											       \sigma {\mu}_{1}   +  \omega \nu_0  \right) } \sin{\omega \tau_+}
					 }
	   }  ,
\label{eqn:fPrimeForComplexConjugatePoles}
\end{align}
where, for convenience we have abbreviated the symbol~$  \tau_+\left( \xi , 0 \right) $ as~$  \tau_+ , $
and the  symbol~$ f_+( \xi , 0 ) $ as~$ f_+ . $
\subsection{The derivative~${f_+}^{\prime}\left( \xi \right)$ has magnitude less than one}
We shall bound the magnitude of~${f_+}^{\prime}\left( \xi \right)$ by studying the properties of the functional form of~${f_+}^{\prime}\left( \cdot \right) .$ For~$t \in \left[  0 , \tau_+  \right]$ let
\begin{align*}
   {\widehat{\phi}} \left( t \right) & \triangleq 
    {\frac
		       {   { \left( 
											       \omega {\mu}_{1}   -  \sigma \nu_0  \right) } e^{-\sigma t} }
					 {  { \left( 
											       \omega {\mu}_{1}   -  \sigma \nu_0  \right) } \cos{\omega t}
														-  { \left( 
											       \sigma {\mu}_{1}   +  \omega \nu_0  \right) } \sin{\omega t}
					 }
	   }  .
\end{align*}
\subsubsection{Relationship between the functions~$ q ( \cdot ) ,  {\widehat{\phi}} \left( \cdot \right) $}
Consider the behaviour of the function~$q ( \cdot ) $  on~$ \left[ 0 , \tau_+ \right] $   together with the behaviour of the
function~${\widehat{\phi}}  \left( \cdot \right)$ on the same interval. We have:
\begin{align*}
q \left( t \right) 
& = - {\frac{  \nu_0 }{  \omega }}  +
                         {\frac{e^{-\sigma t}}{\omega}}
                     \bigl[ 
											     \nu_0  \cos{\omega t} +  {\mu}_{1} \sin{\omega t} 
											\bigr]  . 
\end{align*}
First we jot down a relation to be used in Section~\ref{section:paragraph:growthOfReciprocalOfPhiHat}:
\begin{multline}
{\frac{d}{dt}}{\Bigl\{ { \left( \omega {\mu}_{1}   -  \sigma \nu_0  \right) }  / {\widehat{\phi}}  \left( t \right)  \Bigr\}}
\\
\begin{aligned}
& = {\frac{d}{dt}}  {\Bigl\{  
                         \left( \omega {\mu}_{1}   -  \sigma \nu_0  \right)   e^{\sigma t} \cos{\omega t}  
                    +    \left( \sigma {\mu}_{1}   +  \omega \nu_0  \right)   e^{\sigma t} \sin{\omega t} 
                     \Bigr\}} ,
  \\
& =       -   \left( \sigma^2 + \omega^2  \right)   e^{\sigma t} \,
                     \bigl[ 
											     \nu_0  \cos{\omega t} +  {\mu}_{1} \sin{\omega t} 
											\bigr] ,
	 \\
& =      -   \left( \sigma^2 + \omega^2  \right)   e^{ 2 \sigma t} \bigl( \omega \; q(t) + \nu_0 \bigr) .
\end{aligned}
\label{eqn:timeDerivativeOfReciprocalOfPhi}
\end{multline}
Next we find a relationship which we use to establish that~$1/{\widehat{\phi}}  \left( \cdot \right) $ vanishes with the derivative~$q^{\prime}{\left( t \right)} $:
\begin{multline*}
{ \left( \omega {\mu}_{1}   -  \sigma \nu_0  \right) } e^{-\sigma t} / {\widehat{\phi}}  \left( t \right) 
\\
\begin{aligned}
& =  
   { \left(    \omega {\mu}_{1}   -  \sigma \nu_0  \right) } \cos{\omega t} 
-  { \left(    \sigma {\mu}_{1}   +  \omega \nu_0  \right) } \sin{\omega t} , \\
& = 
							       \omega  {e^{{\vec{+}}\sigma t}}  \times  {\frac{d}{dt}} 
											\bigl\{ q(t) \bigr\} .
\end{aligned}
\end{multline*}
This implies that~$ q^{\prime} \left( t \right) = 0 $ if an only if $ 1 / {\widehat{\phi}} \left( t \right) = 0  .$ Since~$ q^{\prime} \left( t \right) $ is a product of two factors, one exponential in~$t$ and another trigonometric in~$t$, it follows that the roots come from the trigonometric factor.
Denote by~$\tau^*$ the smallest positive common root of the equations:~$q^{\prime} \left( t \right) = 0  , \; 1 /{\widehat{\phi}}\left( t \right) = 0 .$
\subsubsection{$ q^{\prime}( \cdot ) $ has exactly one root in~$\left[ 0 , \pi / \omega\right] $}
The derivative~$q^{\prime} \left( t \right) $ vanishes if and only if
\begin{align*} 
	 { \left(    \omega {\mu}_{1}   -  \sigma \nu_0  \right) } \cos{\omega t}
	& = 
  { \left(    \sigma {\mu}_{1}   +  \omega \nu_0  \right) } \sin{\omega t} .
\end{align*}
Now we show that the last equation has positive coefficients for both~$ \cos{\omega t} $ and~$ \sin{\omega t} . $
It is clear that~$\nu_0 > 0 .$ For every~$\xi > - \kappa ,$ we have: $\mu_1 > 0 , $ and~$ \omega {\mu}_{1}   -  \sigma \nu_0  = \omega \left(  \xi + \kappa \right) > 0 . $ Hence it follows that
\begin{align*}
		\omega {\mu}_{1}   -  \sigma \nu_0 
	, \,
		\sigma {\mu}_{1}   +  \omega \nu_0 
& > 0 , \ \forall \xi > - \kappa .
\end{align*}
On the interval~$ \left[ 0 , \pi / \left( 2 \omega \right) \right] $ the function~$ \cos{\omega t} $ falls monotonically from~1~to~0,
while  the function~$ \sin{\omega t} $ rises monotonically from~0~to~1. Hence there is exactly one time instant~$t$ from~$ \left[ 0 , \pi / \left( 2 \omega \right) \right] $ such that
\begin{align*}
	 { \left(    \omega {\mu}_{1}   -  \sigma \nu_0  \right) } \cos{\omega t} 
	& = 
  { \left(    \sigma {\mu}_{1}   +  \omega \nu_0  \right) } \sin{\omega t} .
\end{align*}
On the interval~$ \left[ \pi / \left( 2 \omega \right)  , \pi /  \omega \right] $ the functions~$  \cos{\omega t}  , \sin{\omega t} $ have opposite signs. Hence there is no time instant~$t$ from~$ \left[ \pi / \left( 2 \omega \right)  , \pi /  \omega \right] $ where~$  q^{\prime}( t )  $ vanishes.

Thus there is exactly one time instant~(namely $\tau^*$) from the interval~$ \left[ 0 , \pi /  \omega  \right] $ where~$  q^{\prime}( t )  $ vanishes.
\subsubsection{The first exit time~$\tau_+$ is greater than~$\tau^*$}
We shall now study the rise and fall of~$q (t)$ over~$ \left[ 0 , \tau_+ \right] .$ Because of these four facts:
(i)~{$q \left( \cdot \right)$ is continuously differentiable on~$ \left[ 0 , + \infty \right] ,$ }
(ii)~{$q \left( 0 \right) = 0 ,$ and , $q \left( \pi /  \left( 2 \omega \right)   \right) = {\vec{-}} {\left( 1 + e^{-\sigma\pi/\omega}\right)} \nu_0/{ \omega } < 0 $,}
(iii)~{$q^{\prime} \left( 0 \right) =  \xi + \kappa > 0 $, and,}
(iv)~{the only critical point of~$q \left( \cdot \right)$ is at the time~$\tau^*,$}
we can make these two inferences:
(a)~the interval~$ \left[ 0 , \tau^* \right] $ is an interval of ascent where the function~$q(\cdot )$ rises from~$q(0) = 0$ to its peak~$  q\left( \tau^* \right) ,$
(b)~the interval~$ \left[ \tau^* , \pi /  \omega  \right] $ is an interval of descent where~$q(\cdot )$ falls from~$  q\left( \tau^* \right) $ to~${\vec{-}} {\left( 1 + e^{-\sigma\pi/\omega}\right)} \nu_0/{ \omega }$. 

Therefore there is exactly one positive time instant~(namely~$\tau_+$) when $q(\cdot )$ equals its initial value~$q(0) = 0.$ And the time instant~$\tau_+$ must satisfy: 
\begin{align*}
\tau^* & < \tau_+  \; < \; \pi /  \omega  , \ \text{if} \ \xi > - \kappa .
\end{align*}
\subsubsection{Growth of~$ q ( \cdot ) $ on~$ \left[ 0 , \tau^* \right] $ and on~$ \left[ \tau^* ,  \pi /  \omega  \right] $
\label{section:paragraph:growthOfQcomplexConjugatePoles}}
We shall now compare the growth of~$q ( \cdot )$ on the two finite intervals:~$ 
 \left[ 0 , \tau^* \right] ,  \left[ \tau^* , \tau_+ \right] .$ 
\begin{align*}
   q^{\prime} \left( t \right)
& =
  {\frac{e^{-\sigma t}}{\omega}}
	\bigl[
  { \left(    \omega {\mu}_{1}   -  \sigma \nu_0  \right) } \cos{\omega t} 
-  { \left(    \sigma {\mu}_{1}   +  \omega \nu_0  \right) } \sin{\omega t} 
  \bigr] , \\
& =    \chi_4
			{\frac{e^{-\sigma t}}{\omega}}
			\cos{ \left( \omega t + \arccos{\frac{  \omega {\mu}_{1}   -  \sigma \nu_0  }{\chi_4}}         
						 \right)  
					} ,
\end{align*}
where~$ \chi_4 \triangleq  \sqrt{  { \left(    \omega {\mu}_{1}   -  \sigma \nu_0  \right) }^2 +
              { \left(    \sigma {\mu}_{1}   +  \omega \nu_0  \right) }^2  
					 } . $
The cosine function is an odd function w.r.t. any of its roots. That is to say:
\begin{align*}
  \cos{\left( t^* + t \right)} & = - \cos{\left( t^* - t \right)}, \ \forall t , \ \text{if} \ \cos{t^*} = 0 .
\end{align*}

\begin{figure}
\begin{center}
\begin{tikzpicture}[scale = 1.0]
 \begin{axis}[
            xtick=\empty,ytick=\empty,
            axis lines=middle,
            no markers, 
            xmin = -0.1, xmax=0.8, ymin = -0.28, ymax = 0.28,
            samples=200,  xlabel={$t$},  x post scale = 1.15, 
				    legend style={at={(0.5,1.01)},anchor=south},
						]    
				    \addplot[domain=-0.1:0.8, black!40, smooth, forget plot] {0.25*cos(4*x*180*7/22 + 15) };
            \addplot[domain=0:0.65, line width = 1.5, smooth] {0.25*cos(4*x*180*7/22 + 15) };
						\addlegendentry{$ {\frac{\chi_4}{\omega}}
			              \cos{ \left( \omega t + \arccos{\frac{  \omega {\mu}_{1}   -  \sigma \nu_0  }{\chi_4}} \right) } $};
            \addplot[domain=0:0.65, darkToffee, line width = 1] {(exp(-2*x))*0.25*cos(4*x*180*7/22 + 15) };					
					  \addlegendentry{\textcolor{darkToffee}{$ e^{-\sigma t} {\frac{\chi_4}{\omega}}
			              \cos{ \left( \omega t + \arccos{\frac{  \omega {\mu}_{1}   -  \sigma \nu_0  }{\chi_4}} \right) } $}};
\node[color=blueForRed1,scale=1]             at (axis cs:0.325, 0) {$\bullet$};
\node[color=blueForRed1,scale=1,below left]  at (axis cs:0.325, 0) {${\tau}^* $};
\node[color=red1,scale=1]               at (axis cs:0.2, 0) {$\bullet$};
\node[color=red1,scale=1,above left]    at (axis cs:0.2, 0) {${\tau}^* - \delta$};
\node[color=red1,scale=1]               at (axis cs:0.48, 0) {$\bullet$};
\node[color=red1,scale=1,above ]        at (axis cs:0.48, 0) {${\tau}^* + \delta$};
\node[color=blueForRed1,scale=1]        at (axis cs:0.65, 0) {$\bullet$};
\node[color=blueForRed1,scale=1,below]  at (axis cs:0.65, 0) {$2 {\tau}^*$};
             	\draw[red1, ultra thick] (axis cs:0.2, 0) to (axis cs: 0.2,0.085);
            	\draw[red1, ultra thick] (axis cs:0.48, 0) to (axis cs: 0.48,-0.055);
     \end{axis}
\end{tikzpicture}
\end{center}
\label{fig:caseThreeDenominatorOfLevelCrossing}
\caption{For~$\delta\in\left[ 0 , \tau^* \right]$: $ \omega \left\lvert q^{\prime} \left( {\tau}^* - \delta \right)  \right\rvert > \omega \left\lvert q^{\prime} \left( {\tau}^* + \delta \right) \right\rvert $.} 
\end{figure}
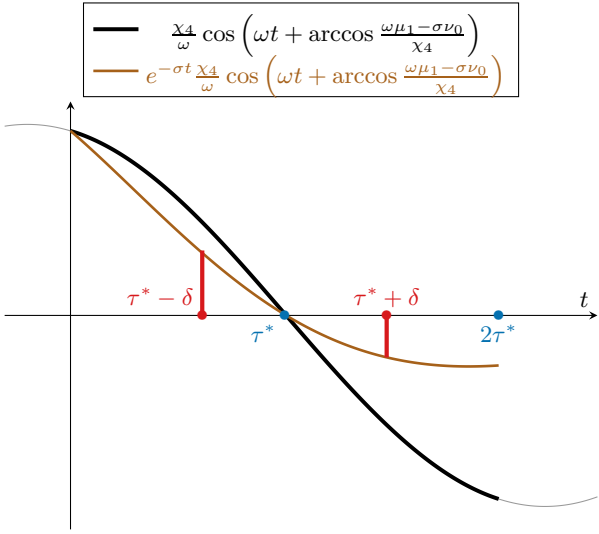
Let the time~$\delta$ be chosen such that~$0 \le \delta \le \tau^* .$ 
Since
\begin{align*}
  \cos{ \left( \omega \tau^* + \arccos{\frac{  \omega {\mu}_{1}   -  \sigma \nu_0  }{\chi_4}}         
						 \right)  
					} 
& = 0, \ \text{we get:}
\end{align*}
\begin{multline}
  \cos{ \left( {\vec{+ \omega\delta}} \; + \; \omega \tau^* + \arccos{\frac{  \omega {\mu}_{1}   -  \sigma \nu_0  }{\chi_4}}         
						 \right)  
					} 
\\ = 
 -  \cos{ \left( {\vec{- \omega\delta}} \; + \; \omega \tau^* + \arccos{\frac{  \omega {\mu}_{1}   -  \sigma \nu_0  }{\chi_4}}         
						 \right)  
					} .
\label{eqn:oddSymmetryOfCosineComplexConjugatePoles}
\end{multline}
Hence we can write:
\begin{multline*}
 \left\lvert q^{\prime} \left( \tau^* - \delta \right) \right\rvert \, e^{\alpha \tau^*} \omega /\chi_4   \\
\begin{aligned}
 & = {\vec{e^{+ \sigma \delta}}}  \cos{ \left( {\vec{- \omega\delta}} \; + \; \omega \tau^* + \arccos{\frac{  \omega {\mu}_{1}   -  \sigma \nu_0  } {\chi_4}}         
						 \right) }  , \\
& >  {\vec{e^{- \sigma \delta}}}   \cos{ \left( {\vec{- \omega\delta}} \; + \; \omega \tau^* + \arccos{\frac{  \omega {\mu}_{1}   -  \sigma \nu_0  }{\chi_4}}         
						 \right) }   , \\
& = e^{{\vec{-}}\sigma \delta}
    \left\lvert \cos{ \left( {\vec{+ \omega\delta}} \; + \; \omega \tau^* + \arccos{\frac{  \omega {\mu}_{1}   -  \sigma \nu_0  }{\chi_4}}         
		\right) } \right\rvert  , \\ 
& =  \left\lvert q^{\prime} \left( \tau^* + \delta \right) \right\rvert \,  e^{\alpha \tau^*} \omega /\chi_4 .
\end{aligned}
\end{multline*}
Clearly, $ \left\lvert  q^{\prime} \left( \tau^* - \delta \right)   \right\rvert
        > \left\lvert  q^{\prime} \left( \tau^* + \delta \right)  \right\rvert $ 
if~$ 0 \le \delta \le \tau^* .$ 
The derivative~$ q^{\prime} \left( \cdot \right) $ is positive over the
 interval~$\left[ 0 , \tau^* \right) ,$ and is negative over the interval~$\left[ \tau^* , \pi /  \omega \right) . $
In other words, the derivative changes sign only at the point~$\tau^* . $ Therefore we can infer that the magnitude of rise in the value of the function over the interval~$\left[ 0 , \tau^* \right) $ is greater than the magnitude of fall over the interval~$\left[ \tau^* , 2\tau^*\right) . $ Indeed, just as in Section~\ref{section:paragraph:growthOfQDistinctPoles},
\begin{align*}
     q \left( \tau^* \right) - q (0) 
& > 
     q \left( \tau^* \right) - q \left( 2 \tau^* \right) .
\end{align*}
Hence~$ q \left( 2 \tau^* \right) > 0.$ And this has the important consequence:
\begin{align}
  \tau_+ & > 2 \tau^* , \ \text{for} \ \xi > - \kappa . \label{eqn:tauPlusBiggerThanTwiceTauStarComplexConjugatePoles} 
\end{align}
\subsubsection{Fall of the magnitude of~$ {\widehat{\phi}} \left( \cdot \right) $ on~$ \left[ 0 , \tau^* \right] $ and its rise on~$ \left[ \tau^* , 2\tau^* \right] $}
Just as we did in Section~\ref{section:paragraph:growthOfPhiDistinctPoles} we shall now 
compare the growth of~${\widehat{\phi}} ( \cdot )$ on the two finite intervals:~$ 
 \left[ 0 , \tau^* \right] ,  \left[ \tau^* , \tau_+ \right] .$  
Let the time~$\delta$ be chosen such that~$0 \le \delta \le \tau^* . $
Then because of the odd symmetry~\eqref{eqn:oddSymmetryOfCosineComplexConjugatePoles} we can write:
\begin{multline*}
       { \left( \omega {\mu}_{1}   -  \sigma \nu_0  \right) } e^{-\sigma \tau^*} 
    /  { \left\{ \chi_4  \left\lvert   {\widehat{\phi}}\left( t - \delta \right)  \right\rvert  \right\} }
\\
\begin{aligned}
 & = {\vec{e^{- \sigma \delta}}}  \cos{ \left( {\vec{- \omega\delta}} \; + \; \omega \tau^* + \arccos{\frac{  \omega {\mu}_{1}   -  \sigma \nu_0  } {\chi_4}}         
						 \right) }  , \\
& <  {\vec{e^{+ \sigma \delta}}}   \cos{ \left( {\vec{- \omega\delta}} \; + \; \omega \tau^* + \arccos{\frac{  \omega {\mu}_{1}   -  \sigma \nu_0  }{\chi_4}}         
						 \right) }   , \\
& = e^{{\vec{+}}\sigma \delta} 
    \left\rvert   \cos{ \left( {\vec{+ \omega\delta}} \; + \; \omega \tau^* + \arccos{\frac{  \omega {\mu}_{1}   -  \sigma \nu_0  }{\chi_4}}         
						 \right) } \right\lvert   , \\ 
& =       { \left( \omega {\mu}_{1}   -  \sigma \nu_0  \right) } e^{-\sigma \tau^*} 
    /  { \left\{ \chi_4  \left\lvert   {\widehat{\phi}}\left( t  + \delta \right)  \right\rvert  \right\} }
.
\end{aligned}
\end{multline*}
Clearly, 
for every~$\delta$ such that $ 0 \le \delta \le \tau^* ,$ 
\begin{align*}
    \left\lvert  {\widehat{\phi}}\left( \tau^* - \delta \right)   \right\rvert
      &  > \left\lvert  {\widehat{\phi}}\left( \tau^* + \delta \right)  \right\rvert ,
\end{align*}
And by setting~$ \delta = \tau^*$ we get:
\begin{align*}
    \left\lvert  {\widehat{\phi}}\left( 0 \right)   \right\rvert
    &    > \left\lvert  {\widehat{\phi}}\left( 2 \tau^* \right)  \right\rvert . 
\end{align*}
\subsubsection{The magnitude of $ {\widehat{\phi}} ( \cdot ) $ falls on~$ \left( \tau^* , \tau_+ \right)$%
\label{section:paragraph:growthOfReciprocalOfPhiHat}}
On the interval~$ \left(  \tau^* , \tau_+ \right)$, the function~$q\left( \cdot \right)$ is positive; in fact it is greater than~$\nu_0 / \omega . $ We use the positive sign of~$ { \left( \omega {\mu}_{1}   -  \sigma \nu_0  \right) } ,$ the positive sign of~$ q\left( \cdot \right)$ on~$ \left[  \tau^* , \tau_+ \right) $, the negative sign of~$ {\widehat{\phi}}  \left( \cdot \right)  $ on the same interval,
and the relation~\eqref{eqn:timeDerivativeOfReciprocalOfPhi} to show
that~$ 1 / \phi ( \cdot ) $ grows in magnitude on~$ \left(  \tau^* , \tau_+ \right] $:
\begin{multline*}
   {\frac{d}{dt}}\Biggl( {\Bigl\{ { \left( \omega {\mu}_{1}   -  \sigma \nu_0  \right) } / {\widehat{\phi}}  \left( t \right)  \Bigr\}}^2 \Biggr)
\\
\begin{aligned}
& = 
   2 {\Bigl\{ { \left( \omega {\mu}_{1}   -  \sigma \nu_0  \right) } / {\widehat{\phi}}  \left( t \right)  \Bigr\}}
   {\frac{d}{dt}}{\Bigl\{ { \left( \omega {\mu}_{1}   -  \sigma \nu_0  \right) } / {\widehat{\phi}}  \left( t \right)  \Bigr\}}^2 ,
	\\
& =
   -  2 \left( \sigma^2 + \omega^2  \right)   e^{ 2 \sigma t} \bigl( \omega \; q(t) + \nu_0 \bigr)
	    {\Bigl\{ { \left( \omega {\mu}_{1}   -  \sigma \nu_0  \right) } / {\widehat{\phi}}  \left( t \right)  \Bigr\}} , 
	\\
& > 0  , \qquad \text{if} \ t \in \left(  \tau^* , \tau_+ \right) .
\end{aligned}
\end{multline*}
Hence we conclude that~$ \left\lvert {\widehat{\phi}}\left( \tau_+ \right) \right\rvert 
  < \left\lvert {\widehat{\phi}}\left( 2 \tau^*  \right)  \right\rvert  . $
\subsubsection{The derivative~${f_+}^{\prime}\left( \xi \right)$ is less than one if~$\xi > - \kappa $}
From the preceding, it follows that:
\begin{align}
   \left\lvert {f_+}^{\prime} \left(  \xi  \right) \right\rvert & =  
   \left\lvert	{\widehat{\phi}}  \left(  \tau_+ \right) \right\rvert \; <  \;   \left\lvert {\widehat{\phi}}  \left(  2\tau^* \right)  \right\rvert \; < \; 1 .
\label{eqn:fPlusPrimeLessThanOneComplexConjugatePoles}
\end{align}
\subsection{An interval mapped to within itself by~${\mathbf{-}}{f_+}\left(  \cdot \right)$}%
We show that for all large enough~$ \xi , $ we get~$ 
\left\lvert  f_+\left( \xi \right) / \xi  \right \rvert < 1
  . $ For this purpose, we shall exactly compute the limits of~$\tau_+ \left( \xi \right), \; \left\lvert  f_+\left( \xi \right) / \xi  \right \rvert $ as~$\xi \to + \infty . $
\subsubsection{As~$\xi $ is increased towards $  \infty ,$ the switching time~$ \tau_+ $ increases towards a finite limit}
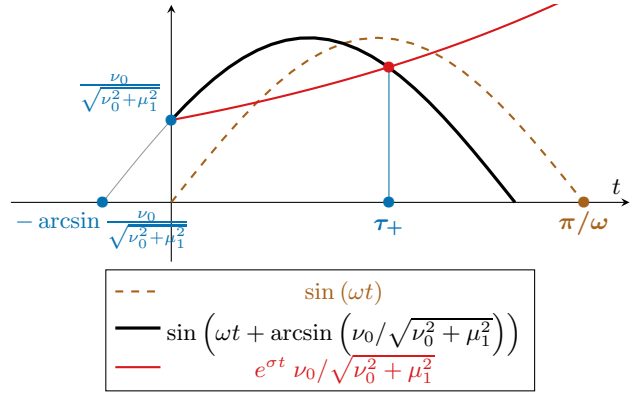
\begin{figure}[b]
\begin{center}
\begin{tikzpicture}
      \begin{axis}[ axis lines = middle, ticks = none, xmin = -70, xmax = 200, ymin = -0.36, ymax = 1.2,
        xlabel={$t$}, 
				x post scale = 1.2,  y post scale = 0.6, 
			 legend style ={at={(0.5,-0.03)},anchor = north},
				]
				    \addplot[domain=0:180, darkToffee, dashed, thick] {sin(x)};
						\addlegendentry[darkToffee]{{$\sin{ \left( \omega t \right) }$}};
            \addplot[domain=-30:0, black!40, forget plot] {sin(x+30)};
					  \addplot[domain=0:150, black, very thick] {sin(x+30)};
						\addlegendentry{$\sin{ \left( \omega t + \arcsin{ \left( {\nu_0} / {\sqrt{\nu_0^2 + \mu_1^2}} \right) }\right) }$};
            \addplot[domain=0:180,red1, thick] {0.5*exp(0.3*x*22/7/180};
						\addlegendentry[red1]{${  e^{\sigma t}  \; {\nu_0} / {\sqrt{\nu_0^2 + \mu_1^2}} }  $};
						\addplot[blueForRed1, thin]coordinates{ (95,0) (95,0.82) };
%
\node[color=blueForRed1,scale=1]  at (axis cs:-30, 0) {$\vec{\bullet}$};
\node[color=blueForRed1,scale=1,below]  at (axis cs:-30, 0) {$ - \arcsin{ \frac{\nu_0}{\sqrt{\nu_0^2 + \mu_1^2}} }$};
\node[color=darkToffee,scale=1]  at (axis cs:180, 0) {$\vec{\bullet}$};
\node[color=darkToffee,scale=1,below]  at (axis cs:180, 0) {$ {\vec{\pi / \omega}}$};
%
%
\node[color=blueForRed1,scale=1]  at (axis cs:0, 0.5) {$\vec{\bullet}$};
\node[color=blueForRed1,scale=1,above left]  at (axis cs:0, 0.5) {$ { \frac{\nu_0}{\sqrt{\nu_0^2 + \mu_1^2}} }$};
\node[color=red1,scale=1]  at (axis cs:95, 0.82) {$\vec{\bullet}$};
\node[color=blueForRed1,scale=1]  at (axis cs:95, 0) {$\vec{\bullet}$};
\node[color=blueForRed1,scale=1,below]  at (axis cs:95, -0.05) {$ {\vec{ \tau_+ }}$};
     \end{axis}
\end{tikzpicture}
\end{center}
\caption{\label{fig:complexConjugatePolesLimitOfTauPlus}If $ \xi \to + \infty ,$ then~$ \nu_0 / \sqrt{\nu_0^2 + \mu_1^2} \to 0 , \tau_+ \to \pi / \omega . $}
\end{figure}
It is clear from Equation~\eqref{eqn:tauPrimeForComplexConjugatePoles} that~$\tau_+\left( \xi , 0 \right)$ is an increasing function of~$\xi  $ if~$ \xi > - \kappa . $

But~$\tau_+\left( \xi , 0 \right)$ is bounded above, as deduced below.
The first exit time~$ \tau_+\left( \xi , 0 \right) $ is the smallest positive time~$t$ solving the
following equation: 
\begin{align*}
			&	    - {\frac{  \nu_0 }{  \omega }}  +
                         {\frac{e^{-\sigma t}}{\omega}}
                     \bigl[ 
											     \nu_0  \cos{\omega t} +  {\mu}_{1} \sin{\omega t} 
											\bigr] 
											\; = \;  0
\\
	\iff & \quad      \quad   \quad   \quad  \;  \; \; \quad   \quad \nu_0  \cos{\omega t} +  {\mu}_{1} \sin{\omega t} 
											\; = \; {  \nu_0 }  {e^{\sigma t}}
\\		
    \iff &      \quad      \sin{ \left( \omega t + \arcsin{ \left( {\frac{\nu_0}{\sqrt{\nu_0^2 + \mu_1^2}}} \right) }\right) }
								\; = \; {\frac{  \nu_0 {e^{\sigma t}} }{  {\sqrt{\nu_0^2 + \mu_1^2}}   }}   ,
\end{align*}
from which we can deduce that for every finite~$\xi > - \kappa , $ 
\begin{align*}
 \tau_+\left( \xi , 0 \right) & < {\frac{\pi}{\omega}} - \arcsin{ \left( {\frac{\nu_0}{\sqrt{\nu_0^2 + \mu_1^2}}} \right) } 
   \; < \; {\frac{\pi}{\omega}} .
\end{align*}
Thus, if~$\xi$ is finite and bigger than~$ - \kappa ,$ then~$ \tau_+ \left( \xi , 0 \right) $ is an increasing function, but is bounded above by~$ \pi / \omega  .$
Hence as~$\xi \to + \infty , $ the function~$  \tau_+\left( \xi , 0 \right) $ converges to a limit, which is less than or equal to~$ \pi / \omega  . $
We calculate this limit precisely. 
\begin{align*}
    \lim_{ \xi \to \infty }   \sin{ \left( \omega \tau_+ + \arcsin{ \left( {\frac{\nu_0}{\sqrt{\nu_0^2 + \mu_1^2}}} \right) }\right) }
& = 
    \lim_{ \xi \to \infty }  {\frac{  \nu_0   e^{\sigma \tau_+} }{  {\sqrt{\nu_0^2 + \mu_1^2}}   }}   ,
\\
\implies
      \sin{ \left( \omega \lim_{ \xi \to \infty }   \tau_+ \right) }
& = 0 ,
\end{align*}
because: (i)~$\tau_+$ is bounded above, (ii)~$\nu_0$ is independent of~$ \xi , $ and (iii)~if~$\xi~\to~\infty  $  then~$ \mu_1~\to~\infty $ . And since
\begin{gather*}
    \lim_{ \xi \to \infty } \tau_+ \left( \xi , 0 \right)  \; \le \;  {\frac{\pi}{\omega}} , \quad
     \sin{ \left( \omega \lim_{ \xi \to \infty }   \tau_+ \left( \xi , 0 \right) \right) } \; = \; 0 ,
\end{gather*}
it follows that
\begin{align*}
    \lim_{ \xi \to \infty }   \tau_+ \left( \xi , 0 \right)  & =  {\frac{\pi}{\omega}} .
\end{align*}
\subsubsection{If~$\xi \to + \infty ,$ then~$ \left\lvert  f_+\left( \xi \right) / \xi  \right \rvert \to e^{- \sigma \pi / \omega} $  }
From Equation~\eqref{eqn:explicitRHSforComplexConjugatePoles} we get:
\begin{align*}
    f_+{\left( \xi \right)}
& =
	 - \kappa -  {\frac{  2 \sigma }{ \left( \sigma^2 + \omega^2 \right) }} \gamma 
  +  {\frac{1}{\omega}}               
				{ \left( \omega {\mu}_{1}   +  \sigma \nu_0  \right) } \cos{\omega \tau_+} e^{-\sigma \tau_+}
\\
& 
            \quad  + {\frac{1}{\omega}}  
			 { \left( \sigma {\mu}_{1}   -  \omega \nu_0  \right) } \sin{\omega \tau_+} e^{-\sigma \tau_+}
 .
\end{align*}
Hence we can do the straightforward derivation below:
\begin{align*}
   \lim_{\xi\to\infty}   \left\lvert {\frac{  f_+\left( \xi \right)}{\xi}}  \right \rvert 
	& =  \Biggl\lvert
	        -   \lim_{\xi\to\infty}{\frac {\kappa +  {{  2  \gamma  \sigma }/{  \sigma^2 + \omega^2  }} }
	                                      {\xi}
	                               }
				\Biggr.
\\
  & \quad \ \
          +    \lim_{\xi\to\infty}{\frac { { \left( \omega {\mu}_{1}   +  \sigma \nu_0  \right) } \cos{\omega \tau_+} e^{-\sigma \tau_+} / {\omega}  }
					                               {\xi}
																  }
\\
&      
   \quad \  \   \Biggl.
	        +    \lim_{\xi\to\infty}{\frac { { \left( \sigma {\mu}_{1}   -  \omega \nu_0  \right) } \sin{\omega \tau_+} e^{-\sigma \tau_+} / {\omega} }
	                                       {\xi}
																	}
	     \Biggr\rvert ,
\\
& =   \left\lvert   0 - e^{- \sigma \pi / \omega} + 0    \right\rvert \; = \; e^{- \sigma \pi / \omega} \; < \; 1 .
\end{align*}
Hence for all sufficiently large~$\xi$ we get the inequality~$    \left\lvert {f_+\left( \xi \right)} \right\rvert    <  \left\lvert {\xi} \right\rvert  . $  
Since~$ {f_+\left( \xi \right)} $ is continuous and monotone decreasing, it follows that there exists a positive~$\eta$ such that whenever~$ \xi $ is greater than or equal to~$\eta , $  we can be sure that $ \lvert f_+ \left( \xi  \right) \rvert \le \xi $ .
\subsection{The derivative~${f_+}^{\prime}\left( \cdot \right)$ is continuous on~$ \left[ 0 , + \infty \right)  $}
Exactly like in the case of distinct real poles~(Section~\ref{section:fPrimeIsContinuous}), in this case too is~${f_+}^{\prime}\left( \cdot \right)$ a continuous function on~$ \left[ 0 , + \infty \right) . $
\subsection{The function~${f_+}\left( \xi \right)$ is a contraction mapping}
We shall find a tight upper bound for the magnitude of the derivative~${f_+}^{\prime}\left( \cdot \right)$ on finite intervals of the form:~$\left[ 0 , \theta \right]  $  with~$0 < \theta < + \infty . $ Just like in Section~\ref{section:subsub:contractionMappingDistinctRelalPoles}  we use the continuity of~$ {f_+}^{\prime}\left( \xi \right)  . $ 


The magnitude of~$ {f_+}^{\prime}\left( \xi \right) $ is continuous on ~$\left[ 0 , \theta \right] .$ By the Weierstrass theorem on extreme values of continuous functions, it follows that the magnitude of~$ {f_+}^{\prime}\left( \cdot \right)  $ attains its maximum on this closed and bounded interval. Denote this maximum value by~${\overline{\lambda}}_{\theta} . $ 
And so we have:
\begin{align}
 \left\lvert {f_+}^{\prime}\left( \xi \right) \right\rvert & \le {\overline{\lambda}}_{\theta} \; < \; 1 , \	  \forall \xi \in \left[ 0 , \theta \right] .
\label{eqn:lambdaForFprimeComplexConjugatePoles}
\end{align}
We apply this at any two points~$\xi ,{\xi}^{\prime} \in \left[ 0 , \theta \right] , $ to get:
\begin{align}
   \left\lvert f_+\left( \xi \right) - f_+\left( {\xi}^{\prime}  \right) \right\rvert 
& = \left\lvert   \int_{\min{ \left\{ \xi ,{\xi}^{\prime} \right\}}}^{\max{ \left\{ \xi ,{\xi}^{\prime} \right\}}}{
f_+^{\prime}\left( {\breve{\xi}} \right) d{\breve{\xi}}} \right\rvert , \nonumber \\
& \le \int_{\min{ \left\{ \xi ,{\xi}^{\prime} \right\}}}^{\max{ \left\{ \xi ,{\xi}^{\prime} \right\}}}{
 \left\lvert   f_+^{\prime}\left( {\breve{\xi}} \right) \right\rvert d{\breve{\xi}}} , \nonumber \\
& \le \int_{\min{ \left\{ \xi ,{\xi}^{\prime} \right\}}}^{\max{ \left\{ \xi ,{\xi}^{\prime} \right\}}}{
 {\overline{\lambda}}_{\theta} d{\breve{\xi}}} , \nonumber \\
& = {\overline{\lambda}}_{\theta}   \left\lvert \xi - {\xi}^{\prime} \right\rvert  .
\label{eqn:contractionMappingPropertyForComplexConjugatePoles}
\end{align}
Hence the function~$  f_+\left( \xi \right) $ is a contraction mapping on~$ \left[ 0 , \theta \right],$ whenever $\theta \ge \eta.$

\section{Asymptotic behaviour of RFS for other zero locations}
We continue with the RFS for the transfer function:
\begin{gather*}
    {\frac{ - \kappa s + \gamma}{ s^2 + a_1 s + a_2}},
    \; \text{with} \, \gamma, a_1 , a_2 > 0,
\end{gather*}
but in contrast to previous sections, we consider the coefficient~$\kappa$ as a free parameter. 
We ask how the asymptotic behaviour of the RFS changes as we change the sign of~$\kappa . $ And we answer via a study of how the switching point transformation function changes as we change the sign of~$\kappa . $
\subsection{Influence of~$\kappa$ on the Switching point transformation function}
To emphasize dependence on~$\kappa , $ let~$ f_{+,\kappa}\left(\cdot \right) $ denote the Switching point transformation function corresponding to a value of~$-\kappa$ for the coefficient of~$s$ in the numerator of the plant transfer function. 
Clearly the function~$ f_{+,\kappa}\left(\cdot\right) $ is completely determined by the dynamics of the ODE: 
\begin{align*}
{\frac{d}{dt}}\left( x - A^{-1}B  \right) & =
    A \left( x - A^{-1}B   \right)  .
\end{align*}
The only effect of varying~$\kappa$ is to vary the $x_1$-coordinate of the sink~$ A^{-1}B $ (see~\eqref{eqn:sinkCoordinates}). Thus, if we vary~$\kappa , $ then the entire vector field of the above ODE is simply translated along the $x_1$-axis. If~${l}, {\widehat{\kappa}}$ are two values for the parameter~$\kappa, $ then  
\begin{align*}
    f_{+,{l}}\left(\xi\right)  & =
    f_{+,{\widehat{\kappa}} }\left(\xi +  {l} - {\widehat{\kappa}} \right) -  {l} + {\widehat{\kappa}} ,
    \ \ \ \forall \xi \in {\mathbb{R}} , \ \text{and}, \\
    f_{+,{l}}^{\prime}\left(\xi\right)  & =
    f_{+,{\widehat{\kappa}} }^{\prime}\left(\xi +  {l} - {\widehat{\kappa}} \right) ,
    \ \ \ \forall \xi \in {\mathbb{R}} .
\end{align*}
From the known properties of $ f_{+,{\widehat{\kappa}} }\left( \cdot \right) , f_{+,{\widehat{\kappa}} }^{\prime}\left( \cdot \right) $ for positive values of~${\widehat{\kappa}},$ we shall derive properties when the parameter~$\kappa$ is either zero or negative.
\subsection{Plant with no finite zero\label{section:plantWithNoZero}}
Let~$ {\widehat{\kappa}}  $ be positive. Then for the RFS corresponding to~$\kappa = 0,$ the origin is a Zeno equilibrium point, in the sense that the trajectories starting there shall remain there, even though there is perpetual switching of the relay. There are no other chattering points. Furthermore, 
\begin{align*}
    f_{+,0}\left(\xi\right)  & =
    f_{+,{\widehat{\kappa}} }\left(\xi - {\widehat{\kappa}} \right) + {\widehat{\kappa}} ,
    \ \ \ \forall \xi \in {\mathbb{R}} , \ \text{and so,}  \\
     - f_{+,0}\left(\xi\right)  & =
    \begin{cases}
        - \xi & \text{if} \ \xi < 0, \\ 
        \int_0^{\xi}{  
         - f_{+,{\widehat{\kappa}} }^{\prime}\left(\zeta - {\widehat{\kappa}}\right) d{\zeta}} 
        & \text{if} \ \xi \ge 0.
    \end{cases}
\end{align*}
We have proved in earlier sections that~$ - f_{+,{\widehat{\kappa}} }\left(\xi - {\widehat{\kappa}} \right) $ is monotonically increasing on~$\left[ 0 , \infty\right) , $ and that the magnitude of the derivative~$f_{+,{\widehat{\kappa}} }^{\prime}\left(\xi - {\widehat{\kappa}} \right) $ is strictly less than one on~$\left( 0  , \infty\right) . $ Hence we can conclude: (i)~that~$ - f_{+,0}\left(\xi \right) $ is monotonically increasing on~$\left[0 , \infty\right) , $ and (ii)~that 
\begin{align*}
    - f_{+,0}\left(\xi\right)  & < \xi, \ \ \text{for} \ \xi > 0 . 
\end{align*}
Hence zero is the only possible fixed point for~$ - f_{+,0}\left( \cdot \right). $ Thus starting with the second element of~$\Xi\left( \xi_0 \right) , $ the successive iterates of~$ - f_{+,0}\left( \cdot \right) $ decrease monotonically and converge to zero.
Therefore we conclude that the all trajectories converge to the origin, for the RFS with the following class of plant transfer functions
\begin{gather*}
    {\frac{\gamma}{ s^2 + a_1 s + a_2}},
    \; \text{with} \, \gamma, a_1 , a_2 > 0.
\end{gather*}
This conclusion could not have been reached if we had tried to apply the Circle criterion to this class of plants. We do not yet know whether this conclusion could have been reached by applying the Popov criterion.
\subsection{Plant zero is negative real}
Let~$ {\widehat{\kappa}}  $ be positive. Then for the RFS corresponding to~$l = - {\widehat{\kappa}} , $ the set~$ \left\{ \left(  x_1 , x_2 \right) :  x_1 \in \left[ - {\widehat{\kappa}} , {\widehat{\kappa}} \right] , \, \text{and} \, x_2 = 0 \right\} $ is the chattering set, and there is no equilibrium point.
Furthermore, 
\begin{align*}
    f_{+, - {\widehat{\kappa}}  }\left(\xi\right)  & =
    f_{+,{\widehat{\kappa}} }\left(\xi - 2 {\widehat{\kappa}} \right) + 2 {\widehat{\kappa}} ,
    \ \ \ \forall \xi \in {\mathbb{R}} , \ \text{and so,}  \\
     - f_{+,0}\left(\xi\right)  & =
    \begin{cases}
        - \xi & \text{if} \ \xi < {\widehat{\kappa}}  , \\ 
        \int_0^{\xi}{  
         - f_{+,{\widehat{\kappa}} }^{\prime}\left(\zeta - 2 {\widehat{\kappa}}\right) d{\zeta}} 
        & \text{if} \ \xi \ge {\widehat{\kappa}}  .
    \end{cases}
\end{align*}
As in Section~\ref{section:plantWithNoZero} we can conclude: (i)~that~$ - f_{+, - {\widehat{\kappa}} }\left(\xi \right) $ is monotonically increasing on~$\left[ + \kappa , \infty\right) , $ and (ii)~that 
\begin{align*}
    - f_{+, - {\widehat{\kappa}} }\left(\xi\right)  & < \xi - 2 {\widehat{\kappa}}  , \ \ \text{for} \ \xi > + {\widehat{\kappa}}  . 
\end{align*}
Starting with the second element of~$\Xi\left( \xi_0 \right) , $ under each iteration of~$- f_{+, - {\widehat{\kappa}} }\left( \cdot \right) $ the switching point decreases by at least~$ 2{\widehat{\kappa}} , $ unless it is already in the chattering set.  
And once the RFS hits a chattering point, it always ``stays'' there.
Hence if the plant transfer function belongs to the class
\begin{gather*}
    {\frac{\kappa s + \gamma}{ s^2 + a_1 s + a_2}},
    \; \text{with} \, \kappa , \gamma, a_1 , a_2 > 0 ,
\end{gather*}
then all trajectories of the RFS converge asymptotically to the chattering set, which is bounded, and of measure zero. 
\section*{Acknowledgements}
I thank the reviewers and Johan Thunberg for their careful reading, and for pointing out errors in earlier versions. I thank Karl Henrik Johansson for valuable discussions.
\appendices
\section{Filippov solutions of our discontinuous ODE\label{sec:fillipovAppendix}}
Equations~\eqref{eqn:secondOrderRplus},~\eqref{eqn:secondOrderRminus} describe a state evolution that can be viewed as either discontinuous, or continuous but many-valued at the switching line~(see the discussion at the end of~\cite{filippov1960ifac}, which is reprinted in Appendix~B of~\cite{jeffrey2018discontinuous}). 
We shall briefly touch upon the existence and non-uniqueness of solutions in the sense of Filippov~\cite{filippov1960ifac}. Thus shall we justify our previous statements of the type: ``All trajectories converge to~\ldots .''%

The Filippov differential inclusion for~\eqref{eqn:secondOrderRplus},~\eqref{eqn:secondOrderRminus} is:
\begin{align*}
    {\dot{x}} & \in
     \begin{cases}
         \left\{ A x - B \right\} , & {\text{if}} \ C x > 0 , \\
         A x + B \times {\mathbf{\left[ -1 , +1 \right] }} , & {\text{if}} \ C x = 0 , \\
         \left\{ A x + B \right\} , & {\text{if}} \ C x  < 0 ,
     \end{cases}
\end{align*}
where the right hand side is multi-valued only on the switching line. Recall that a Filippov solution is one that satisfies this differential inclusion aslmost everywhere on the time interval.
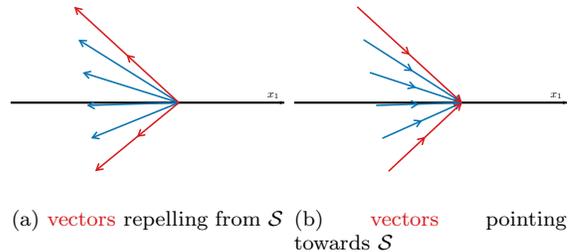
\begin{figure}
\begin{center}
    \subfloat[{\textcolor{red1}{vectors}} repelling from~${\mathcal{S}}$\label{fig:vectorsDiverging}]{%
    \begin{tikzpicture}[scale = 0.48]
\begin{axis}[axis lines = middle, ticks = none,  xmin = 0.2, xmax = 1.5, ymin = -0.7, ymax = 0.8,
    xlabel=$x_1$,  ylabel={},
            x post scale = 1.1, y axis line style={draw=none},
            ]
\draw[color=black,ultra thick]   (axis cs:-2, 0) to (axis cs: 2,0);
        \draw[color=blueForRed1,postaction={decorate}, ->,>=angle 60, very thick]   (axis cs: 1, 0) to (axis cs: 0.52, 0.46) ;
        \draw[color=blueForRed1,postaction={decorate}, ->,>=angle 60, very thick]   (axis cs: 1, 0) to (axis cs: 0.54, 0.22) ;
        \draw[color=blueForRed1,postaction={decorate}, ->,>=angle 60, very thick]   (axis cs: 1, 0) to (axis cs: 0.56, -0.02) ;
        \draw[color=blueForRed1,postaction={decorate}, ->,>=angle 60, very thick]   (axis cs: 1, 0) to (axis cs: 0.58, -0.26) ;
\begin{scope}[very thick,decoration={
    markings,
    mark=at position 0.5 with {\arrow{>}}}
    ]
\draw[color=red1,postaction={decorate}, ->,>=angle 60, very thick]   (axis cs:1, 0) to (axis cs: 0.5,0.7);
\draw[color=red1,postaction={decorate}, ->,>=angle 60, very thick]   (axis cs:1, 0) to (axis cs: 0.6,-0.5);
\end{scope}
\node[color=blueForRed1,scale=1.5]  at (axis cs:0, 0) {$\bullet$};
\node[color=blueForRed1,scale=1,below right]  at (axis cs:0, 0) {$(0,0)$};
\end{axis}%
\end{tikzpicture}%
}
\hspace*{4pt}%
    \subfloat[{\textcolor{red1}{vectors}} pointing towards~${\mathcal{S}}$\label{fig:vectorsConverging}]{%
    \begin{tikzpicture}[scale = 0.48]
\begin{axis}[axis lines = middle, ticks = none,  xmin = 0.2, xmax = 1.5, ymin = -0.7, ymax = 0.8,
    xlabel=$x_1$,  ylabel={},
            x post scale = 1.1, y axis line style={draw=none},
            ]
\draw[color=black,ultra thick]   (axis cs:-2, 0) to (axis cs: 2,0);
\begin{scope}[very thick,decoration={
    markings,
    mark=at position 0.5 with {\arrow{>}}}
    ]
        \draw[color=blueForRed1,postaction={decorate}, ->,>=angle 60, very thick]   (axis cs: 0.53, 0.46) -- (axis cs: 1, 0) ;
        \draw[color=blueForRed1, postaction={decorate}, ->,>=angle 60, very thick]   (axis cs: 0.56, 0.22) -- (axis cs: 1, 0) ;
        \draw[color=blueForRed1, postaction={decorate}, ->,>=angle 60, very thick]   (axis cs: 0.59, -0.02) -- (axis cs: 1, 0) ;
        \draw[color=blueForRed1, postaction={decorate}, ->,>=angle 60, very thick]   (axis cs: 0.62, -0.26) -- (axis cs: 1, 0) ;
    \draw[color=red1,postaction={decorate}, ->,>=angle 60, very thick]   (axis cs: 0.5,0.7) to (axis cs:1, 0) ;
    \draw[color=red1,postaction={decorate}, ->,>=angle 60, very thick]   (axis cs: 0.65,-0.5) to (axis cs:1, 0)  ;
\end{scope}
\end{axis}%
\end{tikzpicture}%
}
\end{center}
    \caption{\label{fig:vectorsAtSwitchingLine}Limiting vectors~(shown in {\textcolor{red1}{red}}) from either side of the switching line, and some other vectors in the Filippov inclusion~(shown in {\textcolor{blueForRed1}{blue}}).}
\end{figure}
\subsubsection*{Plant has a positive real zero}
In this case, on some parts of the switching line each point has a unique solution, and elsewhere each point has infinitely many solutions through it. In the rest of the paper, any interval shall be on the~$x_1$-axis. %

At each point in the closed interval~$\left[ - \kappa , \kappa \right]$ infinitely many solutions exist. At a typical such point, the limiting vector from each side points away from the switching line, as shown in Figure~\ref{fig:vectorsDiverging}. At any such point, the following solutions exist:
\begin{itemize}
    \item{moving in a direction of increasing~$x_2,$ so that the next switching point is in the interval~$\left( - \infty , - \kappa \right),$}
    \item{moving in a direction of decreasing~$x_2,$ so that the next switching point is in the interval~$\left( \kappa , \infty  \right),$}
    \item{sliding along the~$x_1$-axis and reaching the origin,}
    \item{sliding along the~$x_1$-axis towards the origin, but moving off the $x_1$-axis at some point before the origin.}
\end{itemize}
At any point in~$\left( - \infty , - \kappa \right)$ the solution is unique because:
\begin{enumerate}
    \item both the two limiting vectors point towards decreasing~$x_2,$ which is enough to invoke Proposition~5 of~\cite{cortes2008solutionNotionsCSM}.   
This unique solution  through such a switching point is to simply cross over the switching line, and, 
    \item the subsequent trajectory never enters~$ \left[ - \kappa , +  \kappa \right] .$
\end{enumerate}
Similarly, on the interval~$\left( \kappa , \infty \right)$ the unique solutions simply cross over in directions that increase~$x_2.$%

Thus the limit cycle is globally attractive in the~{\textit{weak sense}}~\cite{cortes2008solutionNotionsCSM}, which means that at  every point in the state space, at least one solution exists that converges to the limit cycle.%

Moreover, the sliding mode is a repelling one. A solution that purely slides along the switching set is unstable. If a  slight perturbation takes it off the switching line, then the perturbed trajectory converges to the limit cycle.
\subsubsection*{Plant has no finite zero}
In this case the Filippov solution is unique at every switching point. This conclusions follows for points other than the origin by Proposition~5 of~\cite{cortes2008solutionNotionsCSM}. At the origin too the solution is unique, namely the equilibrium solution. This is unique because the only alternatives available are sliding along the switching line, including sliding with zero speed. But if  a solution exists that slides away from the origin  with non-zero speed, then that would switch in infinitesimally small time and converge in a piecewise differentiable spiral towards the origin. Hence the only viable Filippov solution at the origin is to stay put there. 

    Alternatively we can appeal to Theorem~3 of~\cite{pogromskyHeemelsNijmeijer2003solutionConceptsLinearRelay}, which says when the RFS of Figure~\ref{fig:blockDiagram} has a  plant with a relative degree of {\textit{two,}} there is a unique solution at every initial condition if and only if the plant's second Markov parameter is positive. This holds for this case of the plant because by setting~$\kappa = 0$ in Equation~\eqref{eqn:observerRealization}, we get the first Markov parameter~$CB = 0,$ and the second Markov parameter~$ CAB = \gamma > 0 .$%
\subsubsection*{Plant has a negative real zero}
In this case also is the Filippov solution unique at every switching point. At every point in the interval~$  \left( - \kappa , + \kappa \right) , $ the limiting vectors from either side of the switching line are converging towards the switching line, as is shown in Figure~\ref{fig:vectorsConverging}. And so the sliding mode solution is unique~(by Proposition~5 of~\cite{cortes2008solutionNotionsCSM}). At switching points outside this interval, the transversal crossover solution is unique, by the same proposition.%

    The same conclusion follows from Theorem~2 of~\cite{pogromskyHeemelsNijmeijer2003solutionConceptsLinearRelay}, which says when the RFS of Figure~\ref{fig:blockDiagram} has a  plant with a relative degree of {\textit{one,}} there is a unique solution at every initial condition if and only if the plant's first Markov parameter is positive. This holds for this case of the plant because by setting~$\kappa >  0$ in Equation~\eqref{eqn:observerRealization}, we get the first Markov parameter~$CB = \kappa > 0 .$%
\section{A geometric proof of Schur stability of the linearized switching point transformation, when the plant poles are distinct and negative\label{sec:geometricProof}}

{\textsf{Caution:}} This appendix uses some notation with a meaning different from that in the rest of the paper. For example, the symbols~$\eta, \mu, \theta, \chi , \xi .$ 
Recall that the row vector~$C $ is  the output coefficient matrix in Equation~\eqref{eqn:RFSdynamics}.

For the initial state~$x_0$ let~$W\left(  x_0 \right) \triangleq D_{x_0}\psi\left( x_0 \right),$ which is the Jacobian of the map~$\psi_+.$
Then,
\begin{align}
    W \left( x_0 \right) & =  \left( I - {\frac{1}{ C \, \vec{v} }}  \vec{v} \, C  \right) 
    e^{A\tau_+\left( x_0 \right)}
    \left(  {\frac{ \vec{u} {\vec{u}}^T}{{\vec{u}}^T  \vec{u}  }} \right) ,
    \label{eqn:jacobianOfSwitchingMap}
\end{align}
where~$\vec{u}, \vec{v}$ are column vectors, and are given by:
\begin{align*}
    \vec{u} & =  A \left( x_0 -  A^{-1}B  \right), \\
    \vec{v}  & =  A e^{A\tau_+\left( x_0 \right)}  \left( x_0 -  A^{-1}B  \right) . 
\end{align*}

\begin{lemma}[Schur stability in two dimensions]
Assume that both the eigenvalues of the matrix~$A $ are distinct and negative real, in the following affine, two dimensional, first order constant coefficient ODE:
\begin{align}
    {\frac{d}{dt}}
    \begin{pmatrix}
        x_1 \\ x_2
    \end{pmatrix}
    & =
    A 
    \begin{pmatrix}
        x_1 \\ x_2
    \end{pmatrix}
    .
    \label{eqn:twoDimensionalGeneralODE}
\end{align}
    Assume also that we are given a line~${\mathcal{L}}$ in the $ (x_1,x_2)$-space that does not pass through the origin. Now suppose that we are considering  a trajectory~$\textswab{T}$ that intersects the given line at two distinct points~$P, Q$ in that order. Then the switching point transformation that takes points on the line that are in the neighbourhood of~$P,$ along the flow of the ODE, to points on the line that are in the neighbourhood of~$Q$, has a local linearization that is Schur stable.

    Suppose that there is a second line~${\mathcal{L}}^{\prime}$, that does not pass through the origin, and is parallel to the line~${\mathcal{L}}$.
    If the trajectory~${\textswab{T}}$ emerging from the point~$Q$ intersects the line~${\mathcal{L}}^{\prime}$ at the point~$Q^{\prime}$, then the switching point transformation from the neighbourhood of~$P$ on~${\mathcal{L}}$, along the flow of the ODE, to the neighbourhood of~$Q^{\prime}$ on~${\mathcal{L}}^{\prime}$ has a local linearization that is Schur stable.
\label{lemma:secondOrderSchur}
 \begin{figure}[b!]
\begin{center}
 \begin{tikzpicture}
     \begin{axis}[
          xmin=0.5, xmax=4.5,
          ymin=-0.5, ymax=4.5,
          axis lines=none,
          axis on top=true,
          domain=-4.5:4.5,
          ylabel=$x_2$,
          xlabel=$x_1$,
          ticks = none,
      ]
      \addplot [<-,mark=none,ultra thick,draw=c3,domain=1.0:2.5, samples = 100] {(1 - exp(-0.1*x))*20*pow(sin(40*x),2)};
      \addplot [<-,mark=none,ultra thick,draw=c3,domain=2.5:3.4, samples = 100] {(1 - exp(-0.1*x))*20*pow(sin(40*x),2)};
      \addplot [mark=none,very thick,draw=c2, samples = 5] {2};
      \addplot [mark=none,very thick,draw=c2, samples = 5] {3.5};
   \draw [draw=none,pattern = north west lines , pattern color = c5]
     (1.47-0.30, 2.00-0)
   --(1.47-0.30, 2.00+0.2)
   --(1.47+0.30, 2.00+0.2)
   --(1.47+0.30, 2.00-0)
   --(1.47-0.30, 2.00-0) ;
   \node at (1.47,2.00) {{\Large{\textcolor{c2}{${\mathbf{\bullet}}$}}}};
   \draw [draw=none,pattern = north west lines , pattern color = c5]
     (2.01-0.30, 3.5-0)
   --(2.01-0.30, 3.5+0.2)
   --(2.01+0.30, 3.5+0.2)
   --(2.01+0.30, 3.5-0)
   --(2.01-0.30, 3.5-0) ;
   \node at (2.00,3.5) {{\Large{\textcolor{c2}{${\mathbf{\bullet}}$}}}};
   \draw [draw=none,pattern = north west lines , pattern color = c5]
     (3.17-0.30, 3.5-0)
   --(3.17-0.30, 3.5+0.2)
   --(3.17+0.30, 3.5+0.2)
   --(3.17+0.30, 3.5-0)
   --(3.17-0.30, 3.5-0) ;
   \node at (3.17,3.5) {{\Large{\textcolor{c2}{${\mathbf{\bullet}}$}}}};
\end{axis}
\end{tikzpicture}
\end{center}
\label{figure:twoDimensionalSchurStability}
     \caption{Switching point transformations of Lemma~B.1
}
\end{figure}
\end{lemma}
\begin{proof}
    Denote by~$ \psi_{{\mathcal{L}}\to{\mathcal{L}}}\left(x\right) $ the switching point transformation from points on the line~${\mathcal{L}}$ to itself, under the flow of the ODE~\eqref{eqn:twoDimensionalGeneralODE}. Let the point~$p\in{\mathcal{L}}$ be such that, under the flow of this ODE, it takes a finite and positive time~$\tau$ to start at the point~$p$ and re-enter the line~${\mathcal{L}}$ again.

    To prove Schur stability of the linearization of~$  \psi_{{\mathcal{L}}\to{\mathcal{L}}^{\prime}}\left( \cdot \right) $ at~$p,$ we shall show that~$\psi_{{\mathcal{L}}\to{\mathcal{L}}^{\prime}} $ has a Taylor series expansion at~$p$ with its first order coefficient being less than one in magnitude.

    Our proof has two stages. In the first stage, we show that under any invertible and linear change of variables for the two dimensional state space, we preserve the Schur stability of the linearized switching point transformation. And in the second stage, we arrive at a special choice of state variables in which we establish Schur stability by explicit calculations.

    \noindent
    {\textsf{Stage~(i): {\underline{Schur stability preserved under variable change}}}}
We know that any invertible and linear change of variables preserves the stability of the linearization of an ODE's trajectory. Here we shall show the same to be true of the linearization around a switching point transformation.

    Suppose that the matrix~$T$ is invertible. Then consider:
    \begin{align*}
        \vec{\widetilde{x}} & \triangleq T \vec{x} .
    \end{align*}
    Let the points~$p, q$ be mapped by~$T$ to the points~$ {\widetilde{p}}, {\widetilde{q}} $ respectively.
    Let the lines~$  {\mathcal{L}}, {\mathcal{L}}^{\prime} $ be mapped by~$T$ to the lines~$ {\widetilde{\mathcal{L}}}, {\widetilde{\mathcal{L}}}^{\prime}  $ respectively.
    And let~$  {\widetilde{\psi}}_{{\widetilde{\mathcal{L}}}\to{\widetilde{\mathcal{L}}}^{\prime}} $ be the map on the~$ {\widetilde{x}}$-space that corresponds to the map~$ \psi_{{\mathcal{L}}\to{\mathcal{L}}^{\prime}}$ on the original $x$-space. That is to say:
    \begin{align*}
        {\widetilde{\psi}}_{{\widetilde{\mathcal{L}}}\to{\widetilde{\mathcal{L}}}^{\prime}}\left({\widetilde{x}}\right ) & = T \;
        {\psi}_{{\mathcal{L}}\to{\mathcal{L}}^{\prime}}\left(T^{-1} {\widetilde{x}}\right ) .
    \end{align*}
    Assume that the map~${\widetilde{\psi}}_{{\widetilde{\mathcal{L}}}\to{\widetilde{\mathcal{L}}}^{\prime}} $ has a Schur stable linearization at~${\widetilde{p}}$, which means that there is a real number~${\widetilde{\lambda}}_2$ such that its magnitude is less than one, and  for small~$\epsilon$,
    \begin{align*}
        {\widetilde{\psi}}_{{\widetilde{\mathcal{L}}}\to{\widetilde{\mathcal{L}}}^{\prime}}\left( \vec{{\widetilde{p}}} + \epsilon \; \vec{{\widetilde{l}}} \right)
        & =
        {\widetilde{\psi}}_{{\widetilde{\mathcal{L}}}\to{\widetilde{\mathcal{L}}}^{\prime}}\left( \vec{{\widetilde{p}}}  \right)  +
        \epsilon \lambda \; \vec{{\widetilde{l}}}   +
        \epsilon \; {\widetilde{\rm{rem}}}\left( \vec{{\widetilde{p}}},  \epsilon, {\widetilde{\mathcal{L}}} \right) \; \vec{{\widetilde{l}}}   ,
    \end{align*}
    where~$\vec{{\widetilde{l}}}$ is a unit vector parallel to the line~$ {\widetilde{\mathcal{L}}},$ and the scalar function~$ {\widetilde{\rm{rem}}}\left(\cdot\right)$ is a modified remainder term from the Taylor series, such that
    \begin{align*}
        \lim_{\epsilon\to 0}{ {\widetilde{\rm{rem}}}\left( \vec{{\widetilde{p}}}, \epsilon, {\widetilde{\mathcal{L}}} \right)  } & = 0 .
    \end{align*}
    Then it follows that in the original~$x$-space,
    \begin{align*}
        \shoveleft{ \psi_{{\mathcal{L}}\to{\mathcal{L}}}\left( \vec{p} + \epsilon \; \vec{l} \right) }
  & \\
         \shoveright{
            T^{-1} \times \Biggl(
        {\widetilde{\psi}}_{{\widetilde{\mathcal{L}}}\to{\widetilde{\mathcal{L}}}}\left( T \vec{p}  \right)  +
        \epsilon \lambda \; T \vec{l}   +
        \epsilon \; {\widetilde{\rm{rem}}}\left( T \vec{p},  \epsilon, {\widetilde{\mathcal{L}}} \right) \; T \vec{l}  \Biggr)  ,
    } &
    \end{align*}

    \noindent
    {\textsf{Stage~(ii): {\underline{Schur stability in a special diagonal realization}}}}
    In this stage, we shall first consider two invertible and linear changes in the state variables applied in sequence: (a) a first linear transformation which diagonalizes the coeffecient matrix~$A $ giving a diagonal realization that results in a linear second order ODE that is equivalent to the given ODE~\eqref{eqn:twoDimensionalGeneralODE}, and (b) a second linear transformation which preserves the diagonal property of the ODE's coefficient matrix, but which results in a slope of exactly~$+1$ for the switching line.
    As we carry on with this, we indicate the benefit of pursuing these rather hand-picked changes of state variables.

    The rest of the proof has three substages, where the first and the third substages focus on simplifications due to the two changes in state variables, and the second substage deals with a geometric calculation.

\noindent
    {\textsf{Substage~(ii-a): {\underline{Diagonalizing the $A$ matrix:}}}}
    We make a diagonalizing change of variables, so that we can write down a simple expression for the coordinates of the second switching point and the direction of the ODE's flow there, in terms of the coordinates of the first switching point.

Since~$A $ has distint eigenvalues, it is diagonalizable. In specific, if~$E $ is defined as the square matrix whose two columns are the right eigenvectors of~$A ,$ then~$E $ is invertible and~$E^{-1} A  E  = {\text{diag}}\left\{ -\alpha, -\beta \right\},$ where $-\alpha, -\beta$ are the two eigenvalues of~$A .$ Hence, with the change of variables: ${z} \triangleq E^{-1} x $ we get the equivalent ODE:
\begin{align}
    {\frac{d}{dt}}
    \begin{pmatrix}
        {z}_1 \\ {z}_2
    \end{pmatrix}
    & =
    \begin{bmatrix}
        -\alpha &  0       \\
        0       & -\beta
    \end{bmatrix}
    \begin{pmatrix}
        {z}_1 \\ {z}_2
    \end{pmatrix}
 .
    \label{eqn:twoDimensionalDiagonalODE}
\end{align}
\begin{figure}
  \begin{center}
      \begin{tikzpicture}
\begin{axis}[
  axis lines = middle,
  xmin=-3,xmax=3,
  ymin=-3,ymax=3,
          ylabel=${z}_2$,
          xlabel=${z}_1$,
  ticks=none,
]
    \addplot[c2,domain=-4:4,samples=5,thick] {1.2*x+1};
  \foreach \n in {1,...,6} {
  \addplot[c3,very thick,domain=0:3,samples=200,variable=\t](%
       {-5*(exp(-3*t))},%
       {(\n*1)* (exp(-1*t)}%
     );
  }
  \foreach \n in {1,...,6} {
  \addplot[c3,very thick,domain=0:3,samples=200,variable=\t](%
       {-5*(exp(-3*t))},%
       {(6+\n*2)* (exp(-1*t)}%
     );
  }
  \foreach \n in {1,...,6} {
  \addplot[c3,very thick,domain=0:3,samples=200,variable=\t](%
       {-5*(exp(-3*t))},%
       {(\n*-1)* (exp(-1*t)}%
     );
  }
  \foreach \n in {1,...,6} {
  \addplot[c3,very thick,domain=0:3,samples=200,variable=\t](%
       {-5*(exp(-3*t))},%
       {(-6-\n*2)* (exp(-1*t)}%
     );
  }
  \foreach \n in {1,...,6} {
  \addplot[c3,very thick,domain=0:3,samples=200,variable=\t](%
       {5*(exp(-3*t))},%
       {(\n*-1)* (exp(-1*t)}%
     );
  }
  \foreach \n in {1,...,6} {
  \addplot[c3,very thick,domain=0:3,samples=200,variable=\t](%
       {5*(exp(-3*t))},%
      {(-6-\n*2)* (exp(-1*t)}%
     );
  }
  \foreach \n in {1,...,6} {
  \addplot[c3,very thick,domain=0:3,samples=200,variable=\t](%
       {5*(exp(-3*t))},%
       {(\n*1)* (exp(-1*t)}%
     );
  }
  \foreach \n in {1,...,6} {
  \addplot[c3,very thick,domain=0:3,samples=200,variable=\t](%
      {5*(exp(-3*t))},%
      {(6+\n*2)* (exp(-1*t)}%
     );
  }
  \draw[thin,->, black] (0,-4)--(0,2.95);
   \node at (0.16,1.22) {{\large{\textcolor{c2}{${\mathbf{\bullet}}$}}}};
   \node at (1.32,2.55) {{\large{\textcolor{c2}{${\mathbf{\bullet}}$}}}};
\end{axis}
\end{tikzpicture}
\caption{Integral curves in the diagonal realization}
\label{fig:integralCurvesDiagonalRealization}
  \end{center}
\end{figure}
        Since~$E $ is an invertible and linear, the line~${\mathcal{L}} $ is mapped to a line~${\mathcal{L}_z},$ and this line does not pass through the origin of the ${z}$-space. The slope and intercepts of the line depend on where in the~${z}$-space the two switching points~${p}, {q}$ lie.

        Assume that the first intersection point~${p}$ lies in the first quadrant. Then so must the second intersection point~${q}.$ This is because, no matter what the time~$t$ is, the coordinate~${z}_1(t)$ must have the same sign as its initial value~${z}_1(0),$
Likewise~${z}_2(t)$ must have the same sign as its initial value~${z}_2(0).$  Denote the positive initial value~${z}_1(0)$ by~$\xi,$     and the positive initial value~${z}_2(0)$ by~$\chi.$

        Assume further that~$\alpha > \beta.$ Then the trajectory between points~${p}, {q}$ is a concave curve, when we view it  as the graph of~${z}_2$ as a function of~${z}_1,$ because
\begin{align*}
    {z}_2 (t) & =  \chi \; {\Bigl(  {z}_1(t) / \xi \Bigr)}^{\beta / \alpha}
.
\end{align*}
For this curve, the switching line~${\mathcal{L}_z}$ is the chord that passes through the points~${p} , {q} .$ By concavity, this chord must lie above every point on the curve whose~${z}_1$-coordinate is not a convex combination of the~${z}_1$-coordinates of~${p}, {q}.$ Hence the switching line must lie above the origin, and therefore its~${z}_2$-intercept must be positive.

The switching line~${\mathcal{L}_z}$  must also have positive slope. This follows from the fact that in moving from the point~$ {p}$ to point~$ {q},$ both the coordinate decrements are positive:
\begin{gather*}
    {z}_1(0) -  {z}_1(\tau) = \xi - \xi e^{-\alpha\tau} ,
    \ \
    {z}_2(0) -  {z}_2(\tau) = \chi - \chi e^{-\beta\tau} .
\end{gather*}
Hence the switching line takes the specific form:
\begin{align*}
    {\mathcal{L}_z} &
    =
     \left\{\left.\begin{pmatrix}
              {z}_1  \\  {z}_2
         \end{pmatrix}
         \right\rvert c_1  {z}_1  - c_2  {z}_2  + c_3 = 0 \right\}, \; \text{where}~c_1, c_2, c_3 > 0.
\end{align*}

We have deduced this form from the diagonal realization~\eqref{eqn:twoDimensionalDiagonalODE}, the assumption that~${p}$ lies in the first quadrant, and the assumption that~$\alpha>\beta.$ Clearly, there is no loss of generality in our assumed order of magnitudes of~$\alpha, \beta.$

Nor is there any loss in generality in assuming that~${p}$ lies in the first quadrant. There are two reasons for this.
Firstly the ODE~\eqref{eqn:twoDimensionalDiagonalODE} generates a vector field with mirror symmetry between adjacent quadrants, and so the geometric picture we have developed in the first quadrant is the same as in every other quadrant. And secondly we can make sure that~${p}$ does lie in the 1st quadrant, merely by specially picking the columns of the~$E $ matrix  through a suitable assignment of signs for the right eigenvectors of~$A .$


\noindent
    {\textsf{Substage~(ii-b): {\underline{Calculating exit and re-entry angles:}}}}
    Let~${{\vec{l}}_z}$ denote an unit vector parallel to the line~${\mathcal{L}_z}.$ Then the linearized switching point transformation from~${p}$  to~${q}$  is  simply the uniform scaling of vectors in the one dimensional subspace generated by~${{\vec{l}}_z}.$ We want to show that this scaling coefficient~(denoted by~$\kappa_{\vec{p}}$) has a magnitude less than one.

    Equation~\eqref{eqn:jacobianOfSwitchingMap} describes  the same linearized transformation, but as an operator on~${\mathbb{R}}^2.$ The RHS of  Equation~\eqref{eqn:jacobianOfSwitchingMap}  is a composition of three linear operations. In light of the geometrical picture we have developed~(Figure~\ref{fig:angles-second-order-diagonal}), we can bound the maximum magnifications possible for the three individual matrix operations.

The first operation is the orthogonal projection along~$\vec{u}$:
\begin{gather*}
   {\frac{1}{ {\vec{u}}^T  \vec{u} }} {\vec{u}}  {\vec{u}}^T  ,
\end{gather*}
where 
${\vec{u}}  $ is the exit vector, which is defined as the ODE's velocity vector at the exit point~${p}.$ That is:
\begin{align*}
    {\vec{u}} & =
    \begin{pmatrix}
        -\alpha \xi \\
        -\beta \chi
    \end{pmatrix} .
\end{align*}
When acting on the unit vector~$ {{{\vec{l}}_z}},$ the above orthogonal projection produces a vector of length
$  \left\lvert \cos{\measuredangle\left( {\vec{u}},  {\vec{c}}  \right)}  \right\rvert ,  $
where $ {\vec{c}}$ 
is normal to the switching line~${\mathcal{L}_z},$ and is given by:
\begin{align*}
    {\vec{c}} & =
    \begin{pmatrix}
        {\phantom{-}} c_1 \\
        - c_2
    \end{pmatrix} .
\end{align*}

The second operator is simply the two dimensional exponential map: ${\rm{diag}}\left\{ e^{-\alpha\tau} , e^{-\beta\tau} \right\},$  as per the ODE~~\eqref{eqn:twoDimensionalDiagonalODE}.

The third operation is the oblique projection along~$\vec{v}$ onto the line~${\mathcal{L}_z}$:
\begin{gather*}
    I_2  - {\frac{1}{ {\vec{c}}^T  \vec{v} }} {\vec{c}}  {\vec{v}}^T  ,
\end{gather*}
where~$I_2$ is the $2\times2$ identity matrix, and~${\vec{v}}  $ is the re-entry vector, which is defined as the ODE's velocity vector at the re-entry point~${q}.$ That is:
\begin{align*}
    {\vec{v}} & =
    \begin{pmatrix}
        -\alpha \xi e^{-\alpha\tau} \\
        -\beta \chi e^{-\beta\tau}
    \end{pmatrix} .
\end{align*}
When acting on any unit vector in~${\mathbb{R}}^2$ the above oblique projection produces a vector whose length can be up to~$
1 / \left\lvert \cos{\measuredangle\left( {\vec{v}},  {\vec{c}}  \right)}  \right\rvert .  $

\begin{figure}
\begin{center}
    \begin{tikzpicture}[scale=1.2]
\begin{axis}[
  axis lines = none,
  xmin=-2,xmax=5,
  ymin=-2,ymax=5,
    ticks = none,
]
    \draw[dashed,draw=black!50] (0.5,4)--(4,4);
    \draw[dashed,draw=black!50] (0,-3)--(0,2);
    \node[draw=none, name = q]  at (0,0)  { };
    \node[draw=none, name = sp]  at (0,-2)  { };
    \node[draw=none, name = sp-left]  at (-0.38,-1.9)  {};
    \node[draw=none, name = sp-right]  at (1.2,-1.2*26/25)  {};
    \node[draw=none, name = p]  at (2.68,4)  {};
    \node[draw=none, name = ep]  at (0.68,4)  {};
    \node[draw=none, name = ep-below]  at (2.68-1.4,4-1.4*0.30/0.65) {};
    \node[draw=none, name = ep-above]  at (2.68-0.8,4+0.8*2.6/2.50) {};
    \draw pic["${\eta}$", draw = c5, <->, angle eccentricity=1.25, angle radius=1cm]
    {angle=sp-left--q--sp};
    \draw pic[" ", draw = c2, <->, angle eccentricity=1.2, angle radius=0.7cm]
    {angle=sp--q--sp-right};
    \draw pic["{\textcolor{c5}{${\mu}$}}", draw = c5, <->, angle eccentricity=1.25, angle radius=1cm]
    {angle=ep--p--ep-below};
    \draw pic["{\textcolor{c2}{$\theta_{c}$}}", draw = c2, <->, angle eccentricity=1.3, angle radius=0.7cm]
    {angle=ep-above--p--ep};
    \node[draw=none,rotate = -50,below right] at (0.5,-0.6) {\textcolor{c2}{${\frac{\pi}{2}}-\theta_{c}$}};
    \addplot[c2,domain=-2:5,samples=20,thick]{4*x/2.68};
  \addplot[c3,domain=0.0:0.8,samples=50,variable=\t, ultra thick](%
       {t},%
       {1.5*(ln(5*t+1))}%
     );
  \addplot[<-,c3,domain=0.80:2.68,samples=50,variable=\t, ultra thick](%
       {t},%
       {1.5*(ln(5*t+1))}%
     );
    \draw[c5, thick,->] (2.68,4) -- (2.68-1.4,4-1.4*0.30/0.65);
    \draw[c5, thick,->] (0,0) -- (-0.38,-1.9); 
    \draw[c2, thick,->] (0,0) -- (0.8,-0.8*2.6/2.50); 
    \draw[c2, thick,->] (2.68,4) -- (2.68-0.8,4+0.8*2.6/2.50); 
   \node at (0,0) {{\Large{\textcolor{c2}{${\mathbf{\bullet}}$}}}};
      \node[right] at (0.2,0) {{{\textcolor{c2}{${q}=\left( \xi e^{-\alpha \tau} , \chi e^{-\beta \tau}\right)$}}}};
   \node at (2.68,4) {{\Large{\textcolor{c2}{${\mathbf{\bullet}}$}}}};
      \node[below right] at (2.70,4) {{{\textcolor{c2}{${p}=\left( \xi , \chi \right)$}}}};
\end{axis}
\end{tikzpicture}
\end{center}
\caption{Trajectory exit and re-entry angles,
 as per the diagonal realization~\eqref{eqn:twoDimensionalDiagonalODE}.
The two dashed lines are perpendicular to each other, and are each parallel to a coordinate axis in the~$\vec{z}$-space.}
\label{fig:angles-second-order-diagonal}
\end{figure}
    Using the sub-multiplicative property of the spectral norm, we can bound the linearization's scaling coeffecient~$\kappa_{\vec{p}}$  by the product of the following three factors:
\begin{itemize}
    \item{magnitude of the cosine of the {\textit{exit angle}}
${\measuredangle\left( {\vec{u}},  {\vec{c}}  \right)} , $
}
    \item{$\max\left\{  e^{-\alpha\tau} , e^{-\beta\tau} \right\},$ which is the spectral norm of the two dimensional linear map ${\rm{diag}}\left\{ e^{-\alpha\tau} , e^{-\beta\tau} \right\},$  and,%
}
    \item{
magnitude of the secant of the {\textit{re-entry angle}}
${\measuredangle\left( {\vec{v}},  {\vec{c}}  \right)} . $
}
\end{itemize}
The second of the above factors has a magnitude less than one. So we focus attention on the product of the first and the third factors:
\begin{align}
    {\frac
    {\left\lvert \cos{\measuredangle\left( {\vec{u}},  {\vec{c}}  \right)}  \right\rvert }
    {\left\lvert \cos{\measuredangle\left( {\vec{v}},  {\vec{c}}  \right)}  \right\rvert }
    }
    & =
    {\frac
    {\left\lvert \cos{\measuredangle\left( {\vec{u}},  {\vec{c}}  \right)}  \right\rvert }
    {\left\lvert \cos{ \left( \pi - \measuredangle\left( {\vec{v}},  {\vec{c}}  \right) \right)}  \right\rvert }
    }
    .
\label{eqn:ratioOfCosines}
\end{align}
From Figure~\ref{fig:angles-second-order-diagonal},  we  can see that both the angles~$  \measuredangle\left( {\vec{u}},  {\vec{c}}  \right),$ and $  { \left( \pi - \measuredangle\left( {\vec{v}},  {\vec{c}}  \right) \right)}$ are acute angles, and are given by:
\begin{gather*}
    \measuredangle\left( {\vec{u}},  {\vec{c}}  \right)
     =
     {\mu} + \theta_{c} ,
    \ \
    \ \
{ \pi - \measuredangle\left( {\vec{v}},  {\vec{c}}  \right)  }
     =
     {\eta} 
     \pi / 2 - \theta_{c} ,
\end{gather*}
where~${\mu}$ is the positive acute angle made by the exit velocity vector~$\vec{u}$ with the horizontal~(${z}_1$-axis),
${\eta}$~is the positive acute angle made by the re-entry velocity vector~$\vec{u}$ with the vertical~(${z}_2$-axis), and
${\theta}_{c}$~is the positive acute angle made by the vector~$\vec{c}$ with the horizontal. And
\begin{gather*}
    \mu =  \arctan{\left({\frac{\beta\chi}{\alpha\xi}}\right)} , \
    \eta =  \arctan{\left({\frac{\alpha\xi e^{-\alpha\tau} }{\beta\chi e^{-\beta\tau} }}\right)} , \\
  {\theta}_{c}   =  \arctan{\left({\frac{\beta\chi}{\alpha\xi}}\right)} .
\end{gather*}
We can try to compute the RHS of~\eqref{eqn:ratioOfCosines}, explicitly in terms of trigonometric (and exponential) functions. But this expression while available as an explicit transcendental function, turns out to be unwieldy.

So we shall change state variables, so that we can further simplify the geometric picture of Figure~\ref{fig:angles-second-order-diagonal}.

\noindent
    {\textsf{Substage~(ii-c): {\underline{Switching line with slope~$+1$}:}}}
We had proved in Substage~ii-a that only positive values are possible for the parameters~$c_1, c_2$ which described the switching line~${\mathcal{L}}_z$ in~$({z}_1, {z}_2)$-coordinates. Hence the following linear, anisotropic scaling is invertible: 
    \begin{align*}
        \vec{\widetilde{z}} & \triangleq
         \begin{bmatrix}
             c_1 &   0  \\
             0   & c_2
         \end{bmatrix}
        \vec{z}
    \end{align*}
This second change of state variables does two useful things:
(1)~it gives a slope of~$+1$ to the switching line:
\begin{align*}
    {\widetilde{\mathcal{L}}}_z & \triangleq
     \left\{\left.\begin{pmatrix}
         {\widetilde{z}}_1  \\  {\widetilde{z}}_2
         \end{pmatrix}
         \right\rvert  {\widetilde{z}}_1  - {\widetilde{z}}_2     + c_3 = 0 \right\} .
\end{align*}
and (2)~it preserves the remainder of the geometric picture described in Figure~\ref{fig:angles-second-order-diagonal}.

Thus, the ODE flow in the~${\widetilde{\vec{z}}}$-space is diagonal as before:
\begin{align*}
    {\frac{d}{dt}}
    \begin{pmatrix}
        {\widetilde{z}}_1 \\ {\widetilde{z}}_2
    \end{pmatrix}
    & =
    \begin{bmatrix}
        -\alpha &  0       \\
        0       & -\beta
    \end{bmatrix}
    \begin{pmatrix}
        {\widetilde{z}}_1 \\ {\widetilde{z}}_2
    \end{pmatrix} ,
    \; \text{and let} \\
    \begin{pmatrix}
        {\widetilde{\xi}} \\ {\widetilde{\chi}}
    \end{pmatrix}
     & \triangleq
    \begin{pmatrix}
        {\widetilde{z}}_1 (0) \\ {\widetilde{z}}_2 (0)
    \end{pmatrix}
    =
    \begin{pmatrix}
        c_1 \xi \\ c_2 \chi
    \end{pmatrix}
    .
\end{align*}
Denote by~$\vec{\widetilde{u}}$ the exit vector,
 by~$\vec{\widetilde{v}}$ the re-entry vector, and
 by~$\vec{\widetilde{c}}$ the normal 
 vector:~$ {\left( 1 \; -1 \right)}^T.$

Since both the exit point~$ {( {\widetilde{\xi}} ,  {\widetilde{\chi}} )}^T  $ and the re-entry point~$ {( {\widetilde{\xi}} e^{-\alpha\tau}  ,  {\widetilde{\chi}} e^{-\beta\tau}  )}^T  $ satisfy:~$ {\widetilde{z}}_1  - {\widetilde{z}}_2     + c_3 = 0 ,  $
\begin{align}
    {\frac{ {\widetilde{\xi}} }{ {\widetilde{\chi}} }}
    & =     {\frac{ 1 - e^{-\beta\tau} }{ 1 - e^{-\alpha\tau}  }}  .
\label{eqn:initialAndFinalStatesOnSameLine}
\end{align}

The two  angles~$  \measuredangle\left( {\widetilde{\vec{u}}},  {\widetilde{\vec{c}}}  \right),$ and $  { \left( \pi - \measuredangle\left( {\widetilde{\vec{v}}},  {\widetilde{\vec{c}}}  \right) \right)}$ are acute angles as before, and are given by:
\begin{gather*}
    \measuredangle\left( {\widetilde{\vec{u}}},  {\widetilde{\vec{c}}}  \right)
     =
     {\widetilde{\mu}}
     + {\frac{\pi}{4}} ,
    \ \
    \ \
    { \pi - \measuredangle\left( {\widetilde{\vec{v}}},  {\widetilde{\vec{c}}}  \right)  }
     =
     {\widetilde{\eta}}  
     + {\frac{\pi}{4}} ,
\end{gather*}
where~${\widetilde{\mu}} $ is the positive acute angle made by the exit velocity vector~$\vec{u}$ with the horizontal~(${z}_1$-axis),
${\widetilde{\eta}}  $~is the positive acute angle made by the re-entry velocity vector~$\vec{u}$ with the vertical~(${z}_2$-axis).

The angle:~$ \measuredangle\left( {\vec{u}},  {\vec{c}}  \right)$ is smaller than the angle~${ \pi - \measuredangle\left( {\vec{v}},  {\vec{c}}  \right)  } $,  if the angle~$  {\widetilde{\mu}}  $ is smaller than the angle~$ {\widetilde{\eta}} .$ Consider the ratio:
\begin{align*}
    {\frac
    {
      \tan{ \widetilde{\eta} }
     }
    {
      \tan{ \widetilde{\mu} }
    }
    }
    & =
    \left.
          \left( {\frac{\alpha {\widetilde{\xi}} e^{-\alpha\tau} }{ \beta {\widetilde{\chi}} e^{-\beta\tau} }}  \right)
          \middle/
          \left( {\frac {  \beta {\widetilde{\chi}} } {\alpha {\widetilde{\xi}}  } }  \right)
          \right.
          \\
          & =
          {\frac{ \alpha^2 }{ \beta^2 }}
          {\frac{ {\widetilde{\xi}}^2 }{ {\widetilde{\chi}}^2 }}
          {\frac{ e^{-\alpha\tau} }{ e^{-\beta\tau} }}
          , \ \text{and using~\eqref{eqn:initialAndFinalStatesOnSameLine}} \\
          & =
          {\left(
                 {\frac
                 { \left. \left( e^{+\beta\tau / 2} -  e^{-\beta\tau / 2} \right) \middle/ (2 \beta ) \right. }
                 { \left. \left( e^{+\alpha\tau / 2} -  e^{-\alpha\tau / 2} \right) \middle/ (2 \alpha ) \right. }
                 }
           \right)}^2
           \\
          & =
          {\left(
             {\frac
             { \left. \sinh{\left(\beta\tau / 2 \right)}  \middle/ \beta  \right. }
             { \left. \sinh{\left(\alpha\tau / 2 \right)}  \middle/  \alpha  \right. }
            }
           \right)}^2
           \\
           & < .,
\end{align*}
The last inequality is true because,
for any fixed positive~$t$,
\begin{align*}
 {\frac{1}{\gamma}}    \sinh{\gamma t}
    & = t + {\frac{1}{6}} \gamma^2 t^3 + {\frac{1}{120}} \gamma^4 t^5
+ {\frac{1}{5040}} \gamma^6 t^7 + \ldots
\end{align*}
which is monotonically increasing in~$\gamma.$

    {\textsf{{\underline{Wrapping up the argument}:}}}
 By showing that the re-entry angle into the line~$
    {\widetilde{\mathcal{L}}}_z 
$ is smaller than the exit angle out of it, 
 we have proved the first part of the Lemma.

Figure~\ref{fig:integralCurvesDiagonalRealization}
shows that the entry angle into the parallel line~$
    {\widetilde{\mathcal{L}}}_z^{\prime} $ 
must be smaller than the re-entry angle into the line~$     
    {\widetilde{\mathcal{L}}}_z .$ This proves the remaining part of the Lemma.
\end{proof}

\bibliographystyle{IEEEtran}

\def\polhk#1{\setbox0=\hbox{#1}{\ooalign{\hidewidth
  \lower1.5ex\hbox{`}\hidewidth\crcr\unhbox0}}} \def\cprime{$'$}


\end{document}